\documentclass[final,oneside,notitlepage,11pt]{article}

\usepackage{amsthm}
\usepackage{thmtools}
\usepackage{graphicx}
\usepackage[margin=2cm,font=normalsize,labelfont=bf]{caption}
\usepackage{verbatim}
\usepackage[shortlabels]{enumitem}
\usepackage{fancyhdr}
\usepackage[]{geometry}
\usepackage{xstring}
\usepackage{needspace}
\usepackage{color}
\usepackage{amstext}
\usepackage{nameref}
\usepackage{hyperref}
\usepackage{mathdots}
\usepackage{soul}
\usepackage{chngcntr}
\usepackage[section]{placeins}

\makeatletter
\newcounter{denseversion}
\newcounter{comments}
\newcounter{authorcounter}
\newcounter{adresscounter}

\def\title#1{\gdef\@title{#1}}
\def\@title{}
\def\subtitle#1{\gdef\@subtitle{#1}}
\def\@subtitle{}

\def\authortagsused{0}
\def\adresstag#1{\if!#1!\else$^{\;#1\;}$\fi}

\renewcommand{\author}[2][]{
  \stepcounter{authorcounter}
  \if!#1!\else\gdef\authortagsused{1}\fi
  \ifnum\value{authorcounter}=1
    \def\@authorstringa{#2\adresstag{#1}}
    \def\@authorstringb{#2}
    \def\@authorstringc{#2\adresstag{#1}}
  \else
    \g@addto@macro\@authorstringa{\ and #2\adresstag{#1}}
    \g@addto@macro\@authorstringb{\ and #2}
    \g@addto@macro\@authorstringc{\\#2\adresstag{#1}}
  \fi}
\def\@author{\ifnum\value{denseversion}=0\@authorstringa\else\@authorstringb\fi}

\def\@adressstringa{}
\def\@adressstringb{}
\newcommand{\adress}[2][]{
  \stepcounter{adresscounter}
  \ifnum\value{adresscounter}=1
    \g@addto@macro\@adressstringa{\ifnum\authortagsused=0\def\br{\\}\else\def\br{, }\fi\adresstag{#1}#2}
    \g@addto@macro\@adressstringb{\def\br{\\}\adresstag{#1}\parbox[t]{14cm}{#2}}
  \else
    \g@addto@macro\@adressstringa{\\[\bigskipamount]\adresstag{#1}#2}
    \g@addto@macro\@adressstringb{\\[\medskipamount]\adresstag{#1}\parbox[t]{14cm}{#2}}
  \fi}

\def\preprint#1{\gdef\@preprint{#1}}
\def\@preprint{}
\def\keywords#1{\gdef\@keywords{#1}}
\def\@keywords{}
\def\msc#1{\gdef\@msc{#1}}
\def\@msc{}
\def\email#1{
   \gdef\@email{#1}
   \g@addto@macro\@authorstringc{ {\it (#1)}}}
\def\@email{}
\def\dedication#1{\gdef\@dedication{#1}}
\def\@dedication{}

\def\mybaselinestretch#1{
  \gdef\@mybaselinestretch{#1}
  \renewcommand{\baselinestretch}{\@mybaselinestretch}}

\def\myparskip#1{
  \gdef\@myparskip{#1}
  \setlength{\parskip}{\@myparskip}}

\mybaselinestretch{1.2}
\myparskip{1.5ex}

\binoppenalty=\maxdimen
\relpenalty=\maxdimen

\normalfont
\normalsize
\newlength{\@listleftmargin}
\def\setenumstandard{
  \setlist{leftmargin=\@listleftmargin,itemsep=0pt,topsep=0pt,partopsep=0pt,parsep=\@myparskip}
  \setlist[enumerate]{align=left,labelsep=*,leftmargin=\@listleftmargin,itemsep=0pt,topsep=0pt,partopsep=0pt,parsep=\@myparskip}
}

\def\denseversion{
  \setcounter{denseversion}{1}
  \newgeometry{left=3cm,right=3cm,top=3cm}
  \mybaselinestretch{1.1}
  \myparskip{0.8ex}
  \normalfont
  \def\possiblelinebreak{}
  \fancyfoot[C]{\itshape{--$\,\,$\thepage$\,\,$--}}}

\def\possiblelinebreak{\\}

\renewcommand{\emph}[1]{\def\reserved@a{it}\ifx\f@shape\reserved@a\ul{#1}\else\textit{#1}\fi}

\def\setcrefnames{}
\newcommand{\mytableofcontents}{
   \ifnum\value{denseversion}=0
     \tableofcontents
     \setcrefnames 
   \else
     \renewcommand{\baselinestretch}{1.1}
     \setlength{\parskip}{0ex}
     \normalfont
     \begingroup
     \def\addvspace##1{\vskip0.4em}
     \tableofcontents
     \setcrefnames 
     \endgroup
     \renewcommand{\baselinestretch}{\@mybaselinestretch}
     \setlength{\parskip}{\@myparskip}
     \normalfont
   \fi}

\newlength{\zeilenlaenge}
\def\putindent#1{
  \settowidth{\zeilenlaenge}{#1}
  \ifnum\zeilenlaenge>\textwidth
    #1
  \else
    \noindent #1
  \fi
}

\hypersetup{
    linktocpage = true,
    colorlinks,
    linkcolor=[RGB]{0,0,120},
    citecolor=[RGB]{0,0,120},
    urlcolor=[RGB]{0,0,120}
}

\def\pdfdaten{
  \hypersetup{
    pdftitle = {\@title},
    pdfauthor = {\@author},
    pdfkeywords = {\@keywords},    
    bookmarksopen = true,
    bookmarksopenlevel = 1
  }}  

\def\showkeywords{\begin{flushleft}\footnotesize\textbf{Keywords}: \@keywords\end{flushleft}}
\def\showmsc{\begin{flushleft}\footnotesize\textbf{MSC 2010}: \@msc\end{flushleft}}

\def\tocsection#1{\section*{#1}\addcontentsline{toc}{section}{#1}}

\def\mytitle{}
\def\zmptitle{
  \begin{tabular}{cc}
    \begin{minipage}[c]{0.4\textwidth}
      \begin{flushleft}
        \includegraphics[width=110pt]{../../tex/zmp}
      \end{flushleft}  
    \end{minipage}&
    \begin{minipage}[c]{0.55\textwidth}
      \begin{flushright}
      {\small\sf\@preprint}
      \end{flushright}
    \end{minipage}
  \end{tabular}
  \vskip 2cm}

\def\maketitle{
  \pdfdaten
  \noindent
  \mytitle
  \begin{center}
    \LARGE\@title\\
    \if!\@subtitle!\else\smallskip\LARGE\@subtitle\\\fi
    \bigskip
    \if!\@author!\else\bigskip\large\@author\\\fi
    \ifnum\value{denseversion}=0
      \if!\@adressstringa!\else\bigskip\normalsize\@adressstringa\\\fi
      \if!\@email!\else\ifnum\value{authorcounter}=1\bigskip\normalsize\textit{\@email}\\\else\fi\fi
    \else
    \fi
    \if!\@dedication!\else\bigskip\normalsize{\@dedication}\\\fi
  \end{center}
  \ifnum\value{denseversion}=0\vskip 1.5cm\else\vskip0.5cm\fi}

\def\kobib#1{
  \begin{raggedright}
  \ifnum\value{denseversion}=0\else\small\fi
  \Oldbibliography{#1/kobib}
  \bibliographystyle{#1/kobib}
  \end{raggedright}
  \ifnum\value{denseversion}=0\else
      \noindent
      \if!\@authorstringc!\else
        \ifnum\authortagsused=0\ifnum\value{authorcounter}>1\normalsize\@authorstringc\\[\medskipamount]\else\fi\else\normalsize\@authorstringc\\[\medskipamount]\fi
      \fi
      \if!\@adressstringb!\else\normalsize\@adressstringb\\{}\fi
      \ifnum\authortagsused=0
        \ifnum\value{authorcounter}=1
          \if!\@email!\else\linebreak\normalsize\textit{\@email}\\{}\fi
        \else
        \fi
      \else
      \fi
  \fi}

\let\Oldbibliography\bibliography
\def\bibliography#1{
  \begin{raggedright}
  \ifnum\value{denseversion}=0\else\small\fi
  \Oldbibliography{#1}
  \end{raggedright}
  \ifnum\value{denseversion}=0\else
      \medskip
      \noindent
      \if!\@authorstringc!\else
        \ifnum\authortagsused=0\ifnum\value{authorcounter}>1\normalsize\@authorstringc\\[\medskipamount]\else\fi\else\normalsize\@authorstringc\\[\medskipamount]\fi
      \fi
      \if!\@adressstringb!\else\normalsize\@adressstringb\\{}\fi
      \ifnum\authortagsused=0
        \ifnum\value{authorcounter}=1
          \if!\@email!\else\linebreak\normalsize\textit{\@email}\\{}\fi
        \else
        \fi
      \else
      \fi
  \fi
}

\newenvironment{commentfigure}{}
\newenvironment{sidewayscommentfigure}{\begin{minipage}}{\end{minipage}}
\newenvironment{displaycomment}{\begin{list}{}{\rightmargin=1cm\leftmargin=1cm}\item\sf\begin{small}}{\end{small}\end{list}}

\def\tocmark#1{}

\def\draftstamp#1{
  \def\tocmark##1{
    \ifnum\c@secnumdepth=0\section{##1}\fi
    \ifnum\c@secnumdepth=1\subsection{##1}\fi
    \ifnum\c@secnumdepth=2\subsubsection{##1}\fi
    \ifnum\c@secnumdepth=3\subsubsection{##1}\fi
  }
  \ifnum\value{comments}=0
    \gdef\@draft{DRAFT - Edited on \today\ by #1 - Comments are not displayed}
  \else
    \gdef\@draft{DRAFT - Edited on \today\ by #1 - Comments are displayed}
  \fi
  \fancyhead[C]{\footnotesize\tt\textcolor{red}{\@draft}}}

\def\skript{
  \renewenvironment{displaycomment}{}{}
  \ifnum\value{comments}=0
    \renewenvironment{example*}{\comment}{\endcomment}
    \renewenvironment{remark*}{\comment}{\endcomment}
  \else\fi
  \parindent=0mm	
}

\mybaselinestretch{}
\setenumstandard

\makeatother

\input{kobib}

\usepackage{amssymb}
\usepackage{amsmath}
\usepackage{amstext}
\usepackage{mathrsfs}
\usepackage{euscript}
\usepackage[all,cmtip]{xy}
\usepackage{xstring}
\usepackage{slashed}
\usepackage{tensor}

\entrymodifiers={+!!<0pt,\fontdimen22\textfont2>}

\def\N {\mathbb{N}}
\def\Z {\mathbb{Z}}

\def\R {\mathbb{R}}
\def\C {\mathbb{C}}
\def\T {\mathbb{T}}

\def\id{\mathrm{id}}

\def\hc#1{\mathrm{h}_{#1}}
\def\h {\mathrm{H}}

\def\subset{\subseteq}
\def\supset{\supseteq}

\renewcommand{\varepsilon}{\epsilon}

\newcommand{\incl}{\hookrightarrow}

\makeatletter
\renewenvironment{proof}[1][\nameProof]
  {\par\pushQED{\qed}%
   \normalfont \topsep6\p@\@plus6\p@\relax
   \trivlist
   \item[\hskip\labelsep
         \itshape
         #1\@addpunct{.}]
  \leavevmode}
  {\popQED\endtrivlist\@endpefalse}
\makeatother

\def\notebox#1#2{\begin{minipage}[b]{#1}\sloppy\renewcommand{\baselinestretch}{0.8}\footnotesize \begin{center}#2\end{center}\end{minipage}}

\def\mquad{\hspace{-2em}}

\def\stackref#1#2{\stackrel{\text{\ref{#1}}}{#2}}
\def\eqref#1{\stackref{#1}{=}}

\newlength{\myeqt} 
\newlength{\myeqs} 
\newlength{\myeqm} 
\newlength{\myeqn} 
\newcommand\eqtext[2][\myeqn]{\symtext[#1]{#2}{=}}
\newcommand\symtext[3][\myeqn]{
  \settowidth{\myeqt}{#2}
  \settowidth{\myeqs}{$#3$}
  \addtolength{\myeqs}{\the\myeqm}
  \ifdim\myeqt>\myeqs
    \stackrel{\hspace{-#1}\notebox{#1}{\medskip #2 \\ $\downarrow$\smallskip}\hspace{-#1}}{#3}
  \else
    \stackrel{\text{#2}}{#3}
  \fi}
\newcommand\eqcref[2][\myeqn]{\symcref[#1]{#2}{=}}

\newcommand\symcref[3][\myeqn]{\symtext[#1]{\cref{#2}}{#3}}

\def\brackets#1{\IfStrEq{#1}{-}{}{(#1)}}
\def\subindex#1{\IfStrEq{#1}{-}{}{_{#1}}}

\newcommand{\alxydim}[2]{\begin{aligned}\xymatrix#1{#2}\end{aligned}}



\newlength{\myl}
\newcommand\sheaf[1]{\unitlength 0.1mm
  \settowidth{\myl}{$#1$}
  \addtolength{\myl}{-0.8mm}
  \begin{picture}(0,0)(0,0)
  \put(2,0){\text{\underline{\hspace{\myl}}}}
  \end{picture}#1\hspace{-0.15mm}}

\def\ddt#1#2#3{\left.\frac{\mathrm{d}^{\IfStrEq{#1}{1}{}{#1}}}{\mathrm{d}#2}\IfStrEq{#2}{#3}{\right.}{\right|_{#3}}}

\def\aut{\mathrm{Aut}}

\newcommand{\ueins}{{\mathrm{U}}(1)}

\def\hom{\mathcal{H}\!om}

\def\ev{\mathrm{ev}}

\def\act#1#2{#1/\!\!/#2}
\def\idmorph#1{#1_{dis}}

\def\lw#1#2{{}^{#1\!}#2}
\def\pr{{\mathrm{pr}}}

\newlength{\widthtmp}
\def\length#1{\settowidth{\widthtmp}{#1}\the\widthtmp}

\def\buntech#1#2{\mathcal{B}\hspace{-0.01em}un_{\hspace{0.05em}#1}^{#2}}

\def\bun#1#2{\buntech{#1}{}\brackets{#2}}

\def\grbtech#1{\mathcal{G}\hspace{-0.06em}r\hspace{-0.06em}b_{\hspace{-0.07em}{#1}}}
\def\grb#1#2{\grbtech{#1}\brackets{#2}}

\def\ugrb#1{\grb{\,}{#1}}

\usepackage[latin1]{inputenc}
\usepackage[english]{babel}
\usepackage[english]{cleveref}

\def\refname{References}

\def\quot#1{``#1''}

\def\quand{\quad\text{ and }\quad}
\def\quomma{\quad\text{, }\quad}

\def\nameTheorem{Theorem}
\def\nameTheorems{Theorems}
\def\nameDefinition{Definition}
\def\nameDefinitions{Definitions}
\def\nameCorollary{Corollary}
\def\nameCorollaries{Corollaries}
\def\nameProposition{Proposition}
\def\namePropositions{Propositions}
\def\nameConstruction{Construction}
\def\nameConstructions{Constructions}
\def\nameLemma{Lemma}
\def\nameLemmas{Lemmas}
\def\nameRemark{Remark}
\def\nameRemarks{Remarks}
\def\nameProblem{Problem}
\def\nameProblems{Problems}
\def\nameExample{Example}
\def\nameExamples{Examples}
\def\nameExercise{Exercise}
\def\nameExercises{Exercises}

\def\nameProof{Proof}
\def\nameEquation{Eq.}
\def\nameEquations{Eqs.}
\def\nameSection{Section}
\def\nameSections{Sections}
\def\nameAppendix{Appendix}
\def\nameAppendixs{Appendices}
\def\nameFigure{Figure}
\def\nameFigures{Figures}

\input{theorems}

\title{Higher geometry for non-geometric T-duals}

\author[a]{Thomas Nikolaus}
\email{nikolaus@uni-muenster.de}

\author[b]{Konrad Waldorf}
\email{konrad.waldorf@uni-greifswald.de}

\adress[a]{Universität Münster\\
Mathematisches Institut\\
Fachbereich Mathematik und Informatik\\
Einsteinstrasse 62\\
D-48149 Münster}

\adress[b]{Universität Greifswald\\
Institut für Mathematik und Informatik\\
Walther-Rathenau-Str. 47\\
D-17487 Greifswald}

\newcommand{\poi}[1][]{\mathtt{P}\ifx!#1!\else_{#1}\fi}
\newcommand{\TD}[1][]{\mathbb{TD}\ifx!#1!\else_{#1}\fi}
\newcommand{\AU}[1][]{\mathbb{A}^{\pm}\ifx!#1!\else_{#1}\fi}
\newcommand{\AUpos}[1][]{\mathbb{A}\ifx!#1!\else_{#1}\fi}
\newcommand{\TDpol}[1][]{\mathbb{TD}^{pol}\ifx!#1!\else_{#1}\fi}
\newcommand{\TB}[1][]{\mathbb{TB}\ifx!#1!\else_{#1}\fi}
\newcommand{\TBF}[1][]{\mathbb{TB}\ifx!#1!\else_{#1}\fi^{\text{\textnormal{F1}}}}
\newcommand{\TBfib}[1][]{\mathbb{TB}\ifx!#1!\else_{#1}\fi^{fib}}
\newcommand{\TBFF}[1][]{\mathbb{TB}\ifx!#1!\else_{#1}\fi^{\text{\textnormal{F2'}}}}
\newcommand{\TBFFR}[1][]{\mathbb{TB}\ifx!#1!\else_{#1}\fi^{\text{\textnormal{F2}}}}
\newcommand{\TF}[1][]{\mathbb{TF}^{\pm}\ifx!#1!\else_{#1}\fi}
\newcommand{\TFpos}[1][]{\mathbb{TF}\ifx!#1!\else_{#1}\fi}
\newcommand{\TFpol}[1][]{\mathbb{TF}^{geo}\ifx!#1!\else_{#1}\fi}
\newcommand{\TFgeo}[1][]{\mathbb{TD}^{\frac{1}{2}\text{\textnormal{-geo}}}\ifx!#1!\else_{#1}\fi}
\newcommand{\TFfullgeo}[1][]{\mathbb{TD}^{geo}\ifx!#1!\else_{#1}\fi}

\def\tba{\mathrm{T}\text{-}\mathrm{BG}}
\def\tbaF{\tba^{\mathrm{F1}}}
\def\tbaFF{\tba^{\mathrm{F2}}}
\def\cor{\mathrm{Corr}}
\def\tcor{\mathrm{T}\text{-}\cor}
\def\trip{\mathcal{T}\!rip} 

\def\lele{\leleR'}
\def\rele{\releR'}
\def\leleR{\ell\hspace{-0.105em}e\hspace{-0.08em}\ell\hspace{-0.11em}e}
\def\releR{r\hspace{-0.08em}i\hspace{-0.07em}\ell\hspace{-0.11em}e}
\def\leleRsonZ{\leleR_{\mathfrak{so}(n,\Z)}}

\def\flip{f\!\ell ip}
\def\obj#1{\mathrm{Ob}(#1)}
\def\mor#1{\mathrm{Mor}(#1)}

\def\bra#1#2{\langle#1|#2|}
\def\braket#1#2#3{\bra{#1}{#2}#3\rangle}
\def\lbra#1#2{\langle#1|#2|_{low}}
\def\lbraket#1#2#3{\bra{#1}{#2}#3\rangle_{low}}
\def\ket#1#2{|#1|#2\rangle}
\def\lket#1#2{|#1|#2\rangle_{low}}

\msc{}
\dedication{}

\denseversion
\nolinks

\usepackage{amstext}
\usepackage{amsmath}

\begin{document}

\maketitle 

\begin{abstract}
We investigate topological T-duality in the framework of non-abelian gerbes and higher gauge groups. We show that this framework admits the gluing of locally defined T-duals, in situations where no globally defined ("geometric") T-duals exist. The gluing results into  new, higher-geometrical objects that can be used to study non-geometric T-duals, alternatively to other approaches like non-commutative geometry.
\showkeywords
\end{abstract}

\setcounter{tocdepth}{2}

\mytableofcontents

\setsecnumdepth{1}

\section{Introduction}

\label{sec:intro}

A topological T-background is a principal torus bundle $\pi:E\to X$ equipped with a bundle gerbe $\mathcal{G}$ over $E$, see \cite{Bouwknegt2004a,Bouwknegt2004,Bouwknegt,Bunke2005a,Bunke2006a}. Topological T-backgrounds take care for the underlying topology of a  structure that is considered in the Lagrangian approach to string theory. There, the torus bundle is equipped with a metric and a dilaton field, and the bundle gerbe is equipped with a connection (\quot{B-field}). Motivated by an observation of Buscher about a duality  between these string theoretical structures \cite{Buscher1987}, topological T-backgrounds have been invented in order to investigate the quite difficult underlying topological aspects of this duality. 
In \cite{Bunke2006a}  two T-backgrounds $(E_1,\mathcal{G}_1)$ and $(E_2,\mathcal{G}_2)$ are defined to be T-dual if the pullbacks $p_1^{*}\mathcal{G}_1$  and $p_2^{*}\mathcal{G}_2$ of the two bundle gerbes to the correspondence space
\begin{equation*}
\alxydim{@C=1em}{& E_1 \times_X E_2 \ar[dr]^{p_2} \ar[dl]_{p_1} \\ E_1 \ar[dr]_{\pi_1} && E_2 \ar[dl]^{\pi_2} \\ & X}
\end{equation*}
are isomorphic, and if this isomorphism satisfies a certain local condition relating it with the Poincaré bundle. The relevance of this notion of T-duality is supported  by string-theoretical considerations, see \cite[Remark 6.3]{Fiorenza2018}. One of the basic questions in this setting is to decide, if a given T-background has T-duals, and how the possibly many T-duals can be parameterized. We remark that T-duality is a symmetric relation, but neither reflexive nor transitive. Other formulations of T-duality have been given in the setting of non-commutative geometry \cite{Mathai2005a,Mathai2006a,Mathai2006} and algebraic geometry \cite{Bunke2007a}, and equivalences between these three approaches have been established in \cite{Schneider2007,Bunke2007a}.

Topological T-backgrounds can be grouped into different classes in the following way. 
The Serre spectral sequence associated to the torus bundle comes with a filtration $\pi^{*}\h^3(X,\Z)= F_3\subset F_2 \subset F_1 \subset F_0=\h^3(E,\Z)$, and we classify T-backgrounds by the greatest $n$ such that the Dixmier-Douady class $[\mathcal{G}] \in \h^3(E,\Z)$ is in $F_n$. A result of Bunke-Rumpf-Schick \cite{Bunke2006a} is that a T-Background $(E,\mathcal{G})$ admits T-duals if and only if it is $F_2$. Further, possible choices are related by a certain action of the additive group $\mathfrak{so}(n,\Z)$ of skew-symmetric  matrices $B\in \Z^{n\times n}$, where $n$ is the dimension of the torus. The group $\mathfrak{so}(n,\Z)$ appears there as a subgroup of the group $\mathrm{O}(n,n,\Z)$, which acts on the set of (equivalence classes of) T-duality correspondences in such a way that both \quot{legs} become mixed. The subgroup $\mathfrak{so}(n,\Z)$ preserves the left leg $(E,\mathcal{G})$, and transforms one T-dual right leg into another. The results of \cite{Bunke2006a} have been obtained by constructing classifying spaces for $F_2$ T-backgrounds and T-duality correspondences. The actions of $\mathrm{O}(n,n,\Z)$ and its subgroup $\mathfrak{so}(n,\Z)$ are implemented on the latter space as an action up to homotopy by homotopy equivalences.

If a T-background is only $F_1$, then it does not have any T-duals; these are then called \quot{mysteriously missing} \cite{Mathai2006} or \quot{non-geometric} T-duals \cite{Hull2007}. The approach via non-commutative geometry allows to define them as bundles of non-commutative tori \cite{Mathai2005a,Mathai2006a,Mathai2006}.
In this paper, we propose a new, 2-stack-theoretical ansatz, which allows to remain in ordinary (commutative) but higher-categorical geometry. The basic idea behind our approach is quite simple: every $F_1$ background is locally $F_2$, and so has locally defined T-duals. On overlaps, these are related by the $\mathfrak{so}(n,\Z)$-action of Bunke-Rumpf-Schick. The main result of this article is to fabricate a framework in which  this gluing  can be performed, resulting in a globally defined, new object that we call a \emph{half-geometric T-duality correspondence}.

Half-geometric T-duality correspondences are called \quot{half-geometric}  because they have a well-defined \quot{geometric} left leg, but no \quot{geometric} right leg anymore. It can  be seen as a baby version of what Hull calls a T-fold (in the so-called doubled geometry perspective): upon choosing a \quot{polarization}, a half-geometric T-duality correspondence splits into locally defined T-duality correspondences between the geometric left leg and an only locally defined geometric right leg. Additionally, it provides a consistent set of $\mathfrak{so}(n,\Z)$ matrices that perform the gluing. 
Our main result about half-geometric T-duality correspondences is that
(up to isomorphism) every $F_1$ T-background is the left leg of a unique half-geometric T-duality correspondence.

T-folds should be regarded and studied  as generalized (non-geometric) backgrounds for string theory, and we believe that  understanding the underlying topology is a necessary prerequisite. Since currently no  definition of a T-fold in general topology is available, this has to be seen as a programme for the future. Only two subclasses of T-folds are currently well-defined: the first subclass consists of  T-duality correspondences, and the result of Bunke-Rumpf-Schick states that these correspond to the (geometric) $F_2$  T-backgrounds. The second subclass consists of our new half-geometric T-duality correspondences. It contains the first subclass, and our main result states that it corresponds to  the (geometric) $F_1$  T-backgrounds. We plan to continue this programme in the future and to define and study more general classes of T-folds.

The framework we develop consists of  2-stacks that are represented by appropriate Lie 2-groups. This is fully analogous to the situation that a stack of principal bundles is represented by a Lie group. Using  2-stacks is essential for performing our gluing constructions. In principle, the reader is free to choose any model for such 2-stacks, for instance non-abelian bundle gerbes  \cite{aschieri,Nikolaus},  principal 2-bundles \cite{wockel1,pries2,Waldorf2016}, or any $\infty$-categorical model \cite{Lurie2009,nikolausc}. For our calculations we use a simple and very explicit cocycle model for the non-abelian cohomology groups that classify our 2-stacks. 
Our work is based on two new strict Lie 2-groups, which we construct in \cref{sec:braket,sec:TD2-group}:  
\begin{itemize}

\item 
A strict (Fr\' echet) Lie 2-group $\TBFFR$, see \cref{sec:braket}. It represents a 2-stack $\tbaFF$, whose objects are precisely all $F_2$ T-backgrounds (\cref{prop:classF2}). 

\item
A strict Lie 2-group $\TD$, see \cref{sec:TD2-group}. It is a \emph{categorical torus} in the sense of Ganter \cite{Ganter2014}, and represents a 2-stack $\tcor$ of T-duality correspondences. Its truncation to a set-valued presheaf is precisely the presheaf of \emph{T-duality triples} of \cite{Bunke2006a} (\cref{prop:TCorTriples,prop:BTDCorr}). 

\end{itemize}
These two Lie 2-groups have two important features. The first is that they admit Lie 2-group homomorphisms
\begin{equation}
\label{eq:intro:leleandrele}
\leleR: \TD \to \TBFFR
\quand
\releR: \TD \to \TBFFR
\end{equation}
that represent the projections to the left leg and the right leg of a T-duality correspondence, respectively. The second feature is that they admit  a \emph{strict} $\mathfrak{so}(n,\Z)$-action by Lie 2-group homomorphisms. This  is an important improvement of the action of  \cite{Bunke2006a} on classifying spaces which is only an action up to homotopy, i.e., in the homotopy category.
Taking the semi-direct products for these actions, we obtain new Lie 2-groups:
\begin{itemize}

\item 
A Lie 2-group 
\begin{equation*}
\TBF := \TBFFR \ltimes \mathfrak{so}(n,\Z)
\end{equation*}
of which we prove that it represents the 2-stack $\tbaF$ of $F_1$ T-backgrounds (\cref{prop:classF1}).

\item
A Lie 2-group
\begin{equation*}
\TFgeo := \TD \ltimes \mathfrak{so}(n,\Z)
\end{equation*}
that represents a 2-stack which does not correspond to any classical geometric objects, and  which is by definition  the 2-stack of half-geometric T-duality correspondences.

\end{itemize}
We show that the left leg projection $\leleR$ is $\mathfrak{so}(n,\Z)$-equivariant, and that it hence induces a Lie 2-group homomorphism
\begin{equation*}
\leleRsonZ: \TFgeo \to \TBF\text{.}
\end{equation*}
Our main result about existence and uniqueness of half-geometric T-duality correspondences is that this homomorphism induces an isomorphism in non-abelian cohomology (\cref{th:main}). 

On a technical level, two innovations in the theory of Lie 2-groups are developed in this article. The first is the notion of a  \emph{crossed intertwiner} between crossed modules of Lie groups (\cref{def:CI}). Crossed modules are the same as strict Lie 2-groups, and crossed intertwiners are in between strict homomorphisms and fully general homomorphisms (\quot{butterflies}).  The notion of crossed intertwiners is designed exactly so that the implementations needed in this article can be performed; for instance,   the two Lie 2-group homomorphisms in \cref{eq:intro:leleandrele} are crossed intertwiners. The second innovation is the notion of a \emph{semi-direct product} for a discrete group acting by crossed intertwiners on a strict Lie 2-group (\cref{sec:semidirect}). We also explore various new aspects of non-abelian cohomology in relation to crossed intertwiners and semi-direct products, for example the exact sequence of \cref{prop:fibresequence}.

In the background of our construction lures the higher automorphism group $\mathrm{AUT}(\TD)$ of $\TD$. In an upcoming paper  we compute this automorphism group in collaboration with Nora Ganter. We will show there that $\pi_0(\mathrm{AUT}(\TD[n]))=\mathrm{O}^{\pm}(n,n,\Z)$, a group that  was already mentioned in \cite{Mathai2006}, which contains  the split-orthogonal group $\mathrm{O}(n,n,\Z)$ as a subgroup of index two. We will see  that $\mathrm{AUT}(\TD[n])$ splits over the subgroup $\mathfrak{so}(n,\Z)$. The action of $\mathfrak{so}(n,\Z)$ on $\TD[n]$ we describe here turns out to be the action of $\mathrm{AUT}(\TD[n])$ induced along this splitting, and so embeds our approach into an even more abstract theory. Our plan is to provide new classes of T-folds within this theory.

This article is organized in the following way. In \cref{sec:T} we discuss T-backgrounds and introduce the two Lie 2-groups that represent $F_2$ and $F_1$ T-backgrounds. In \cref{sec:Tdual} we construct the Lie 2-group $\TD$ and put it into relation to ordinary T-duality.  In \cref{sec:TFgeo} we introduce the Lie 2-group $\TFgeo$ that represents half-geometric T-duality correspondences, and prove our main results. In \cref{sec:2groups} we develop our inventions in the theory of Lie 2-groups, and in \cref{sec:poi} we summarize all facts about the Poincaré bundle that we need.

\paragraph{Acknowledgements. } We would like to thank Ulrich Bunke, Nora Ganter, Ruben Minasian, Ingo Runkel, Urs Schreiber, Richard Szabo, Christoph Schweigert and Peter Teichner
for many valuable discussions. KW was supported by the German Research Foundation under project code WA 3300/1-1.

\section{Higher geometry for topological T-backgrounds}

\label{sec:T}

\setsecnumdepth{2}

In this section we study three bicategories of T-backgrounds. The bicategorical framework is a necessary prerequisite for applying 2-stack-theoretical methods, as the gluing property only holds in that setting. As an important improvement of the theory of T-backgrounds, we introduce Lie 2-groups that represent the corresponding 2-stacks (see \cref{def:classifying} for the precise meaning). We will later obtain most statements by only considering the representing 2-groups.

\subsection{T-backgrounds as 2-stacks}

Let $X$ be a smooth manifold. We recall \cite{stevenson1,waldorf1} that $\ueins$-bundle gerbes over $X$ form a bigroupoid, with the set of 1-isomorphism classes of objects in bijection to $\h^3(X,\Z)$.  Let $n>0$ be a fixed integer; all definitions are relative to this integer (often without writing it explicitly). We recall the definition of a (topological) T-background; their bi-groupoidal structure will be essential in this article. We denote by $\T^{n}:= \ueins \times ... \times \ueins$ the $n$-torus.  

\begin{definition}
\begin{enumerate}[(a)]

\item 
A  \emph{T-background} over $X$ is a principal $\mathbb{T}^{n}$-bundle $\pi:E\to X$ together with a $\ueins$-bundle gerbe $\mathcal{G}$ over $E$. 

\item
A \emph{1-morphism}  $(E,\mathcal{G})\to(E',\mathcal{G}')$ is a pair $(f,\mathcal{B})$ of an isomorphism $f: E \to E'$ of principal $\T^{n}$-bundles over $X$ and a bundle gerbe 1-morphism  $\mathcal{B}: \mathcal{G} \to f^{*}\mathcal{G}'$ over $E$. 

\item
A \emph{2-morphism}
\begin{equation*}
\alxydim{@C=4em}{(E,\mathcal{G}) \ar@/^1.2pc/[r]^{(f,\mathcal{B}_1)}="1"\ar@/_1.2pc/[r]_{(f,\mathcal{B}_2)}="2" \ar@{=>}"1";"2"|{} & (E',\mathcal{G}')}
\end{equation*}
is a bundle gerbe 2-morphism $\beta: \mathcal{B}_1 \Rightarrow \mathcal{B}_2$ over $E$. Vertical composition is the vertical composition of bundle gerbe 2-morphisms, and horizontal composition is given by
\begin{equation*}
\alxydim{@C=4em}{(E,\mathcal{G}) \ar@/^1.2pc/[r]^{(f_1,\mathcal{B}_1)}="1"\ar@/_1.2pc/[r]_{(f_1,\mathcal{B}_1')}="2" \ar@{=>}"1";"2"|*+{\beta_1} & (E',\mathcal{G}') \ar@/^1.2pc/[r]^{(f_2,\mathcal{B}_2)}="1"\ar@/_1.2pc/[r]_{(f_2,\mathcal{B}_2')}="2" \ar@{=>}"1";"2"|*+{\beta_2} & (E'',\mathcal{G}'')}
=
\alxydim{@C=3cm}{(E,\mathcal{G})\ar@/^1.5pc/[r]^{(f_2\circ f_1,f_{1}^{*}\mathcal{B}_{2} \circ \mathcal{B}_{1})}="1"\ar@/_1.5pc/[r]_{(f_2\circ f_1,f_{1}^{*}\mathcal{B}_{2}' \circ \mathcal{B}_{1}')}="2" \ar@{=>}"1";"2"|*+{f_1^{*}\beta_2\circ \beta_1} & (E'',\mathcal{G}'')\text{.}}
\end{equation*}

\end{enumerate}
\end{definition}

T-backgrounds over $X$ form a bigroupoid $\tba(X)$, and the assignment $X\mapsto \tba(X)$ is a presheaf of bigroupoids over smooth manifolds.
The following statement is straightforward to deduce from the fact that principal bundles and bundle gerbes form (2-)stacks over the site of smooth manifolds (i.e. they satisfy descent with respect to surjective submersions).

\begin{proposition}
\label{prop:tbasheaf}
The presheaf $\tba$ is a 2-stack.
\qed
\end{proposition}

The grouping of T-backgrounds into classes depending on the class of the bundle gerbe will be done in a different, but equivalent way, compared to \cref{sec:intro}.
A T-background $(E,\mathcal{G})$ is called \emph{$F_1$} if $\mathcal{G}$ is fibre-wise trivializable, i.e. for every point $x\in X$ the restriction $\mathcal{G}|_{E_x}$ of $\mathcal{G}$ to the fibre $E_x$ is trivializable.  For dimensional reasons, every T-background with torus dimension $n\leq 2$ is $F_1$. The condition of being fibre-wise trivial extends for free to neighborhoods of points:

\begin{lemma}
\label{F1F2:b}
A T-background $(E,\mathcal{G})$ is $F_1$ if and only if every point $x\in X$ has an open neighborhood $x\in U \subset X$ such that the restriction of $\mathcal{G}$ to the preimage $E_U := \pi^{-1}(U)$ is trivializable.
\end{lemma}

\begin{proof}
Choose a contractible open neighborhood $x\in U \subset X$ that supports a local trivialization $\phi: U \times \T^{n} \to E_U$. Since $U$ is contractible, we have $\phi^{*}\mathcal{G} \cong \pr_{\T^{n}}^{*}\mathcal{H}$ for a bundle gerbe $\mathcal{H}$ on $\T^n$. Since $(E,\mathcal{G})$ is $F_1$, we know that $\phi^{*}\mathcal{G}|_{\{x\} \times \T^{n}}=(\pr_{\T^{n}}^{*}\mathcal{H})|_{\{x\} \times \T^{n}}$ is trivializable. Pulling back along $\T^{n} \to \{x\} \times \T^{n}$ shows that $\mathcal{H}$ is trivializable. \end{proof}

We recall that the cohomology $\h^k(E,\Z)$ of the total space of a fibre bundle has a filtration 
\begin{equation}
\label{eq:serre}
\h^k(E,\Z) \supset F_1\h^k(E,\Z) \supset... \supset F_k\h^k(E,\Z)=\pi^{*}\h^k(X,\Z) \text{.} 
\end{equation}
Here, $\xi\in \h^k(E,\Z)$ is in $F_{i}\h^k(E,\Z)$ if the following condition is satisfied for every continuous map $f:C \to\ X$ defined on an $(i-1)$-dimensional CW complex $C$: if $\tilde f: f^{*}E \to E$ denotes the induced map, then $\tilde f^{*}\xi=0$. Obviously, we have the following result.

\begin{lemma}
A T-background is $F_1$ if and only if $[\mathcal{G}] \in F_1\h^3(E,\Z)$. 
\qed
\end{lemma}

We let $\tbaF(X)$ be the full sub-bicategory of $\tba(X)$ on the $F_1$ T-backgrounds. Consider the full sub-bicategory $\mathcal{F}_1(X) \subset \tbaF(X)$ on the single object $(X \times \T^{n},\mathcal{I})$, where $\mathcal{I}$ denotes the trivial bundle gerbe. Then, $X \mapsto \mathcal{F}_1(X)$ is a sub-presheaf of $\tbaF$. By \cref{prop:tbasheaf} one can form the closure  $\overline{\mathcal{F}_1} \subset \tba$ of $\mathcal{F}_1$ under descent, i.e., $\overline{\mathcal{F}_1}$ is a sub-2-stack of $\tba$ that 2-stackifies $\mathcal{F}_1$. For example, one can use the 2-stackification  $\mathcal{F} \mapsto \mathcal{F}^{+}$ described in \cite{nikolaus2}, and then choose an equivalence between $\mathcal{F}^{+}_1$ and a sub-2-stack $\overline{\mathcal{F}_1}$ of $\tba$, which exists since $\tba$ is a 2-stack. 
 
\begin{proposition}
\label{prop:tbaFsheaf}
We have $\overline{\mathcal{F}_1}=\tbaF$.  In particular,
$\tbaF$ is a 2-stack.
\end{proposition}

\begin{proof}
It is clear from the definition of $\tbaF$ that $\overline{\mathcal{F}_1} \subset \tbaF$. 
Let $(E,\mathcal{G})$ be some object in $\tbaF(X)$. By \cref{F1F2:b} there exists an open cover $\mathcal{U}=\{U_i\}_{i\in I}$ over which $(E,\mathcal{G})$ is isomorphic to the single object $(X \times \T^{n},\mathcal{I})$. Thus, the descent object $(E,\mathcal{G})|_{\mathcal{U}}  \in \mathrm{Des}_{\tbaF}(\mathcal{U})$ is isomorphic to an object $x\in \mathrm{Des}_{\mathcal{F}_1}(\mathcal{U})$. This shows that $(E,\mathcal{G})$  is in $\overline{\mathcal{F}_1}(X)$. 
\end{proof}

\noindent
Next we consider the following sub-bicategory $\mathcal{F}_2(X) \subset \tba(X)$.
\begin{enumerate}[(a)]

\item 
Its only object is $(X \times \T^{n},\mathcal{I})$. 

\item
The 1-morphisms are of the form $(f,\id_{\mathcal{I}})$.

\item
The 2-morphisms are all 2-morphisms in $\tba(X)$.

\end{enumerate}
The assignment $X \mapsto \mathcal{F}_2(X)$ forms a sub-presheaf $\mathcal{F}_2 \subset  \tba$. Its closure in $\tba$ under descent is by definition the 2-stack $\tbaFF := \overline{\mathcal{F}_2}$.
A T-background over $X$ is called \emph{$F_2$} if it is in $\tbaFF(X)$. By construction, $\mathcal{F}_2 \subset \mathcal{F}_1$, so that we have 2-stack inclusions
\begin{equation*}
\tbaFF \subset \tbaF \subset \tba
\end{equation*}
of which the first is full only on the level of 2-morphisms, and the second is full.
This has to be taken with care: two $F_2$ T-backgrounds can be isomorphic as T-backgrounds without being isomorphic as $F_2$ T-backgrounds.

\begin{lemma}
\label{lem:F2filt}
A T-background is $F_2$ if and only if $[\mathcal{G}] \in F_2\h^3(E,\Z)$. 
\end{lemma}

\begin{proof}
Let $(E,\mathcal{G})$ be an $F_2$ T-background.
Let $C$ be a 1-dimensional CW complex and $f:C \to X$ be continuous. Let $\{U_i\}_{i\in I}$ be a cover of $X$ by open sets over which $(E,\mathcal{G})$ trivializes. For dimensional reasons, the pullback cover $\{f^{-1}(U_i)\}_{i\in I}$ of $C$ can be refined to an open cover  with no non-trivial 3-fold intersections. Descend data of $f^{*}(E,\mathcal{G})$ with respect to this cover has only trivial 2-morphisms. In other words, the bundle gerbe $\tilde f^{*}\mathcal{G}$ is obtained by gluing trivial gerbes $\mathcal{I}$ along identity 1-morphisms $\id_{\mathcal{I}}$ and identity 2-morphisms. This gives the trivial bundle gerbe, hence $[\tilde f^{*}\mathcal{G}]=0$. 

Conversely, since $F^2\h^3 \subset F^1\h^3$, we know that $(E,\mathcal{G})$ is locally trivializable. Thus, there exists an open cover $\{U_i\}_{i\in I}$ and 1-morphisms $(U_i \times \T^{n},\mathcal{I}) \cong (E,\mathcal{G})|_{U_i}$. We form an open cover $V_i := \pi^{-1}(U_i)$ of $E$, such that the bundle gerbe $\mathcal{G}$ is isomorphic to one whose surjective submersion is the disjoint union of the open sets $V_i$. In particular, there are principal $\ueins$-bundles $P_{ij}$ over $V_i \cap V_j$. By construction, we have diffeomorphisms $\phi_{ij}:V_i \cap V_j \to (U_i \cap U_j) \times \T^{n}$. Assuming that all double intersections $U_i \cap U_j$ are contractible, and we can assume that $P_{ij} \cong \phi_{ij}^{*}\pr^{*}\poi_{B_{ij}}$, for a matrix $B_{ij}\in \mathfrak{so}(n,\Z)$ and $\pr: (U_i \cap U_j) \times \T^{n} \to \T^{n}$ the projection, see \cref{sec:poi}. We remark that the two inclusions $V_i \cap V_j \to V_i,V_j$ correspond under the diffeomorphisms $\phi_{ij}$ to the maps $\pr_1(x,a) := (x,a)$ and $\pr_2(x,a)=(x,ag_{ij}(x))$, respectively, where $g_{ij}$ are the transition functions of the $E$.

Next we consider a triangulation of $X$ subordinate to the open cover $\{U_i\}_{i\in I}$, take the dual of that, discard all simplices of dimension $\geq 2$, and let $C$ be the remaining 1-dimensional CW-complex with its inclusion $f:C \to X$. Let $\tilde f: f^{*}E \to E$ the induced map. By assumption, $\tilde f^{*}\mathcal{G}$  is trivializable. Consider the open sets $\tilde V_i := \tilde f^{-1}(V_i)$ that cover $f^{*}E$, so that $\tilde f^{*}\mathcal{G}$ has the $\ueins$-bundles $\tilde P_{ij}:=\tilde f^{*}P_{ij}\cong\tilde f^{*}\phi_{ij}^{*}\pr_{*}\poi_{B_{ij}}$. That $\tilde f^{*}\mathcal{G}$ is trivializable means that there exist principal $\ueins$-bundles $Q_i$ over $\tilde V_i$ with bundle isomorphisms $\tilde P_{ij} \otimes \pr_2^{*}Q_j \cong \pr_1^{*}Q_i$ over $\tilde V_i \cap \tilde V_j$. We have a diffeomorphism $\tilde \phi_i: \tilde V_i \to f^{-1}(U_i) \times \T^{n}$, where $f^{-1}(U_i)$ is star-shaped and hence contractible. Thus, we have $Q_i \cong \tilde\phi_i^{*}\pr^{*}\poi_{C_i}$, for matrices $C_i \in \mathfrak{so}(n,\Z)$. Evaluating above bundle isomorphism at a point $(x,a)$ with $f(x) \in U_i \cap U_j$, we obtain an isomorphism
\begin{equation*}
P_{ij}|_{(x,a)} \otimes Q_j|_{x,ag_{ij}(x)} \cong Q_{i}|_{(x,a)}\text{.}
\end{equation*}
In terms of the principal bundles over $\T^{n}$, this means
\begin{equation*}
\poi_{B_{ij}}|_a \otimes \poi_{C_j}|_{ag_{ij}(x)} \cong \poi_{C_i}|_a\text{.}
\end{equation*}
Using the equivariance of the Poincaré bundle described in \cref{sec:poi}, this leads to $B_{ij}+C_j= C_i$. With this statement about matrices, we go back into the original situation over $X$, where it means that the bundle gerbe $\mathcal{G}$ is isomorphic to one with respect to the open cover $V_i$, all  whose principal $\ueins$-bundles are all trivial. This shows that the given T-background is $F_2$.
\end{proof}

\begin{remark}
For completeness, we remark that a T-background $(E,\mathcal{G})$ is called $F_3$ if  $\mathcal{G}$ admits a $\T^{n}$-equivariant structure. This is equivalent to the statement that $\mathcal{G}$ is the pullback of a bundle gerbe over $X$, which is in turn equivalent to $[\mathcal{G}]\in F_3\h^3(E,\Z)$.
\end{remark}

\setsecnumdepth{2}

\subsection{A  2-group that represents  $F_2$ T-backgrounds}

\label{sec:braket}

We define a strict (Fréchet) Lie 2-group $\TBFF[n]$ that represents the 2-stack $\tbaFF$ in the sense of \cref{def:classifying}, see \cref{sec:2groups} for more information. 

\begin{remark}
\label{re:notation}
We fix the following notations, which will be used throughout this article. \begin{itemize}
\item 
For a matrix $A\in \R^{n\times n}$ and $v,w\in \R^{n}$ we write, as usual, 
\begin{equation*}
\braket vAw :=\sum_{i,j=1}^{n} A_{ij}v_iw_j \in \R\text{.}
\end{equation*}
Suppose $B\in \R^{n\times n}$ is skew-symmetric, i.e. $B\in \mathfrak{so}(n,\R)$. We let $B_{low}\in \R^{n\times n} $ be the unique lower-triangular nilpotent matrix with $B=B_{low}-B_{low}^{tr}$. We write
$\lbraket vBw := \braket v{B_{low}}w$. 
Note that
\begin{equation*}
\braket vBw=\lbraket v{B}w-\lbraket w{B}v\text{.}
\end{equation*}
We use the notation $\bra vA$ for the linear form $w \mapsto \braket vAw$, and similarly $\lbra vB$ for $w\mapsto \lbraket vBw$.

\item
For the abelian groups $\T^{n}$, in particular $\mathbb{T}^1=\ueins$, we use additive notation. If $B\in \mathfrak{so}(n,\Z)$ and $v\in \Z^{n}$, then $\bra vB$ and $\lbra vB$ induce well-defined group homomorphisms $\T^{n} \to \ueins$. 

\item
If $\alpha \in C^{\infty}(\T^{n}, \ueins)$ and $b\in \T^{n}$, then we define $\lw b \alpha \in C^{\infty}(\T^{n},\ueins)$ by $\lw b \alpha(a) := \alpha(a-b)$, which is an action of $\T^{n}$ on $C^{\infty}(\T^{n},\ueins)$. Note that $\lw a \bra vB=\bra vB - \braket vBa$. 
\end{itemize}
\end{remark}

We define the Lie 2-group $\TBFF[n]$ as a crossed module $(G,H,t,\alpha)$ in the following way, see \cref{sec:2groups}. 
We put $G := \T^{n}$, $H := C^{\infty}(\mathbb{T}^{n},\ueins)$ with the point-wise Lie  group structure and the usual Fréchet manifold structure, $t:H \to G$ is defined by $t(\tau):= 0$ and $\alpha:G \times H \to H$ is defined by $\alpha(g,\tau):= \lw g  \tau$. We instantly find 
\begin{equation}
\label{eq:piTBFF}
\pi_0(\TBFF[n]) = \mathbb{T}^{n}
\quand
\pi_1(\TBFF[n]) = C^{\infty}(\T^{n},\ueins)\text{,}
\end{equation}
with $\pi_0$ acting on $\pi_1$ by $(g,\tau)\mapsto \lw g\tau$.  

\begin{remark}
There is a slightly bigger but weakly equivalent 2-group denoted by $\TBFFR[n]$. It fits better to our T-duality 2-group of \cref{sec:Tdual} and will therefore be used later. 
For $\TBFFR[n]$ we put $G := \R^{n}$, $H := C^{\infty}(\mathbb{T}^{n},\ueins) \times \Z^{n}$ with the direct product group structure, $t:H \to G$ is defined by $t(\tau,m) := m$ and $\alpha:G \times H \to H$ is defined by $\alpha(g,(\tau,m)) :=(\lw g  \tau,m)$.
There is a strict intertwiner (see \cref{sec:strictint}) 
\begin{equation*}
\TBFFR[n] \to \TBFF[n]
\end{equation*}
defined by reduction mod $\Z^{n}$,  $\R^{n} \to \T^{n}$, and the group homomorphism
\begin{equation*}
\Z^{n} \times C^{\infty}(\T^{n},\ueins) \to C^{\infty}(\T^{n},\ueins):(\tau,m) \mapsto \tau\text{.}
\end{equation*}
It induces identities on $\pi_0$ and $\pi_1$ and preserves the $\pi_0$ action on $\pi_1$; it is hence a weak equivalence. 
\end{remark}

\label{sec:F2gerbes}

We will often suppress the index $n$ from the notation of both Lie 2-groups.
In order to establish a relation between the 2-group $\TBFF$ and $F_2$ T-backgrounds we consider the presheaf $\sheaf{B\TBFF}$ of smooth $B\TBFF$-valued functions, and show that its 2-stackification $(\sheaf{B\TBFF})^{+}$ is isomorphic to  $\tbaFF$. 
We describe $B\Gamma$-valued functions for a general Lie 2-group $\Gamma$ in \cref{sec:semistrictcocycles}; reducing it to the present situation we obtain over a smooth manifold $X$ the following bicategory $\sheaf{B\TBFF}(X)$:
\begin{itemize}

\item 
It has one object.

\item
The 1-morphisms are smooth maps $g: X \to \T^{n}$, the composition is point-wise addition.

\item
There are only 2-morphisms
from a 1-morphism  to itself; these  are all smooth maps $\tau:X \to C^{\infty}(\T^{n},\ueins)$. The vertical composition is the point-wise addition, and the horizontal composition
is
\begin{equation*}
\alxydim{@C=2cm}{\ast \ar@/^2pc/[r]^{g_1}="1" \ar@/_2pc/[r]_{g_1}="2" \ar@{=>}"1";"2"|*+{\tau_1} & \ast \ar@/^2pc/[r]^{g_2}="1" \ar@/_2pc/[r]_{g_2}="2" \ar@{=>}"1";"2"|*+{\tau_2} & \ast}
=
\alxydim{@C=2.3cm}{\ast \ar@/^2pc/[r]^{g_1+g_2}="1" \ar@/_2pc/[r]_{g_1+g_2}="2" \ar@{=>}"1";"2"|*+{\tau_2+\lw {g_2} \tau_1} & \ast\text{.}}
\end{equation*}

\end{itemize}
We consider the following strict 2-functor 
\begin{equation*}
F_2: \sheaf{B\TBFF}(X) \to \mathcal{F}_2(X)\text{.}
\end{equation*}
It associates to the single object the T-background with the trivial $\T^{n}$-bundle $X \times \T^{n}$ over $X$ and the trivial gerbe $\mathcal{I}$ over $X\times \T^{n}$. To a 1-morphism $g$ it associates the 1-morphism consisting of the bundle morphism $f_g: X \times \T^{n} \to X \times \T^{n}:(x,a) \mapsto (x,g(x)+a)$, and the identity 1-morphism $\id_{\mathcal{I}}$ between $\mathcal{I}$ and $f_g^{*}\mathcal{I}=\mathcal{I}$. This respects strictly the composition. To a 2-morphism $\tau:g\Rightarrow g$ it associates the bundle gerbe 2-morphism $\beta_{\tau}: \id_{\mathcal{I}} \Rightarrow \id_{\mathcal{I}}$ over $X \times \T^{n}$ which is induced by the smooth map $\beta_{\tau}:X \times \T^{n}\to \ueins:(x,a) \mapsto -\tau(x)(g(x)+a)$. It is clear that 
\begin{equation*}
\beta_{\tau_1+\tau_2}=\beta_{\tau_1}+\beta_{\tau_2}=\beta_{\tau_2} \bullet \beta_{\tau_1}\text{,}
\end{equation*}
i.e. the vertical composition is respected. It remains to show that the horizontal composition is respected, i.e. 
\begin{equation}
\label{eq:horcomprespF2}
\beta_{\tau_2} \circ \beta_{\tau_1} = \beta_{\tau_2 + \lw{g_2} \tau_1}\text{,}
\end{equation}
where $\tau_1:g_1\Rightarrow g_1$ and $\tau_2:g_2\Rightarrow g_2$. On the left we have horizontal composition in $\tba (X)$, so that it is $f_{g_1}^{*}\beta_{\tau_2} \circ \beta_{\tau_1}$ in terms of the composition bundle gerbe 2-morphisms.  We have 
\begin{multline*}
(f_{g_1}^{*}\beta_{\tau_2} \circ \beta_{\tau_1})(x,a)=\beta_{\tau_2}(x,g_1+a)+\beta_{\tau_1}(x,a)= -\tau_2(x)(g_1+g_2+a)-\tau_1(x)(g_1+a)\\=-(\tau_2+\lw {g_2}\tau_1)(x)(g_1+g_2+a)=\beta_{\tau_2 + \lw{g_2} \tau_1}(x,a)\text{.}
\end{multline*}
This proves \cref{eq:horcomprespF2} and shows that $F_2$  is a 2-functor.

\begin{proposition}
\label{prop:classF2}
The 2-functor $F_2$ induces an equivalence $(\sheaf{B\TBFF})^{+} \cong \tbaFF$.
\end{proposition}

\begin{proof}
We show that the 2-functor $F_2$ is an isomorphism of presheaves. Then, it becomes an isomorphism between the 2-stackifications; this is the claim. The 2-functor is bijective on the level of objects; in particular it is essentially surjective.  It is also bijective on the level of 1-morphisms, since the automorphisms of the trivial $\T^{n}$-bundle over $X$ are exactly the smooth $\T^{n}$-valued functions on $X$. Finally, it is bijective on the level of 2-morphisms, since the automorphisms of $\id_{\mathcal{I}}$ are exactly the smooth $\ueins$-valued functions on $X \times \T^{n}$. \end{proof}

\begin{remark}
\label{sec:F2:coc}
For any Lie 2-group $\Gamma$ one can describe the objects in  $(\sheaf{B\Gamma})^{+}(X)$ by  $\Gamma$-cocycles with respect to an open cover $\{U_i\}_{i\in I}$ of $X$;  see \cref{sec:semistrictcocycles} for a general description. A $\TBFF[n]$-cocycle  consists of smooth maps
\begin{equation*}
g_{ij}:U_i \cap U_j \to \mathbb{T}^{n}
\quand
\tau_{ijk}:U_i \cap U_j \cap U_k \to C^{\infty}(\T^{n},\ueins)
\end{equation*}
satisfying the cocycle conditions
\begin{align}
\label{coc:F2:1}
g_{ik} &= g_{jk} + g_{ij}
\\
\label{coc:F2:2}
\tau_{ikl} + \lw{g_{kl}} \tau_{ijk} & =\tau_{ijl}+ \tau_{jkl}\text{.}
\end{align}
Two  $\TBFF[n]$-cocycles $(g,\tau)$ and $(g',\tau')$ are equivalent if there exist smooth maps 
\begin{equation*}
h_i:U_i \to \T^{n}
\quand
\varepsilon_{ij}:U_i \cap U_j \to C^{\infty}(\T^{n},\ueins)
\end{equation*}
such that
\begin{align}
\label{coc:F2:3}
g_{ij}'+h_i &= h_j+g_{ij}
\\
\label{coc:F2:4}
\tau_{ijk}'+\lw{g'_{jk}}\varepsilon_{ij}+\varepsilon_{jk} &= \varepsilon_{ik}+ \lw{h_k} \tau_{ijk}\text{.}
\end{align}
\end{remark}

\begin{remark}
\label{re:cocTBFFR}
Similarly, a $\TBFFR$-cocycle consists of numbers $m_{ijk}\in \Z^{n}$ and smooth maps $a_{ij}:U_i \cap U_j \to \R^{n}$ and $\tau_{ijk}$ satisfying 
\begin{align}
\label{coc:F2R:1}
a_{ik} &= a_{jk} + a_{ij}+m_{ijk}
\\
\label{coc:F2R:2}
\tau_{ikl} + \lw{a_{kl}} \tau_{ijk} & =\tau_{ijl}+ \tau_{jkl}\text{.}
\end{align}
Here we have identified smooth maps $U_i \cap U_j \cap U_k \to \Z^{n}$ with elements in $\Z^{n}$, since we can refine any open cover by one whose open sets and all finite intersections are either empty or connected. 
Note that \cref{coc:F2R:1} implies
\begin{equation}
\label{coc:F2R:3}
m_{ikl}+m_{ijk}=m_{ijl}+m_{jkl}\text{.}
\end{equation}
Two $\TBFFR$-cocycles $(a,\tau,m)$ and $(a',\tau',m')$ are equivalent if there exist numbers $z_{ij}\in \Z^{n}$ and smooth maps $p_i: U_i \to \R^{n}$ and $\varepsilon_{ij}:U_i \cap U_j \to C^{\infty}(\T^{n},\ueins)$ satisfying
\begin{align}
\label{eq:F2R:4}
a_{ij}'+p_i &= z_{ij}+p_j+a_{ij}
\\
\label{eq:F2R:5}
\tau_{ijk}'+\lw{a'_{jk}}\varepsilon_{ij}+\varepsilon_{jk} &= \varepsilon_{ik}+ \lw{p_k} \tau_{ijk}\text{.}
\end{align}
Note that \cref{eq:F2R:4} implies
\begin{equation}
\label{coc:F2R:6}
m'_{ijk}+z_{ij}+z_{jk} =z_{ik}+m_{ijk}\text{.}
\end{equation}
\end{remark}

\begin{remark}
\label{re:cocF3}
Cocycles for $\TBFF$ or $\TBFFR$ where $\tau_{ijk}$ takes values in the constant $\ueins$-valued functions $\ueins \subset C^{\infty}(\T^{n},\ueins)$ correspond to $F_3$ T-backgrounds. \end{remark}

\subsection{A  2-group that represents  $F_1$ T-backgrounds}

\label{sec:TBFFR}

We show in this section that $F_1$ T-backgrounds are obtained by letting the (additive) group $\mathfrak{so}(n,\Z)$ of skew-symmetric $n \times n$ matrices with integer entries act on the 2-group $\TBFFR$ constructed in the previous section. In this context $\mathfrak{so}(n,\Z)$ appears as the group $\h^2(\T^n,\Z)$ of isomorphism classes of principal $\ueins$-bundles over $\T^{n}$; a group isomorphism is induced by an assignment $B \mapsto \poi_B$ of a principal $\ueins$-bundle $\poi_B$ to a matrix $B\in \mathfrak{so}(n,\Z)$, which we described in \cref{sec:poi}.

We define for each $B\in\mathfrak{so}(n,\Z)$ a crossed intertwiner 
\begin{equation*}
F_B:\TBFFR[n] \to \TBFFR[n]\text{,}
\end{equation*} 
in the notation of \cref{def:CI} as the triple $F_B=(\id_{\R^{n}},f_B,\eta_B)$,
with smooth maps
\begin{equation*}
f_B:C^{\infty} (\T^{n},\ueins) \times \Z^{n} \to C^{\infty} (\T^{n},\ueins) \times \Z^{n}
\quand 
\eta_B: \R^{n} \times \R^{n} \to C^{\infty} (\T^{n},\ueins)
\end{equation*}
given by
\begin{align*}
f_B(\tau,m) := (\tau - \bra mB,m)
\quand
\eta_B(a,a') :=  \lbraket {a'}B{a}\text{,}
\end{align*}
where $\eta_B(a,a')$ is considered as a constant $\ueins$-valued map.
It is straightforward to check the axioms of a crossed intertwiner: \cref{CI1*,CI2*} are trivial, \cref{CI4*} follows from the skew-symmetry of $B$, and \cref{CI5*} follows from the bilinearity and $\R$-invariance of $\eta_B$. 

For $B_1,B_2\in \mathfrak{so}(n,\Z)$ we compute the composition of the corresponding crossed intertwiners using the formula of the \cref{def:CIcomp}:
\begin{multline*}
F_{B_2} \circ F_{B_1}=(\id,f_{B_2},\eta_{B_2})\circ (\id,f_{B_1},\eta_{B_1}) \\=(\id, f_{B_2} \circ f_{B_1},\eta_{B_2} +\eta_{B_1} )=(\id,f_{B_1+B_2},\eta_{B_1+B_2})=F_{B_2+B_1}\text{.} \end{multline*}
Thus, we have an action of $\mathfrak{so}(n,\Z)$ on $\TBFFR$ by crossed intertwiners in the sense of \cref{def:action}.
As explained in \cref{sec:semidirect} we can now form a semi-strict Lie 2-group
\begin{equation*}
\TBF  := \TBFFR \ltimes \mathfrak{so}(n,\Z)\text{,}
\end{equation*}
whose two invariants are
\begin{equation}
\label{eq:piF1}
\pi_0(\TBF)=\T^{n} \times \mathfrak{so}(n,\Z)
\quand
 \pi_1(\TBF)=C^{\infty}(\T^{n},\ueins)\text{.}
\end{equation}

Since semi-strict 2-groups have no description as crossed modules, we can only describe $\TBF$ as a monoidal category, where the monoidal structure is the multiplication. Reducing the general theory of \cref{sec:semidirect} to the present situation, we obtain the following. The objects of $\TBF$ are pairs $(a,B)$ with $a\in \R^{n}$ and $B\in \mathfrak{so}(n,\Z)$, and multiplication is the  direct product group structure,
\begin{equation*}
(a_2,B_2)\cdot (a_1,B_1) =(a_2+a_1,B_2+B_1)\text{.} 
\end{equation*}  
The morphisms are tuples $(\tau,m,a,B)$ with source $(a,B)$ and target $(m+a,B)$. The composition is 
\begin{equation*}
(\tau_2,m_2,m_1+a_1,B) \circ (\tau_1,m_1,a_1,B) := (\tau_2+\tau_1,m_1+m_2,a_1,B)\text{.}
\end{equation*}
Multiplication is given by
\begin{multline*}
(\tau_2,m_2,a_2,B_2)\cdot (\tau_1,m_1,a_1,B_1)\\= (\tau_2-\lbraket {a_1}{B_2}{m_1}+\lw{a_2}\tau_1-\lw{a_2} \bra {m_1}{B_2},m_2+m_1,a_2+a_1,B_2+B_1)\text{.}
\end{multline*}
The semi-strictness of this 2-group corresponds to the fact that this multiplication is not strictly associative; instead, it has an  associator that satisfies a pentagon axiom. The associator is given by the formula 
\begin{equation*}
\lambda((a_3,B_3),(a_2,B_2),(a_1,B_1)) = (-\lbraket {a_1}{B_3}{a_2},0,a_3+a_2+a_1,B_3+B_2+B_1)\text{.}
\end{equation*}

\label{sec:F1:gerbesandbackgrounds}

Next we establish a relation between the 2-group $\TBF$ and $F_1$ T-backgrounds. We consider the presheaf $\sheaf{B\TBF}$ of smooth $B\TBF$-valued functions. Consulting \cref{sec:semistrictcocycles}, the bicategory $\sheaf{B\TBF}(X)$  over a smooth manifold $X$ is the following:
\begin{itemize}

\item 
It has one object.

\item
The morphisms are pairs $(a,B)$ of a smooth map $a: X \to \R^{n}$ and a skew-symmetric matrix $B\in \mathfrak{so}(n,\Z)$. The composition is (point-wise) addition.

\item
There are only 2-morphisms between $(a,B)$ and $(a',B')$ if $a'-a\in \Z^{n}$ and $B'=B$; in this case a 2-morphism is a pair $(\tau,m)$ with $\tau: X \to C^{\infty}(\T^{n},\ueins)$ and $m\in \Z^{n}$ such that $a'-a=m$. The vertical composition is (pointwise) addition, and the horizontal composition
is
\begin{equation*}
\alxydim{@C=2cm}{\ast \ar@/^2pc/[r]^{(a_1,B_1)}="1" \ar@/_2pc/[r]_{(a_1+m_1,B_1)}="2" \ar@{=>}"1";"2"|*+{(\tau_1,m_1)} & \ast \ar@/^2pc/[r]^{(a_2,B_2)}="1" \ar@/_2pc/[r]_{(a_2+m_2,B_2)}="2" \ar@{=>}"1";"2"|*+{(\tau_2,m_2)} & \ast}
=
\alxydim{@C=8cm}{\ast \ar@/^3pc/[r]^{(a_1+a_2,B_1+B_2)}="1" \ar@/_3pc/[r]_{(a_1+a_2+m_1+m_2,B_1+B_2)}="2" \ar@{=>}"1";"2"|*+{\tau_2+\lw{a_2}\tau_1-\lbraket {a_1}{B_2}{m_1}-\lw{a_2} \bra {m_1}{B_2}} & \ast\text{.}}
\end{equation*}

\item
The associator
\begin{equation*}
\lambda_{(a_3,B_3),(a_2,B_2),(a_1,B_1)}: ((a_3,B_3)\circ (a_2,B_2))\circ (a_1,B_1) \Rightarrow (a_3,B_3)\circ ((a_2,B_2)\circ (a_1,B_1))
\end{equation*}
is the pair $(-\lbraket {a_1}{B_3}{a_2},0)$. 

\end{itemize}
Next we construct a 2-functor 
\begin{equation*}
F_1: \sheaf{B\TBF}(X) \to \mathcal{F}_1\text{.}
\end{equation*}
It associates to the single object the T-background $(X \times \T^{n},\mathcal{I})$. To a 1-morphism $(a,B)$ it associates the 1-morphism $(f_a,\mathcal{B}_B)$, consisting of the bundle morphism 
\begin{equation}
\label{eq:deffa}
f_a: X \times \T^{n} \to X \times \T^{n}:(x,b) \mapsto (x,a(x)+b)
\end{equation}
and the following  1-morphism $\mathcal{B}_B: \mathcal{I} \to f_a^{*}\mathcal{I}=\mathcal{I}$  over $X \times \T^{n}$. We recall that  the groupoid $\aut(\mathcal{G})$ of automorphisms of a bundle gerbe $\mathcal{G}$ over a smooth manifold $M$ is a module category over the  monoidal groupoid $\bun\ueins M$ of principal $\ueins$-bundles over $M$ in terms of a functor
\begin{equation}
\label{eq:bunact}
\aut(\mathcal{G}) \times \bun\ueins M \to \aut(\mathcal{G}): (\mathcal{A},P) \mapsto \mathcal{A} \otimes P\text{.}
\end{equation}
The 1-morphism $\mathcal{B}_B$ we want to construct is obtained by letting the principal $\ueins$-bundle $\pr_{\T^{n}}^{*}\poi_B$ over $X\times \T^{n}$ act on the identity, where $\poi_B$ is the matrix-depending version of the Poincaré bundle  defined in \cref{sec:npoi}. Thus, $\mathcal{B}_B:=\id \otimes \pr_{\T^{n}}^{*}\poi_B$. 
To a 2-morphism 
\begin{equation*}
(\tau,m,a,B):(a,B) \Rightarrow (a+m,B)
\end{equation*}
it assigns the 2-morphism $\beta_{\tau,m,a,B}: (f_{a},\mathcal{B}_B) \Rightarrow (f_{a},\mathcal{B}_B)$ over $X \times \T^{n}$ induced  by acting on $\id_{\mathcal{B}_B}$ with an automorphism of the trivial $\ueins$-bundle over $X\times \T^{n}$, given  the smooth map 
\begin{equation*}
\kappa_{\tau,m,a,B}:X \times \T^{n}\to \ueins:(x,b) \mapsto -\tau(x)(a(x)+b)\text{,}
\end{equation*}
i.e., $\beta_{\tau,m,a,B} := \id_{\mathcal{B}_B} \otimes \kappa_{\tau,n,a,B}$.
It is straightforward to check that the vertical composition is respected.
The horizontal composition is  not strictly preserved: we have $f_{a_2+a_1}=f_{a_2}\circ f_{a_1}$ but $\mathcal{B}_{B_2+B_1}\neq f_{a_1}^{*}\mathcal{B}_{B_2} \circ \mathcal{B}_{B_1}$. A compositor 
\begin{equation*}
c_{(a_1,B_1),(a_2,B_2)}: (f_{a_2},\mathcal{B}_{B_2})  \circ  (f_{a_1},\mathcal{B}_{B_1}) \Rightarrow (f_{a_2+a_1},\mathcal{B}_{B_2+B_1})
\end{equation*}
is induced over $\{x\} \times \T^{n}$ from the isomorphism
\begin{equation*}
\alxydim{@C=3cm}{\poi_{B_1} \otimes r_{x_1a_1(x)}^{*}\poi_{B_2}  \ar[r]^-{\id \otimes \tilde R_{B_2}(a_1(x))^{-1}} & \poi_{B_1} \otimes \poi_{B_2}= \poi_{B_2+B_1}}
\end{equation*}
of $\ueins$-bundles over $ \T^{n}$, where the bundle morphism $\tilde R_B$ is explained in \cref{sec:poi}.  We have to show that this compositor satisfies a pentagon diagram.
This diagram can be reduced to the following condition about the equivariance of the Poincaré bundle:
\begin{equation*}
\alxydim{@R=5em@C=-2em}{
&& \poi_{B_1} \otimes r_{a_1}^{*}\poi_{B_2}  \otimes r_{a_2+a_1}^{*}\poi_{B_3}  \ar@{=>}[dll]_-{\id \otimes \id \otimes r_{a_1}^{*}\tilde R_{B_3}(a_2)^{-1} \quad} \ar@{=>}[drr]^-{\quad\id \otimes \tilde R_{B_2}(a_1)^{-1} \otimes \id}&&
\\\poi_{B_1} \otimes r_{a_1}^{*}\poi_{B_2}\otimes r_{a_1}^{*}\poi_{B_3} \hspace{-3em} \ar@{=>}[dr]_-{\id \otimes \tilde R_{B_2}(a_1)^{-1}\otimes \tilde R_{B_3}(a_1)^{-1}}&&  &&\hspace{-3em} \poi_{B_1} \otimes \poi_{B_2}\otimes r_{a_2+a_1}^{*}\poi_{B_3}   \ar@{=>}[dl]^-{\id \otimes \id \otimes \tilde R_{B_3}(a_1+a_2)^{-1}}
\\&\hspace{-2em}\poi_{B_1} \otimes \poi_{B_2}\otimes\poi_{B_3} \ar@{=>}[rr]_{ \lbraket {a_1}{B_3}{a_2}} &&\poi_{B_1} \otimes \poi_{B_2}\otimes \poi_{B_3}\hspace{-2em}&}
\end{equation*}
Splitting this diagram into the three tensor factors, the only nontrivial diagram is the third tensor factor, where it becomes
\begin{equation*}
\alxydim{@C=7em@R=3em}{r_{a_2+a_1}^{*}\poi_{B_3} \ar@{=>}[d]_{ r_{a_1}^{*}\tilde R_{B_3}(a_2)^{-1}} \ar@{=>}[r]^-{\tilde R_{B_3}(a_1+a_2)^{-1}}  & \poi_{B_3} \\ r_{a_1}^{*}\poi_{B_3} \ar@{=>}[r]_{ \tilde R_{B_3}(a_1)^{-1}}  &  \poi_{B_3} \ar@{=>}[u]_{\lbraket {a_1}{B_3}{a_2}}}
\end{equation*}
The commutativity of this diagram is precisely \cref{eq:equivpoiBtilde}; this finishes the definition of the compositor. 
Now it remains to show that the horizontal composition is respected relative to the compositor. 
Employing the horizontal composition in $\tba (X)$ and the definitions of the 2-functor on 1-morphisms and of the compositor, this condition  becomes the commutativity of the diagram
\begin{equation}
\label{eq:horcompresp}
\alxydim{@C=19em@R=3em}{\poi_{B_1} \otimes r_{x_1a_1}^{*}\poi_{B_2} \ar@{=>}[r]^-{\kappa_{\tau_2,m_2,a_1+a_2,B_2} +\kappa_{\tau_1,m_1,a_1,B_1} } \ar@{=>}[d]|{\id \otimes \tilde R_{B_2}(a_1)^{-1}} & \poi_{B_1} \otimes r_{x_1a_1}^{*}\poi_{B_2} \ar@{=>}[d]|{\id \otimes \tilde R_{B_2}(a_1+m_1)^{-1}} 
\\
\poi_{B_1}\otimes \poi_{B_2} \ar@{=>}[r]_-{\kappa_{\tau_2+\lw{\alpha_2}{\tau_1}-\lbraket {a_1}{B_2}{m_1}-\lw{a_2} \bra {m_1}{B_2},m_2+m_1,a_2+a_1,B_2+B_1}} & \poi_{B_1}\otimes \poi_{B_2} }
\end{equation}
It can be checked in a straightforward way using the definition of $\kappa$ and \cref{eq:equivpoiBZcomb}.
This finishes the construction of the 2-functor $F_1$.

\begin{proposition}
\label{prop:classF1}
The 2-functor $F_1$ induces an equivalence $(\sheaf{\TBF})^{+} \cong \tbaF$.
\end{proposition}

\begin{proof}
We show that the 2-functor $F_1$ is an isomorphism of presheaves for each contractible open set $U \subset X$. Then, it becomes an isomorphism between the stackifications; this is the claim. The 2-functor is bijective on the level of objects; in particular it is essentially surjective. The Hom-functor
\begin{equation*}
F_1: \hom_{\sheaf{B\TBF}(U)}(\ast,\ast) \to \hom_{\mathcal{F}_1}(U \times \T^{n},U \times \T^{n}) = \hom_{\tba(U)}(U \times \T^{n},U \times \T^{n})
\end{equation*}
is clearly fully faithful, and we claim that it is  essentially surjective. Indeed, if $(f,\mathcal{B})$ is some automorphism of $U \times \T^{n}$, then $f=f_a$ for some smooth map $a:U \to \R^{n}$ since $U$ is simply-connected, and  $\mathcal{B}=\pr_{\T^{n}}^{*}\poi_B$ for some matrix $B \in \mathfrak{so}(n,\Z)$, as the cohomology of $U \times \T^{n}$ has no contributions from $U$ since $U$ is 2-connected. 
\end{proof}

\begin{remark}
In \cref{sec:semidirect:cocycles} we describe cocycles for semi-strict Lie 2-groups. A $\TBF[n]$-cocycle with respect to an open cover $\{U_i\}_{i\in I}$ is a quadruple $(B,a,m,\tau)$ consisting of matrices $B_{ij}\in \mathfrak{so}(n,\Z)$, numbers $m_{ijk}\in \Z^{n}$, and smooth maps 
\begin{equation*}
a_{ij}:U_i \cap U_j \to \R^{n}
\quand
\tau_{ijk}:U_i \cap U_j \cap U_k \to C^{\infty}(\T^{n},\ueins)
\end{equation*}
subject to the relations
\begin{align}
\label{coc:F1:1}
B_{ik} &= B_{jk} + B_{ij}
\\
\label{coc:F1:2}
a_{ik} &= a_{jk} + a_{ij} + m_{ijk}
\\
\label{coc:F1:3}
\tau_{ikl}+\lw{a_{kl}}\tau_{ijk}-\lw{a_{kl}}\bra{m_{ijk}}{B_{kl}} &= \lbraket{a_{ik}}{B_{kl}}{m_{ijk}} +\lbraket{a_{ij}}{B_{kl}}{a_{jk}}+ \tau_{ijl}+ \tau_{jkl}\text{.} 
\end{align}
Note that the \cref{coc:F1:2} implies
\begin{equation}
\label{coc:F1:5}
m_{ikl}+ m_{ijk} = m_{ijl} + m_{jkl}\text{.}
\end{equation}
Two $\TBF$-cocycles $(B,a,m,\tau)$ and $(B',a',m',\tau')$ are equivalent if there exist matrices $C_i \in \mathfrak{so}(n,\Z)$, numbers $z_{ij}\in \Z$ and smooth maps 
\begin{equation*}
p_i: U_i \to \R^{n}
\quand
\varepsilon_{ij}:U_i \cap U_j \to C^{\infty}(\T^{n},\ueins)
\end{equation*}
such that 
\begin{align}
\label{coc:F1:6}
C_j + B_{ij}&=B'_{ij}+ C_i
\\
\label{coc:F1:7}
z_{ij} + p_j+ a_{ij}&=a_{ij}'+p_i
\end{align}
and
\begin{multline}
\label{coc:F1:8}
\tau_{ijk}'+ \lbraket{p_i}{B_{jk}'}{a_{ij}'}- \lbraket {p_j +a_{ij}}{B_{jk}'}{z_{ij}}+ \lw{a'_{jk}}\varepsilon_{ij}-\lw{a'_{jk}}\bra{z_{ij}}{B_{jk}'}- \lbraket {a_{ij}}{B_{jk}'}{p_j}+ \varepsilon_{jk}
\\=\varepsilon_{ik} - \lbraket{a_{ik}}{C_k}{m_{ijk}}+\lw{p_k}\tau_{ijk}-\lw{p_k}\bra{m_{ijk}}{C_k}- \lbraket{a_{ij}}{C_k}{a_{jk}}
\end{multline}
Note that \cref{coc:F1:7} implies
\begin{equation}
\label{coc:F1:9}
m'_{ijk}+z_{ij}+z_{jk} = z_{ik} +m_{ijk}\text{.}
\end{equation}
We remark that the subclass of $F_2$ backgrounds is represented by cocycles with $B_{ij}=0$; in that case, above equations are precisely the cocycle conditions for the 2-group $\TBFFR[n]$.
\end{remark}

\begin{remark}
\label{re:TBF:2}
\label{re:TBF:3}
\label{re:F2ToF1}
The sequence
\begin{equation*}
\TBFFR \to \TBF \to \idmorph{\mathfrak{so}(n,\Z)}\text{,}
\end{equation*}
of semi-strict Lie 2-groups and semi-strict homomorphisms
induces by \cref{prop:fibresequence}  the following exact sequence in cohomology:
\begin{equation*}
\h^1(X,\TBFFR)/\mathfrak{so}(n,\Z) \to \h^1(X,\TBF) \to \h^1(X,\mathfrak{so}(n,\Z)) \to 0\text{.}
\end{equation*}
Exactness implies the following results:
\begin{enumerate}[(a)]
\item 
Every $F_1$ T-background $(E,\mathcal{G})$ has an underlying $\mathfrak{so}(n,\Z)$-principal bundle, and if this bundle is trivializable, then $(E,\mathcal{G})$ is isomorphic (as $F_1$ T-backgrounds) to an $F_2$ T-background. 

\item
Since a principal $\mathfrak{so}(n,\Z)$-bundle is nothing but a collection of $\frac{1}{2}n(n-1)$ principal $\Z$-bundles, every such bundle is trivializable if $\h^1(X,\Z)=0$, for instance if $X$ is connected and simply connected. Hence, every $F_1$ T-background over a connected and simply-connected smooth manifold $X$ is in fact $F_2$; in particular, it is T-dualizable (\cref{th:bsr}). This question has been considered in \cite{Belov2007}, also see \cref{sec:half-geo:pol}.

\item
The map
$\h^1(X,\TBFFR) \to \h^1(X,\TBF)$ is indeed not injective, corresponding to the fact that the inclusion $\tbaFF \incl \tbaF$ is not full. In fact, two $F_2$ T-backgrounds are isomorphic as $F_1$ T-backgrounds if and only if they are related by the $\mathfrak{so}(n,\Z)$-action (the \quot{only if} holds only for connected $X$).

\end{enumerate}
\end{remark}

\subsection{T-backgrounds with trivial torus bundle}

\label{ex:F1}
\label{prop:cup}

In this section we investigate $F_1$ T-backgrounds with trivial torus bundle. 
For this purpose we consider a sequence
\begin{equation}
\label{eq:trivtor}
\alxydim{}{\idmorph{\mathfrak{so}(n,\Z)} \times B\Z^{n} \times B\ueins \ar[r]^-{I} & \TBF_{\phantom{X}} \ar[r]^-{T} & \idmorph{\T^{n}}}
\end{equation}
of semi-strict homomorphisms. 
The homomorphism $T$ sends an object $(a,B)$ of $\TBF$ to  $a\in \T^{n}$ reduced mod $\Z^{n}$, and a morphism $(\tau,m,a,B)$ to the identity morphism of $a$. 
The homomorphism $I$ sends an object $(B,\ast,\ast)$ to the object $(0,B)$ of $\TBF$ and an endomorphism $(B,m,t)$ of $(B,\ast,\ast)$ to the endomorphism $(\tau_{t,m},0,0,B)$ of $(0,B)$ in $\TBF$, where $\tau_{t,m}(c)= t+mc$. 

\begin{lemma}
\label{lem:trivtorus}
The following sequence is exact:
\begin{equation*}
\alxydim{}{\h^1(X,\mathfrak{so}(n,\Z)) \times \h^2(X,\Z^{n}) \times \h^3(X,\Z) \ar[r]^-{I_{*}} & \h^1(X,\TBF) \ar[r]^-{T_{*}} & \h^2(X,\Z^{n}) \ar[r] & 0\text{.}}
\end{equation*}
Here we have used the usual identifications between non-abelian cohomology and ordinary cohomology, see \cref{re:ordinarycohomology}. 
\end{lemma}

\begin{proof}
The homomorphism $T \circ I$ is trivial, and $T_{*}$ is obviously surjective. Suppose $(B,a,m,\tau)$ is a $\TBF$-cocycle whose class vanishes under $T$. 
Thus, there exists $b_i:U_i \to \T^{n}$ such that $a_{ij}=b_j -b_i$. We can assume that $b_i$ lift to smooth maps $p_i:U_i \to \R^{n}$; then we obtain $z_{ij}\in \Z^{n}$ such that $a_{ij}=p_j-p_i+z_{ij}$. These establish an equivalence between $(B_{ij},a_{ij},m_{ijk},\tau_{ijk})$ and $(B_{ij},0,0,\tau'_{ijk})$, where $\tau_{ijk}': U_i \cap U_j \cap U_k \to C^{\infty}(\T^{n},\ueins)$ is a 3-cocycle.

For $1\leq p \leq n $ we consider the group homomorphism $w_p:C^{\infty}(\T^{n},\ueins) \to \Z$ that extracts the winding number of a map around the $p$-th torus component. We also consider the evaluation at $0$, which is a Lie group homomorphism $\ev_0: C^{\infty}(\T^{n},\ueins) \to \ueins$. 
The composition of $\tau'_{ijk}$ with $w_p$ is locally constant and thus a 3-cocycle  $m_{ijk}\in \Z^{n}$. The composition with $\ev_0$ is a 3-cocycle $t_{ijk}:U_i \cap U_j \cap U_k \to \ueins$. We claim that $\tau_{ijk}'$ is equivalent to the 3-cocycle $\tau_{t_{ijk},m_{ijk}}$, so that $(B_{ij},0,0,\tau'_{ijk})$ is in the image of $I$. 
Indeed, by definition of a winding number,  $(x,a)\mapsto \tau'_{ijk}(x)(a)-m_{ijk}a$ lifts to a smooth map $\tilde\tau_{ijk}:U_i \cap U_j \cap U_k \to C^{\infty}(\T^{n},\R)$. The lift $\tilde\tau_{ijk}$ satisfies the cocycle condition only up to a constant $\varepsilon_{ijkl}\in \Z$; hence,  the smooth maps  $\beta_{ijk}:U_i \cap U_j \cap U_k \to C^{\infty}(\T^{n},\R)$ defined by
\begin{equation*}
\beta_{ijk}(x)(a):= \tilde\tau_{ijk} (x,a)-\tilde\tau_{ijk}(x,0)
\end{equation*}
do form a cocycle. Since $\check \h^2(X,\sheaf \R)=0$, there exists $e_{ij}:U_i \cap U_j \to \R$ with coboundary $\beta_{ijk}$. Pushing to $\ueins$-valued maps, $e_{ij}$ establishes an equivalence between $\tilde\tau_{ijk}$ and $t_{ijk}$; hence between $\tau'_{ijk}$ and $\tau_{t_{ijk},m_{ijk}}$.  
\end{proof}

The sequence  \cref{eq:trivtor} restricts over the $F_2$ T-backgrounds, ending up in a diagram of semi-strict Lie 2-groups  and semi-strict homomorphisms:
\begin{equation}
\alxydim{@R=3em}{B\Z^{n} \times B\ueins \ar[d] \ar[r] & \TBFFR_{\phantom{X}} \ar[d] \ar[r] & \idmorph{\T^{n}} \ar@{=}[d]\\ \idmorph{\mathfrak{so}(n,\Z)} \times B\Z^{n} \times B\ueins  \ar[r]_-{I} & \TBF_{\phantom{X}} \ar[r]_-{T}  & \idmorph{\T^{n}} \text{.}}
\end{equation}

Concerning the geometric counterparts of the homomorphisms $I$ and $T$, it is clear that $T$ represents the 2-functor 
\begin{equation*}
\tbaF(X) \to \bun{{\T^{n}}}{X}: (E,\mathcal{G}) \mapsto E
\end{equation*}
that takes an $F_1$ T-background to its underlying torus bundle. 
Concerning the strict homomorphism $I$, we   describe the 2-functor
\begin{equation}
\label{eq:geoI}
\bun{\mathfrak{so}(n,\Z)}X \times \grb{\Z^{n}}X \times \grb\ueins X \to \tba
\end{equation}
represented by $I$.
We start by treating the first factor and  assume that we have an $\mathfrak{so}(n,\Z)$-bundle $Z$ over $X$. We construct the following bundle gerbe $\mathcal{R}_{\mathfrak{so}(n,\Z)}(Z)$ over $X \times \T^{n}$. Its surjective submersion is $Z \times \T^{n} \to X \times \T^{n}$. Its 2-fold fibre product $Z^{[2]} \times \T^{n}$ is equipped with a smooth (and hence locally constant) map $B: Z^{[2]} \to \mathfrak{so}(n,\Z)$, since $Z$ is a principal bundle, and it is equipped with the projection  $\pr_{\T^{n}}$ to $\T^{n}$. We consider the principal $\ueins$-bundle $P := \pr_{\T^{n}}^{*}\poi_B$ over $Z^{[2]} \times \T^{n}$, and over the triple fibre product the bundle gerbe product induced by the equality $\pr_{12}^{*}B + \pr_{23}^{*}B=\pr_{12}^{*}B$ over $Z^{[3]}$.

For the second factor, let $\mathcal{H}$ be a $\Z^{n}$-bundle gerbe over $X$.  We assume that $\mathcal{H}$ is defined over a surjective submersion $\pi:Y \to X$, with a 
principal $\Z^{n}$-bundle $P$ over $Y^{[2]}$ and a bundle gerbe product $\mu$ over $Y^{[3]}$. We define the following $\ueins$-bundle gerbe $\mathcal{R}_{\Z}(\mathcal{H})$ over $X \times \T^{n}$. 
Its surjective submersion is $\tilde Y :=Y \times \T^{n}\to X \times \T^{n}$. We consider the map $\tau:( Y^{[2]} \times \T^{n}) \times \Z^{n} \to \ueins: (y_1,y_2,a,m)\mapsto am$, which is fibrewise a 
group homomorphism. The principal $\ueins$-bundle of $\mathcal{R}_{\Z}(\mathcal{H})$ is the parameter-dependent bundle extension $\tau_{*}(P)$. Similarly, we extend the bundle gerbe product $\tau_{*}(\mu)$.

For the third factor, we simply pull back a $\ueins$-bundle gerbe $\mathcal{G}$ over $X$ to $X \times \T^{n}$. Putting the three pieces together, the 2-functor \cref{eq:geoI} is defined by
\begin{equation*}
(Z,\mathcal{H},\mathcal{G})\mapsto (X \times \T^{n},\mathcal{R}_{\mathfrak{so}(n,\Z)}(Z)\otimes \mathcal{R}_{\Z}(\mathcal{H})\otimes \pr_X^{*}\mathcal{G})\text{.}
\end{equation*}

In the following two lemmas we compute the Dixmier-Douady classes of the bundle gerbes $\mathcal{R}_{\mathfrak{so}(n,\Z)}(Z)$ and $\mathcal{R}_{\Z}(\mathcal{H})$, in order to see of which type the resulting T-backgrounds are. 
\begin{lemma}
\label{lem:DDsonZ}
Let $Z$ be a principal $\mathfrak{so}(n,\Z)$-bundle over $X$. For $1\leq p,q\leq n$ we have a principal $\Z$-bundle $Z^{pq} :=(\pr_{pq})_{*}(Z)$ with a corresponding class $[Z^{pq}]\in\h^1(X,\Z)$. 
We have
\begin{equation*}
\mathrm{DD}(\mathcal{R}_{\mathfrak{so}(n,\Z)}(Z)) = \sum_{1\leq q < p \leq n} \pr_p^{*}\gamma\cup\pr_q^{*}\gamma\cup \pr_X^{*}[Z^{pq}] \in \h^3(X \times \T^{n},\Z)\text{,}
\end{equation*}
where $\pr_p: X \times \T^{n} \to S^1$ is the projection to the $p$-th component of $\T^{n}$, and $\gamma\in \h^1(S^1,\Z)$ is a generator.  
\end{lemma}

\begin{proof}
We observe that $\pr_p^{*}\gamma\cup\pr_q^{*}\gamma$ is the first Chern class of $\pr_{pq}^{*}\poi$, and that the bundle gerbe $\mathcal{R}_{\mathfrak{so}(n,\Z)}(Z)$ is of the form
\begin{equation*}
\mathcal{R}_{\mathfrak{so}(n,\Z)}(Z) = \bigotimes_{1\leq q\leq p \leq n} \pr_{pq}^{*}\mathcal{R}_{\mathfrak{so}(2,\Z)}(Z^{pq})\text{,}
\end{equation*}
where $\pr_{pq}: X \times \T^{n} \to X \times \T^{2}$ projects to the two indexed components. We choose an open cover $\{U_i\}_{i\in I}$ of $X \times \T^{n}$ such that $Z^{pq}$ admits sections, leading to transition matrices $z_{ij}\in \Z$. Then, the bundle gerbe $\mathcal{R}_{\mathfrak{so}(2,\Z)}(Z^{pq})$ becomes isomorphic to one with principal $\ueins$-bundle $P_{ij}:=\pr_{\T^{2}}^{*}\poi^{z_{ij}}$ over $U_i \cap U_j$, and with the bundle gerbe product $\mu_{ijk}$ induced from the cocycle condition $z_{ij}+z_{jk}=z_{ik}$. We can additionally assume that there exist sections $s_i: U_i \to \pr_{\T^2}^{*}\poi$, leading to transition functions $g_{ij}:U_i \cap U_j \to \ueins$. Then, $P_{ij}$ has the section $s_i^{z_{ij}}$, and the calculation
\begin{equation*}
\mu_{ijk}(s_i^{z_{ij}} \otimes s_j^{z_{jk}})= \mu_{ijk}(s_i^{z_{ij}} \otimes s_i^{z_{jk}})\cdot g_{ij}^{z_{jk}} = s_i^{z_{ik}} \cdot g_{ij}^{z_{jk}}
\end{equation*}
show that $\mathcal{R}_{\mathfrak{so}(2,\Z)}(Z^{pq})$ is classified by the 3-cocycle $\eta_{ijk} := g_{ij}^{z_{jk}}$.

We compute cup products in $\Z$-valued \v Cech cohomology, where the cup product of a $k$-cocycle $\alpha_{i_0,...,i_k}$ with an $l$-cocycle $\beta_{i_0,...,i_l}$ is given by $\alpha_{i_0,...,i_k}\cdot \beta_{i_k,...,i_{k+l}}$, see \cite[Section 1.3]{brylinski1}. 
We choose lifts $\tilde g_{ij}:U_i \cap U_j \to \R$ of $g_{ij}$. Then, $q_{ijk}:=(\delta\tilde g)_{ijk}$ is a $\Z$-valued 3-cocycle that corresponds to $g_{ij}$ under the connecting homomorphism of the exponential sequence. Likewise,  $\tilde\eta_{ijk}:=\tilde g_{ij}z_{jk}$ is a lift of $\eta_{ijk}$, and $p_{ijkl}=(\delta\tilde\eta)_{ijkl}$ corresponds to $\eta_{ijk}$.  We calculate
\begin{multline*}
p_{ijkl} = \tilde\eta_{ikl}+\tilde\eta_{ijk} - \tilde\eta_{jkl}-\tilde\eta_{ijl}= \tilde g_{ik}z_{kl}+\tilde g_{ij}z_{jk} -\tilde g_{jk}z_{kl}-\tilde g_{ij}z_{jl}\\=(\tilde g_{ik}-\tilde g_{jk}-\tilde g_{ij})z_{kl}=q_{ijk}z_{kl}\text{.}
\end{multline*}
This shows that $[p]=[q]\cup [z]$. Since $[z]=[Z^{pq}]$ and $[q]=\mathrm{c}_1(\poi)$, this gives the claim.
\end{proof}

\begin{lemma}
\label{lem:DDGCinfty}
Let $\mathcal{H}$ be a $\Z^{n}$-bundle gerbe over $X$ classified by a class  $[\mathcal{H}] \in \h^2(X,\Z^{n})$. Let $[\mathcal{H}]_p \in \h^2(X,\Z)$ denote its $p$-th component. Then, 
\begin{equation*}
\mathrm{DD}(\mathcal{R}_{\Z}(\mathcal{H})) =  \sum_{p=1}^{n}[\mathcal{H}]_p \cup \pr_p^{*}\gamma \in \h^3(X \times \T^{n},\Z)\text{.} 
\end{equation*} 
\end{lemma}

\begin{proof}
We can assume that we have an open cover $\{U_i\}_{i\in I}$ on which the $\Z^{n}$-bundle gerbe $\mathcal{H}$ is given by a constant 3-cocycle $m_{ijk}\in \Z^{n}$. Then, $\mathcal{R}_{\Z}(\mathcal{H})$ has the surjective submersion \begin{equation*}
Y:= \coprod_{i\in I} U_i \times \T^{n} \to X \times \T^{n}\text{,}
\end{equation*}  
and the cocycle $U_i \cap U_j \cap U_k \times \T^{n}:(x,a) \mapsto m_{ijk}a$. In order represent its Dixmier-Douady class by a $\Z$-valued 4-cocycle, we consider the surjective submersion 
\begin{equation*}
Y' := \coprod_{i\in I} U_i \times \R^{n} \to X \times \T^{n}
\end{equation*}
so that the
cocycle $m_{ijk}a$ lifts to an $\R$-valued map $U_i \cap U_j \cap U_k \times \R^{n} \times_{\T^{n}} \R^{n} \times_{\T^{n}} \R^{n} \to \R$ defined by $(x,a_1,a_2,a_3)\mapsto m_{ijk}a_1$. Its coboundary is
\begin{equation}
\label{eq:4cocGm}
m_{ijk}a_1 + m_{ikl}a_1 - m_{jkl}a_2 - m_{ijl}a_1 = m_{jkl}(a_1-a_2)\text{,}
\end{equation}
which is a $\Z$-valued 4-cocycle representing the Dixmier-Douady class of $\mathcal{G}_m$.
In order to compute the cup product $(w_p)_{*}(\mathcal{H}) \cup \pr_p^{*}\gamma$ we have to represent both classes by $\Z$-valued \v Cech cocycles. 
By construction, the class of $(w_p)_{*}(\mathcal{H})$ is represented by the $\Z$-valued \v Cech 3-cocycle $m_{p,ijk}$, the $p$-th component of $m_{ijk}$. To represent the class $\pr_p^{*}\gamma$  by a $\Z$-valued 2-cocycle we consider again the surjective submersion $Y'$.
Then, the smooth map $\pr_p: X \times \T^{n} \to \ueins$, whose homotopy class represents $\pr_p^{*}\gamma$, lifts to the real valued projection  $\widetilde{\pr_{p}}: Y' \to \R$. Then, the $\Z$-valued 2-cocycle we are looking for is  
\begin{equation*}
\pr_{p,ij}: U_i \cap U_j \times \R^{n} \times_{\T^{n}} \R^{n} \to \Z:(x,a_1,a_2)\mapsto a_{p,1}-a_{p,2}\text{.}
\end{equation*}
Thus, our cup product $(w_p)_{*}(\mathcal{H}) \cup \pr_p^{*}\gamma$ is represented by the $\Z$-valued \v Cech 4-cocycle
\begin{equation*}
U_i \cap U_j \cap U_k \cap U_l \times \R^{n} \to \Z:  (x,a) \mapsto m_{p,ijk} \cdot (a_{p,1}-a_{p,2}) \text{.}
\end{equation*}
Summation over $p$ shows the coincidence with \cref{eq:4cocGm}.
\end{proof}

Summarizing above constructions and results, we have discussed a 2-functor
\begin{equation*}
\bun{\mathfrak{so}(n,\Z)}X \times \grb{\Z^{n}}X \times \grb\ueins X \to \tba(X)
\end{equation*}
that constructs an $F_1$ T-background with trivial  torus bundle, from a  principal $\mathfrak{so}(n,\Z)$-bundle, a $\Z^{n}$-bundle gerbe, and a $\ueins$-bundle gerbe over $X$. By \cref{lem:trivtorus}, these T-backgrounds are, up to isomorphism, \emph{all} $F_1$ T-backgrounds with trivial torus bundle.  We remark that the filtration $F_{k}=F_{k}\h^3(E,\Z)$  of \cref{eq:serre} in case of the trivial torus bundle $E=X \times \T^{n}$ can be expressed in terms of the Künneth formula by
\begin{equation*}
F^2\setminus F^3 \cong \h^2(X,\Z) \times \h^1(\T^{n},\Z)
\quomma
F^1 \setminus F^2 \cong  \h^1(X,\Z) \times \h^2(\T^{n},\Z)\text{.}
\end{equation*}
With \cref{lem:DDsonZ,lem:DDGCinfty} we can read off to which steps in the filtration the given structure contributes: the $\ueins$-bundle gerbe  $\mathcal{G}$  contributes to $F^3$, the $\Z^{n}$-bundle gerbe $\mathcal{H}$ contributes to $F^2$, and the $\mathfrak{so}(n,\Z)$-bundle $Z$ contributes to $F^1$. 

\begin{example}
\label{ex:F1ex}
We consider $X=S^1$ and $n=2$, and the trivial $\T^2$-bundle $E:=\T^3=S^1 \times \T^2$ over $S^1$.  We consider $Z:=\R$ (under the identification $\mathfrak{so}(2,\Z)\cong \Z$), and the corresponding bundle gerbe $\mathcal{G}:=\mathcal{R}_{\mathfrak{so}(n,\Z)}(Z)$ over $E$. We have $[Z^{12}]=\gamma\in \h^1(S^1,\Z)$ and obtain from \cref{lem:DDsonZ}: 
\begin{equation*}
\mathrm{DD}(\mathcal{G}) = \pr_1^{*}\gamma\cup\pr_2^{*}\gamma\cup \pr_3^{*}\gamma \in \h^3(\T^3,\Z)\text{,}
\end{equation*}
i.e. $\mathcal{G}$ represents the canonical class of $\T^3$.
This explicit example of a T-background has been described in \cite{Mathai2006}. It is interesting because it is not an $F_2$ T-background, and hence is not T-dualizable in the classical sense, see \cite{Bunke2006a} and \cref{th:bsr}. We will see that it gives rise to a half-geometric T-duality transformation in the formalism introduced in this paper, see \cref{ex:F12}.
\end{example}

\section{Higher geometry for topological T-duality}

\label{sec:Tdual}

In this section we discuss a bicategory of T-duality correspondences, and introduce a strict Lie 2-group $\TD$ that represents a 2-stack of T-duality correspondences.  We also discuss the relation to T-duality triples of \cite{Bunke2005a,Bunke2006a}.

\setsecnumdepth{2}

\subsection{T-duality-correspondences as 2-stacks}

\label{sec:geotduality}

Let $X$ be a smooth manifold.

\begin{definition}
\label{def:corr}
\begin{enumerate}[(a)]

\item 
\label{def:corr:a}
A \emph{correspondence} over $X$ consists of T-backgrounds $(E,\mathcal{G})$ and $(\hat E,\widehat{\mathcal{G}})$ over  $X$, and of a bundle gerbe isomorphism $\mathcal{D}: \pr_1^{*}\mathcal{G} \to \pr_2^{*}\widehat{\mathcal{G}}$ over $E \times_X \hat E$. 

\item
\label{def:corr:b}
A \emph{1-morphism} $((E,\mathcal{G}),(\hat E,\widehat{\mathcal{G}}),\mathcal{D})\to ((E',\mathcal{G}'),(\hat E',\widehat{\mathcal{G}}'),\mathcal{D}')$ consists of  1-morphisms  
\begin{equation*}
(f,\mathcal{B}):(E,\mathcal{G}) \to (E',\mathcal{G}')
\quand
(\hat f,\widehat{\mathcal{B}}):(\hat E,\widehat{\mathcal{G}}) \to (\hat E',\widehat{\mathcal{G}}')
\end{equation*}
between the T-backgrounds and of a bundle gerbe 2-morphism
\begin{equation*}
\alxydim{@C=6em@R=3em}{\pr_1^{*}\mathcal{G} \ar[d]_{\pr_1^{*}\mathcal{B}} \ar[r]^{\mathcal{D}} & \pr_2^{*}\widehat{\mathcal{G}} \ar@{=>}[dl]|*+{\zeta} \ar[d]^{\pr_2^{*}\widehat{\mathcal{B}}} \\ \pr_1^{*}f^{*}\mathcal{G}' \ar[r]_{(f,\hat f)^{*}\mathcal{D}'} & \pr_2^{*} \hat f^{*}\widehat{\mathcal{G}}'}
\end{equation*}
over $E\times_X \hat E$. Composition  is the composition of T-background 1-morphisms  together with the stacking 2-morphisms. 

\item
\label{def:corr:c}
A \emph{2-morphism} consists of 2-morphisms $\beta_1:(f,\mathcal{B}) \Rightarrow (f,\mathcal{B}')$ and $\beta_2:(\hat f,\widehat{\mathcal{B}}) \Rightarrow (\hat f,\widehat{\mathcal{B}}')$ of T-backgrounds such that
\begin{equation*}
\alxydim{@C=6em@R=3em}{\pr_1^{*}\mathcal{G} \ar@/_3pc/[d]_{\pr_1^{*}\mathcal{B}'}="4" \ar[d]^{\pr_1^{*}\mathcal{B}}="3" \ar[r]^{\mathcal{D}} & \pr_2^{*}\widehat{\mathcal{G}}\ar@{=>}[dl]|*+{\zeta}  \ar[d]^{\pr_2^{*}\widehat{\mathcal{B}}} \\ \pr_1^{*}f^{*}\mathcal{G}' \ar[r]_{(f,\hat f)^{*}\mathcal{D}'} & \pr_2^{*} \hat f^{*}\widehat{\mathcal{G}}' \ar@{=>}"3";"4"|{\pr_1^{*}\beta_1}}
=
\alxydim{@C=6em@R=3em}{\pr_1^{*}\mathcal{G} \ar[d]_{\pr_1^{*}\mathcal{B}'}  \ar[r]^{\mathcal{D}} & \pr_2^{*}\widehat{\mathcal{G}} \ar@/^3pc/[d]^{\pr_2^{*}\widehat{\mathcal{B}}}="5"  \ar@{=>}[dl]|*+{\zeta'} \ar[d]_{\pr_2^{*}\widehat{\mathcal{B}}'}="6" \\ \pr_1^{*}f^{*}\mathcal{G}' \ar[r]_{(f,\hat f)^{*}\mathcal{D}'} & \pr_2^{*} \hat f^{*}\widehat{\mathcal{G}}'
\ar@{=>}"5";"6"|{\pr_2^{*}\beta_2}
}
\end{equation*}
over $E \times_X \hat E$. Horizontal and vertical composition are those of 2-morphisms between T-backgrounds. 

\end{enumerate}
\end{definition}

Correspondences over $X$ form a bigroupoid $\cor(X)$, 
and the assignment $X \mapsto\ \cor(X)$  is a presheaf of bigroupoids over smooth manifolds. Since principal bundles and bundle gerbes form (2-)stacks over smooth manifolds, we have the following.

\begin{proposition}
\label{prop:corr2stack}
The presheaf $\cor$ is a 2-stack.  \qed
\end{proposition}

If $\mathcal{C}=((E,\mathcal{G}),(\hat E,\widehat{\mathcal{G}}),\mathcal{D})$ is a correspondence, then the two T-backgrounds $L(\mathcal{C}):=(E,\mathcal{G})$ and $R(\mathcal{C}) :=(\hat E,\widehat{\mathcal{G}})$ are called the \emph{left leg} and the \emph{right leg}, respectively. Projecting to left and right legs forms 2-functors $L,R: \cor(X) \to \tba(X)$. 
There is  another 2-functor $()^{\vee}: \cor(X)^{op}\to\cor(X)$, which takes $\mathcal{C}$ to the correspondence 
\begin{equation*}
\mathcal{C}^{\vee}:=((\hat E,\widehat{\mathcal{G}}),(E,\mathcal{G}),s^{*}\mathcal{D}^{-1})\text{,}
\end{equation*}
where $s:\hat E \times_X E \to E \times_X \hat E$ exchanges components. 

Next we want to define  a sub-2-stack $\tcor\subset \cor$ of \emph{T-duality} correspondences. 
We define the correspondence $\mathcal{T}_X:=((X \times \T^{n},\mathcal{I}),(X \times \T^{n},\mathcal{I}),\id_{\mathcal{I}}\otimes \pr^{*}\poi[n])$ over any smooth manifold $X$, where $\id_{\mathcal{I}} \otimes \pr^{*}\poi[n]$ denotes the action of the pullback of the $n$-fold Poincaré bundle $\poi[n]$ over $\T^{n} \times \T^{n}$ (see \cref{sec:poi}) on the identity 1-morphism $\id_{\mathcal{I}}$ between trivial bundle gerbes over $X \times \T^{n}\times \T^{n}$, see \cref{eq:bunact}. The correspondence $\mathcal{T}_X$ is the prototypical T-duality correspondence; general T-duality correspondences are obtained via gluing of $\mathcal{T}_X$ in a certain way. The gluing has to be performed along particular automorphisms of $\mathcal{T}_X$, which we describe next. 
We define for a pair $(a,\hat a)$ of smooth maps  $a,\hat a: X \to \R^{n}$  a 1-morphism 
\begin{equation*}
\mathcal{A}_{a,\hat a} :=((f_{a},\id_{\mathcal{I}}),(f_{\hat a},\id_{\mathcal{I}}),\zeta_{a,\hat a}) : \mathcal{T}_X \to \mathcal{T}_X \text{,}
\end{equation*}
where $f_a: X \times \T^{n} \to X \times \T^{n}$ is defined in \cref{eq:deffa} and $\zeta_{a,\hat a}$ is a 2-isomorphism  
\begin{equation*}
\zeta_{a,\hat a}: \id_{\mathcal{I}} \otimes \pr^{*}\poi[n] \Rightarrow (f_{a},f_{\hat a})^{*}(\id_{\mathcal{I}} \otimes \pr^{*}\poi[n]) = \id_{\mathcal{I}} \otimes \pr^{*}(r_{a,\hat a})^{*}\poi[n]\text{,}
\end{equation*}
which we define in the following way. We consider the point-wise scalar product of $a$ and $\hat a$ as a smooth map $a\hat a:X \to \ueins$, and regard it as an automorphism of the trivial $\ueins$-bundle over $X$. We act with this automorphism on the identity 1-morphism $\id_{\mathcal{I}}$, obtaining another 1-morphism that we denote again by $a\hat a: \mathcal{I} \to \mathcal{I}$. Next we consider the bundle morphism $\tilde R(a,\hat a): \poi[n] \to r_{a,\hat a}^{*}\poi[n]$ over $\T^{2n}$ from \cref{sec:poi}. Putting the pieces together, we define $\zeta_{a,\hat a} := a\hat a \otimes \tilde R(a,\hat a)$.
The composition of two 1-morphisms of the form $\mathcal{A}_{a,\hat a}$ is in general not of this form. We denote by $\mathcal{A}$ the class of automorphisms of $\mathcal{T}_X$ generated by the automorphisms $\mathcal{A}_{a,\hat a}$, i.e. $\mathcal{A}$ consists of all possible finite compositions of automorphisms of the form $\mathcal{A}_{a,\hat a}$ and their inverses. The class $\mathcal{A}$ is exactly the class of automorphisms along which we allow to glue. Unfortunately, at current time, no better description of the class $\mathcal{A}$ is known to us.

We let $\mathcal{T}(X)$ be the sub-bicategory of $\cor(X)$ with the single object $\mathcal{T}_X$,  all 1-morphisms of the class $\mathcal{A}$,  and all 2-morphisms. The assignment $X \mapsto \mathcal{T}(X)$ is a sub-presheaf of $\cor(X)$, and we let \begin{equation*}
\tcor := \overline{\mathcal{T}} \subset \cor
\end{equation*} 
be its closure under descent, which exists due to \cref{prop:corr2stack}.

\begin{definition}
\label{def:tdual}
A correspondence $\mathcal{C}$ over $X$ is called \emph{T-duality correspondence} if it is in $\tcor(X)$.
A T-background $(E,\mathcal{G})$ is called \emph{T-dualizable}, if there exists a T-duality correspondence $\mathcal{C}$ and a 1-isomorphism $L(\mathcal{C})\cong(E,\mathcal{G})$.
\end{definition}

\begin{remark}
We have the following consequences of this definition:
\begin{enumerate}[(a)]

\item 
\label{lem:F2fromTcorr}
The legs of a T-duality correspondence are $F_2$ T-backgrounds. 

\item
If $\mathcal{C}$  is a T-duality correspondence, then it obviously satisfies the so-called \emph{Poincaré condition} $\mathcal{P}(x)$ for every point $x\in X$, namely that there exists a 1-isomorphism $\mathcal{C}|_{\{x\}} \cong \mathcal{T}_{\{x\}}$ \cite{Bunke2006a}. Spelling out what this means, there exist $\T^{n}$-equivariant maps $t: \T^{n} \to E|_x$ and $\hat t: \T^{n} \to \hat E|_x$ together with trivializations $\mathcal{T}:t^{*}\mathcal{G} \to \mathcal{I}$ and $\widehat{\mathcal{T}}:\hat t^{*}\widehat{\mathcal{G}} \to \mathcal{I}$ such that the bundle gerbe isomorphism 
\begin{equation*}
\alxydim{@C=5em}{\mathcal{I} \ar[r]^{\pr_1^{*}\mathcal{T}^{-1}} & \pr_1^{*}t^{*}\mathcal{G} \ar[r]^{(t,\hat t)^{*}\mathcal{D}} & \pr_2^{*}\hat t^{*}\widehat{\mathcal{G}} \ar[r]^{\pr_2^{*}\widehat{\mathcal{T}}} & \mathcal{I}}
\end{equation*}
between trivial gerbes over $\T^{n} \times \T^{n}$ is given (up to a 2-isomorphism  over $\T^{n} \times \T^{n}$) by acting with the Poincaré bundle $\poi[n]$ on the identity $\id_{\mathcal{I}}$. Conversely, it follows from \cref{prop:TCorTriples} below that a correspondence that satisfies the Poincaré condition $\mathcal{P}(x)$ for all $x\in X$ is a T-duality correspondence. 

\item
The 2-functor $\mathcal{C} \mapsto \mathcal{C}^{\vee}$ restricts to T-duality correspondences, and thus induces a 2-functor $\tcor(X)^{op} \to \tcor(X)$. In particular, T-duality correspondence is  a symmetric relation on $\hc 0 {(\tba(X))}$. In general, it is neither reflexive nor transitive. 

\item
By construction, the assignment $X\mapsto \tcor(X)$ is a 2-stack; this is an important advantage of our definition over the  T-duality triples of \cite{Bunke2006a}.

\end{enumerate}
\end{remark}

\setsecnumdepth{2}

\subsection{A  2-group that represents T-duality correspondences}

\label{sec:TD2-group}
\label{sec:crossedmodule}
\label{re:tildeF}

For $x,y\in \R^{2n}$ we define the notation
\begin{equation*}
[x,y] := \sum_{i=1}^{n}x_{n+i}y_i\text{,}
\end{equation*}
which is a bilinear form on $\R^{2n}$ with matrix
\begin{equation*}
\tilde F :=\begin{pmatrix}
0 & 0 \\
E_n & 0 \\
\end{pmatrix}\text{.}
\end{equation*}
Often we  write  $x=a\oplus \hat a$ for $a,\hat a\in \R^{n}$, so that $[x_1,x_2]=[a_1 \oplus \hat a_1,a_2 \oplus \hat a_2] =\hat a_1a_2$.

We consider the \emph{categorical torus} associated to the bilinear form $[-,-]$ in the sense of Ganter \cite{Ganter2014}.
This is a strict Lie 2-group $\TD[n]$, defined as a  crossed module  $(G,H,t,\alpha)$ with  $G:=\R^{2n}$ and  $H:=\Z^{2n} \times \ueins$. The group homomorphism $t: H \to G$  is projection and inclusion,
$t(m,t) := m$.
The action $\alpha:G \times H \to H$ is given by
\begin{equation*}
\alpha(x,(z,t)) 
:=  (z,t-[x,z]) \text{.}
\end{equation*}
We have 
\begin{equation}
\label{eq:piTD}
\pi_0(\TD[n]) = \mathbb{T}^{2n}
\quand
\pi_1(\TD[n])=\ueins
\end{equation}
Since the projection $\obj{\TD[n]} \to \pi_0(\TD[n])$ is a surjective submersion, $\TD[n]$ is smoothly separable in the sense of \cite{Nikolausa}. 
Note that the induced action of $\R^{2n}$ on $U:= \mathrm{ker}(t)=\ueins$ is trivial. This shows that we have a \emph{central} extension
\begin{equation*}
B\ueins \to \TD[n] \to \idmorph{\T^{2n}}
\end{equation*}
in the sense of \cite{pries2}. We will often write just $\TD$ instead of $\TD[n]$.

We recall that central extensions of a Lie group $K$ by the Lie 2-group $B\ueins$ are classified by $\h^4(BK,\Z)$ \cite{pries2}.  In our case, $K=\T^{2n}$, we consider the \emph{Poincaré class}
\begin{equation*}
poi_n:=\sum_{i=1}^{n}\pr_i^{*}\mathrm{c_1} \cup \pr_{i+n}^{*}\mathrm{c}_1 \in \h^4(B\T^{2n} ,\Z)\text{,}
\end{equation*}
where $\mathrm{c}_1\in \h^2(B\ueins,\Z)$ is the universal first Chern class, and $\pr_i: \T^{2n} \to \ueins$ is the projection to the $i$-th factor.

\begin{proposition}
\label{prop:TDclass}
The central extension $\TD[n]$ is classified by $poi_n\in \h^4(B\T^{2n},\Z)$. 
\end{proposition}

\begin{proof}
Since $\h^4(B\T^{2n},\Z)$ has no torsion, it suffices to compare the images in real cohomology. Chern-Weil theory provides an algebra homomorphism $\mathrm{Sym}^{k}((\R^{2n})^{*}) \cong \h^{2k}(B\T^{2n},\R)$ for all $k>0$, and Ganter proves \cite[Theorem 4.1]{Ganter2014} that the class of $\TD[n]$ corresponds to 
\begin{equation*}
F:= \tilde F + \tilde F^{tr}=\begin{pmatrix}0 & E_n \\
E_n & 0 \\
\end{pmatrix}\in \mathrm{Sym}^2((\R^{2n})^{*})\text{.}
\end{equation*}
It is well-known that $\mathrm{c}_1\in \h^2(B\ueins,\R)$ corresponds to  $\id \in \mathrm{Sym}^1(\R^{*})=\R^{*}$, so that $\pr_{i}^{*}\mathrm{c}_1\in \mathrm{Sym}^{1}((\R^{2n})^{*})=(\R^{2n})^{*}$ is $\pr_i: \R^{2n} \to \R$. It remains to notice that
\begin{equation*}
\sum_{i=1}^{n}\pr_i \cdot \pr_{n+i} = F \text{,}
\end{equation*}
which is a straightforward calculation.
\end{proof}

\begin{remark}
The Lie 2-group $\TD[n]$ can be directly related to the $n$-fold Poincaré bundle, via multiplicative gerbes. Multiplicative gerbes over $\T^{2n}$ are also classified by $\h^4(B\T^{2n},\Z)$ and in fact equivalent as a bicategory to categorical central extensions of $\T^{2n}$ by $B\ueins$ \cite[Thm. 3.2.5]{Waldorf}. Ganter shows \cite[Prop. 2.4]{Ganter2014} that the multiplicative bundle gerbe associated to $\TD[n]$ is the trivial bundle gerbe over $\T^{2n}$, with the multiplicative structure given by a principal $\ueins$-bundle $L$ over $\T^{2n} \times \T^{2n}$ which descends from the surjective submersion $Z:=\R^{2n} \times \R^{2n} \to \T^{2n} \times \T^{2n}$ and the \v Cech 2-cocycle $\alpha:Z \times_{\T^{2n} \times \T^{2n}} Z \to \ueins$ defined by
\begin{equation*}
\alpha((x_1+z_1,x_2+z_2),(x_1,x_2)) = [x_1,z_2]\text{.}
\end{equation*}
In order to identify this bundle, we  consider $Z':=\R^{2n} \to \T^{2n}$ and the commutative diagram
\begin{equation*}
\alxydim{@R=3em}{Z \ar[r]^{\pr_{23}} \ar[d] & Z' \ar[d] \\ \T^{2n} \times \T^{2n} \ar[r]_-{\pr_{23}} & \T^{2n}\text{,}}
\end{equation*}
where $\pr_{23}((a_1 \oplus \hat a_1),(a_2\oplus \hat a_2)) := (\hat a_1\oplus a_2)$. We consider for $Z' \to \T^{2n}$ the \v Cech 2-cocycle 
$\alpha': Z' \times_{\T^{2n}} Z' \to \ueins$ defined by $\alpha'(x+z,x):= [z,x]$. Comparing with \cref{eq:curvpoin}, this is the cocycle of the $n$-fold Poincaré bundle $\poi[n]$.
It is straightforward to check that $\alpha=(\pr_{23} \times \pr_{23})^{*}\alpha'$. 
In other words, $L\cong \pr_{23}^{*}\poi[n]$. Thus,  the Lie 2-group $\TD[n]$ corresponds to the trivial gerbe over $\T^{2n}$ with multiplicative structure given by $\pr_{23}^{*}\poi[n]$ over $\T^{2n} \times \T^{2n}$. 
\end{remark}

\begin{remark}
A $\TD[n]$-cocycle with respect to an open cover $\{U_i\}_{i\in I}$ is a tuple $(a,\hat a,m,\hat m,t)$ consisting of numbers $m_{ijk},\hat m_{ijk} \in \Z^{n}$  and smooth maps
\begin{equation*}
a_{ij},\hat a_{ij}: U_i \cap U_j \to \R^{n}
\quomma
t_{ijk}:U_i \cap U_j \cap U_k \to \ueins
\end{equation*}
satisfying the following cocycle conditions:
\begin{align*}
a_{ik} &= m_{ijk} + a_{jk}+a_{ij}
\\
\hat a_{ik} &= \hat m_{ijk} + \hat a_{jk}+\hat a_{ij}
\\
m_{ikl}+m_{ijk} &= m_{ijl}+m_{jkl}
\\
\hat m_{ikl}+\hat m_{ijk} &= \hat m_{ijl}+\hat m_{jkl}
\\
t_{ikl}+t_{ijk}-m_{ijk}\hat a_{kl} &= t_{ijl}+t_{jkl}
\end{align*}
Two $\TD$-cocycles $(a,\hat a,m,\hat m,t)$ and $(a',\hat a',m',\hat m',t')$ are equivalent if there exist numbers $z_{ij},\hat z_{ij} \in \Z^{n}$, smooth maps  $p_i,\hat p_i:U_i \to \R^{n}$ and $e_{ij}: U_i \cap U_j \to \ueins$ such that
\begin{align*}
a_{ij}' + p_i &= z_{ij}+p_j+a_{ij}
\\
\hat a_{ij}' + \hat p_i &= \hat z_{ij}+\hat p_j+\hat a_{ij}
\\
m_{ijk}'+z_{ij}+z_{jk} &= z_{ik} +m_{ijk}
\\
\hat m_{ijk}'+\hat z_{ij}+\hat z_{jk} &= \hat z_{ik}+\hat m_{ijk}
\\
t_{ijk}'+e_{ij}-\hat a'_{jk}z_{ij}+e_{jk} &= e_{ik}+t_{ijk}-\hat p_km_{ijk}
\end{align*}
\end{remark}

\label{sec:tdualityandgerbes}

We consider the presheaf $\sheaf{B\TD}$ of smooth $B\TD$-valued functions. Over a smooth manifold $X$ the bicategory $\sheaf{B\TD}(X)$ is the following:
\begin{itemize}

\item 
It has one object.

\item
The 1-morphisms are pairs $(a,\hat a)$ of smooth maps $a,\hat a:X \to \R^{n}$; the composition is pointwise addition.

\item
A 2-morphism from $(a,\hat a)$ to $(a',\hat a')$ is a triple $(t,m,\hat m)$ consisting of $m,\hat m\in \Z^{n}$ and a smooth map $t:X \to \ueins$, such that $a'=a+m$ and $\hat a'=\hat a+\hat m$. Vertical composition is (pointwise) addition, and horizontal composition is
\begin{equation*}
\alxydim{@C=2.5cm}{\ast \ar@/^2pc/[r]^{(a_1,\hat a_1)}="1" \ar@/_2pc/[r]_{(a_1+m_1,\hat a_1+\hat m_1)}="2" \ar@{=>}"1";"2"|*+{(t_1,m_1,\hat m_1)} & \ast \ar@/^2pc/[r]^{(a_2,\hat a_2)}="1" \ar@/_2pc/[r]_{(a_2+m_2,\hat a_2+\hat m_2)}="2" \ar@{=>}"1";"2"|*+{(t_2,m_2,\hat m_2)} & \ast}
=
\alxydim{@C=7cm}{\ast \ar@/^3pc/[r]^{(a_1+a_2,\hat a_1+\hat a_2)}="1" \ar@/_3pc/[r]_{(a_1+a_2+m_1+m_2,\hat a_1+\hat a_2+\hat m_1+\hat m_2)}="2" \ar@{=>}"1";"2"|*+{(t_2+t_1-m_1\hat a_2,m_1+m_2,\hat m_1+\hat m_2)} & \ast\text{.}}
\end{equation*}

\end{itemize}
We define a 2-functor
\begin{equation*}
C:\sheaf{B\TD}(X) \to \mathcal{T}(X)\text{,}
\end{equation*}
where $\mathcal{T}$ is the presheaf defined in \cref{sec:geotduality}.
It  sends the single object to the correspondence $\mathcal{T}_X$.  A 1-morphism $(a,\hat a)$ is sent to the 1-morphism $\mathcal{A}_{a,\hat a}$. A 2-morphism $(t,m,\hat m):(a,\hat a) \Rightarrow (a+m,\hat a+\hat m)$ is sent to the 2-morphism $(\alpha_{t,m,\hat m},\beta_{t,m,\hat m})$ whose 2-morphisms $\alpha_{t,m,\hat m}:\id_{\mathcal{I}} \Rightarrow \id_{\mathcal{I}}$ and $\beta_{t,m,\hat m}:\id_{\mathcal{I}} \Rightarrow \id_{\mathcal{I}}$ over $X \times \T^{n}$ are given by acting with the $\ueins$-valued functions $\alpha_{t,m,\hat m}(x,c):=-t(x)-\hat mc$ and $\beta_{t,m,\hat m}(x,c):=-  t(x)-m(c+\hat a(x))$, considered as automorphisms of the trivial $\ueins$-bundle, on the identity 2-morphism $X \times \T^{n}$. 
The condition for 2-morphisms is:
\begin{equation*}
\alxydim{@C=6em@R=3em}{\mathcal{I} \ar@/_5pc/[d]_{\id_{\mathcal{I}}}="1" \ar[d]^{\id_{\mathcal{I}}}="2" \ar@{<=}"1";"2"|{\pr_1^{*}\alpha_{t,m,\hat m}} \ar[r]^{\id_{\mathcal{I}} \otimes \pr^{*}\poi[n]} & \mathcal{I} \ar@{=>}[dl]|*+{\zeta_{a,\hat a}} \ar[d]^{\id_{\mathcal{I}}} \\ \mathcal{I} \ar[r]_{\id_{\mathcal{I}} \otimes \pr^{*}r_{a,\hat a}^{*}\poi[n]} & \mathcal{I}}
=\alxydim{@C=6em@R=3em}{\mathcal{I} \ar[d]_{\id_{\mathcal{I}}} \ar[r]^{\id_{\mathcal{I}} \otimes \pr^{*}\poi[n]} & \mathcal{I} \ar@/^5pc/[d]^{\id_{\mathcal{I}}}="2" \ar@{=>}[dl]|*+{\zeta_{a+m,\hat a+\hat m}} \ar[d]_{\id_{\mathcal{I}}}="1" \\ \mathcal{I} \ar[r]_{\id_{\mathcal{I}} \otimes \pr^{*}r_{a,\hat a}^{*}\poi[n]} & \mathcal{I} \ar@{<=}"1";"2"|{\pr_2^{*}\beta_{t,m,\hat m}}}
\end{equation*}
over $(X \times \T^{n}) \times_X (X \times \T^{n})$. This boils down to an identity
\begin{equation}
\label{eq:cond2morphTD}
\pr_1^{*}\alpha_{t,m,\hat m} \cdot a\hat a\cdot  \tilde R(a,\hat a)= (a+m)(\hat a+\hat m) \cdot \tilde R(a+m,\hat a+\hat m)\cdot \pr_2^{*}\beta_{t,m,\hat m}\text{.}
\end{equation}
It is straightforward to check this identity using \cref{ex:nfold}.
In order to complete the definition of the 2-functor $C$, we need to provide an associator and to check the axioms. The vertical composition is respected since the  following equations hold:
\begin{align*}
\alpha_{t_2,m_2,\hat m_2}+\alpha_{t_1,m_1,\hat m_1}&=\alpha_{t_2+t_1,m_2+m_1,\hat m_2+\hat m_1}
\\
\beta_{t_2,m_2,\hat m_2}+\beta_{t_1,m_1,\hat m_1}&=\beta_{t_2+t_1,m_2+m_1,\hat m_2+\hat m_1}
\end{align*}
The associator
\begin{multline*}
c_{(a_1,\hat a_1),(a_2,\hat a_2)}: ((f_{a_2},\id_{\mathcal{I}}),(f_{\hat a_2},\id_{\mathcal{I}}),\zeta_{a_2,\hat a_2}) \circ ((f_{a_1},\id_{\mathcal{I}}),(f_{\hat a_1},\id_{\mathcal{I}}),\zeta_{a_1,\hat a_1})\\\Rightarrow ((f_{a_1+a_2},\id_{\mathcal{I}}),(f_{\hat a_1+\hat a_2},\id_{\mathcal{I}}),\zeta_{a_1+a_2,\hat a_1+\hat a_2})
\end{multline*}
is defined as follows. The horizontal composition on the left is 
\begin{equation*}
((f_{a_1+a_2},\id_{\mathcal{I}}),(f_{\hat a_1+\hat a_2},\id_{\mathcal{I}}),(a_2\hat a_2 + a_1\hat a_1+\hat a_1a_2) \otimes \pr^{*}\tilde R(a_2+a_1,\hat a_2+\hat a_1))\text{.}
\end{equation*}
We set $c_{(a_1,\hat a_1),(a_2,\hat a_2)} := (\id_{\id_{\mathcal{I}}} \otimes \hat a_2a_1,\id_{\id_{\mathcal{I}}})$.
The condition for 2-morphisms is
\begin{equation*}
(\hat a_2a_1 +a_2\hat a_2 + a_1\hat a_1+\hat a_1a_2)\otimes\pr^{*}\tilde R(a_2+a_1,\hat a_2+\hat a_1) = (a_1+a_2)(\hat a_1+\hat a_2) \otimes \pr^{*}\tilde R(a_2+a_1,\hat a_2+\hat a_1)\text{,}
\end{equation*}
and obviously satisfied.

\begin{lemma}
\label{lem:functorC}
The 2-functor $C: \sheaf{B\TD}(X) \to \mathcal{T}(X)$ is an equivalence.
\end{lemma}

\begin{proof}
Both categories have a single object. On the level of 1-morphisms, it is obviously essentially surjective. Now consider two 1-morphisms $(a,\hat a)$ and $(a',\hat a')$. The set of 2-morphisms between $(a,\hat a)$ and $(a',\hat a')$ in $\sheaf{B\TD}(X)$ are triples $(t,m,\hat m)$ with $t: X \to \ueins$ and $m,\hat m\in \Z^{n}$ such that $a'=a+m$ and $\hat a'=\hat a+\hat m$. The set of 2-morphisms between $C(a,\hat a)=((f_a,\id),(f_{\hat a},\id),\zeta_{a,\hat a})$ and $C(a',\hat a')=((f_{a+m},\id),(f_{\hat a+\hat m},\id),\zeta_{a+m,\hat a+\hat m})$ consists  of pairs $(\alpha,\beta)$ of smooth maps $\alpha,\beta:X \times \T^{n} \to \ueins$ such that \cref{eq:cond2morphTD} is satisfied, \begin{equation}
\label{eq:cond2morphTDagain}
\alpha(x,c)  = mb(x) -c\hat m+md+\beta(x,d)\text{.}
\end{equation}
We have to show that the map $(t,m,\hat m)\mapsto (\alpha_{t,m,\hat m},\beta_{t,m,\hat m})$ is a bijection, where
\begin{equation*}
\alpha_{t,m,\hat m}(x,c):=-t(x)-\hat mc
\quand
\beta_{t,m,\hat m}(x,c):=-  t(x)-m(c+\hat a(x))\text{.}
\end{equation*}
Since $m$ and $\hat m$ are uniquely determined, an equality $\alpha_{t,m,\hat m}=\alpha_{t',m,\hat m}$ implies already $t=t'$. Thus, our map is injective. Conversely, given $(\alpha,\beta)$, we define $t(x) := -mb(x)-\beta(x,0)$. Then we get from \cref{eq:cond2morphTDagain}
\begin{equation*}
\alpha_{t,m,\hat m}(x,c)=mb(x)+\beta(x,0)-\hat mc=\alpha(x,c)
\end{equation*}
and
\begin{equation*}
\beta_{t,m,\hat m}(x,c)=mb(x)+\beta(x,0)-mc-mb(x)= \alpha(x,0)-mb(x)-mc=\beta(x,c)\text{.}
\end{equation*}
This shows that our map is also surjective.
\end{proof}

Since $\tcor$ was the 2-stackification of $\mathcal{T}$, we obtain from \cref{lem:functorC}: 

\begin{proposition}
\label{prop:BTDCorr}
The 2-functor $C$ induces an  isomorphism
$\sheaf{B\TD}^{+} \cong \tcor$.
\end{proposition}

\subsection{T-duality  triples}

In \cite{Bunke2006a} a category of \quot{T-duality triples} was defined. In this section we relate that definition to our notion of T-duality correspondences and the Lie 2-group $\TD$.

The paper \cite{Bunke2006a} is  written with respect to an arbitrary 1-categorical model  for $\ueins$-gerbes. In order to compare this with our definitions, we choose the 1-truncation $\hc 1 (\ugrb X)$ as a model. Then, we have the following definitions:
\begin{enumerate}[(a)]

\item 
A \emph{triple} over $X$ consists of $F_2$ T-backgrounds $(E,\mathcal{G})$ and $(\hat E,\widehat{\mathcal{G}})$ over  $X$, and of a 2-isomorphism class of  1-isomorphism $[\mathcal{D}]: \pr_1^{*}\mathcal{G} \to \pr_2^{*}\widehat{\mathcal{G}}$ over $E \times_X \hat E$. 

\item
A \emph{morphism} between triples $((E,\mathcal{G}),(\hat E,\widehat{\mathcal{G}}),[\mathcal{D}])$ and $((E',\mathcal{G}'),(\hat E',\widehat{\mathcal{G}}'),[\mathcal{D}'])$ consists of two 1-morphisms  $(f,[\mathcal{B}]):(E,\mathcal{G}) \to (E',\mathcal{G}')$ and $(\hat f,[\widehat{\mathcal{B}}]):(\hat E,\widehat{\mathcal{G}}) \to (\hat E',\widehat{\mathcal{G}}')$  of T-backgrounds, where the gerbe 1-morphisms are taken up to 2-isomorphisms, such that the diagram
\begin{equation*}
\alxydim{@C=6em@R=3em}{\pr_1^{*}\mathcal{G} \ar[d]_{\pr_1^{*}[\mathcal{B}]} \ar[r]^{[\mathcal{D}]} & \pr_2^{*}\widehat{\mathcal{G}}  \ar[d]^{\pr_2^{*}[\widehat{\mathcal{B}}]} \\ \pr_1^{*}f^{*}\mathcal{G}' \ar[r]_{(f,\hat f)^{*}[\mathcal{D}']} & \pr_2^{*} \hat f^{*}\widehat{\mathcal{G}}'}
\end{equation*}
is commutative in the 1-truncation $\hc 1 (\ugrb- (E \times_X \hat E))$, see \cite[Def. 4.5]{Bunke2006a}.

\item
A triple is called \emph{T-duality triple}, if the two T-backgrounds are $F_2$, and if its restriction to any point $x\in X$ is isomorphic to $\mathcal{T}_{\{x\}}$. The category $\trip(X)$ of T-duality triples is the full subcategory on these.

\end{enumerate}

\begin{lemma}
\label{th:classspace}
The geometric realization $|B\TD|$ is a classifying space for T-duality triples.
\end{lemma}

\begin{proof}
By \cite[Thm. 2.14]{Bunke2006a}, the classifying space for T-duality triples is the homotopy fibre of the map $B\T^{2n} \to K(\Z,4)$ whose homotopy class is $poi_n\in \h^4(B\T^{2n},\Z)$. On the other hand, a central extension $B\ueins \to \Gamma \to \idmorph{K}$ induces a fibre sequence of classifying spaces, so that $|B\Gamma|$ is the homotopy fibre of the map $BK \to |BB\ueins| \simeq K(\Z,4)$. The classification of central extensions in \cite{pries2} exhibits its class in $\h^4(BK,\Z)$ as the homotopy class of 
this map. Now, \cref{prop:TDclass} shows the claim. 
\end{proof}

There is an obvious functor
\begin{equation}
\label{eq:functortcorrtriples}
\hc 1 (\tcor(X)) \to \trip(X)
\end{equation}
obtained by reducing every 1-morphism to its 2-isomorphism class. On the level of objects, this is well-defined because of \cref{lem:F2fromTcorr}. On the level of morphisms,  \emph{every} 1-morphism in $\cor(X)$ yields a morphism of triples, and it is straightforward to see that 2-isomorphism 1-morphisms result in the same morphism. In general, the functor \cref{eq:functortcorrtriples} is not an equivalence, but  \cref{th:classspace,prop:BTDCorr}
imply:

\begin{proposition}
\label{prop:TCorTriples}
The functor of \cref{eq:functortcorrtriples} induces a bijection 
\begin{equation*}
\hc 0 (\tcor(X)) \cong \hc 0 (\trip(X))\text{.}
\end{equation*}
In particular, the  notion of T-dualizability given in \cref{def:tdual} coincides with the one of  \cite{Bunke2006a}.
\qed
\end{proposition}

Thus, all results of \cite{Bunke2006a} can be transferred to our setting. For instance, we have the following important result \cite[Theorem 2.23]{Bunke2006a}:

\begin{theorem}
\label{th:bsr}
A T-background is T-dualizable if and only if it is $F_2$.
\end{theorem}

A cohomological formulation of this theorem is presented below as \cref{th:tdualbij}.
Even later we extend this result to half-geometric T-duality correspondences and $F_1$ T-backgrounds, see \cref{th:main}. \Cref{th:bsr} can then be deduced as a special case.

\subsection{Implementations of legs and the flip}

\label{sec:legsandflip}

We define a crossed intertwiner
\begin{equation*}
\flip: \TD[n] \to \TD[n]
\end{equation*}
as the triple $(\phi,f,\eta)$ with 
\begin{equation*}
\phi(a \oplus \hat a) :=\hat a \oplus a
\quomma
f(m \oplus \hat m,t) := (\hat m \oplus m,t)
\quand
\eta(x,x') :=(0,[x,x'])\text{.}
\end{equation*}

\begin{proposition}
The crossed intertwiner $\flip$  represents the  2-functor $\mathcal{C} \mapsto \mathcal{C}^{\vee}$, in the sense that the diagram
\begin{equation*}
\alxydim{@R=3em}{B\TD(X) \ar[r]^-{C} \ar[d]_{B(\flip)} & \tcor(X) \ar[d]^{()^{\vee}} \\ B\TD(X) \ar[r]_-{C} & \tcor(X)}
\end{equation*}
is strictly commutative.
\qed
\end{proposition}

\begin{remark}
We remark  that $\flip$ is not strictly involutive:  $\flip^2=\flip \circ \flip$ is the crossed intertwiner $(\id,\id,\tilde\eta)$ with
$\tilde\eta(a_1\oplus \hat a_1,a_2 \oplus \hat a)= a_1 \hat a_2+ \hat a_1 a_2$.
As an element of the automorphism 2-group of $\TD$ discussed in an upcoming paper, it is only coherently isomorphic to the identity. 
\end{remark}

We define homomorphisms
\begin{equation*}
\lele: \TD[n] \to \TBFF[n]
\quand
\rele: \TD[n] \to \TBFF[n]
\end{equation*}
that will produce the left leg and right leg of a T-duality correspondence. 
The crossed intertwiner  $\rele$ is strict; it consists of the group homomorphisms $\phi(a \oplus \hat a):= \hat a$ and $f(m \oplus \hat m,t)(c) := t+mc$, whereas $\eta:=0$. 
We define  $\lele := \rele \circ \flip$. This gives 
\begin{equation*}
\phi(a \oplus \hat a) = a
\quomma
f(m \oplus \hat m,t)(c) =t+\hat mc
\quand
\eta(a \oplus \hat a,a' \oplus \hat a')(c) =\hat aa'\text{.}
\end{equation*} 
By inspection of the definitions, we see the following:

\begin{proposition}
The crossed intertwiners $\lele$ and $\rele$ represent the left leg and the right leg projections of a T-duality correspondence, in the sense that the diagrams
\begin{equation}
\label{eq:tdualitylelerele}
\alxydim{@R=3em}{B\TD(X) \ar[r]^-{C} \ar[d]_{B(\lele)} & \tcor(X) \ar[d]^{L} \\ B\TBFF(X) \ar[r]_-{F_2} & \tbaFF(X)}
\quand
\alxydim{@R=3em}{B\TD(X) \ar[r]^-{C} \ar[d]_{B(\rele)} & \tcor(X) \ar[d]^{R} \\ B\TBFF(X) \ar[r]_-{F_2} & \tbaFF(X)}
\end{equation}
are strictly commutative.
\qed
\end{proposition}

\begin{remark}
\label{re:leleR}
The homomorphisms $\lele$ and $\rele$ can be lifted into the bigger 2-group $\TBFFR[n]$. Indeed, $\rele$ can be lifted to a strict intertwiner $\releR: \TD[n] \to \TBFFR[n]$ given by $\phi(a \oplus \hat a):= \hat a$ and $f(m \oplus \hat m,t) := (\tau_{t,m},\hat m)$, where $\tau_{t,m}(c)= t+mc$. The left leg $\leleR: \TD[n] \to \TBFFR[n]$ is again defined by $\leleR := \releR \circ \flip$, resulting in
\begin{equation}
\label{eq:lele}
\phi(a \oplus \hat a)=a
\quomma
f(m \oplus \hat m,t) = (\tau_{t,\hat m},m)
\quand
\eta(a \oplus \hat a,a' \oplus \hat a')(c) =\hat aa'\text{.}
\end{equation}
\end{remark}

We identify the induced map
\begin{equation*}
\leleR_{*}:\h^1(X,\TD) \to \h^1(X,\TBFFR)
\end{equation*} 
on the level of cocycles. Suppose $(a , \hat a,m, \hat m,t)$ is a $\TD$-cocycle with respect to an open cover $\{U_i\}_{i\in I}$. Applying the general construction of \cref{sec:CI:induced} to \cref{eq:lele} we collect the $a_{ij}:U_i \cap U_j \to \R^{n}$ and $m_{ijk}\in \Z^{n}$ as they are, and put
\begin{equation}
\label{eq:td:lele:tau}
\tau_{ijk}(x)(a):=t_{ijk}(x)+(a-a_{ik}(x))\hat m_{ijk}-a_{ij}(x)\hat a_{jk}(x)\text{.}
\end{equation}
We have the following result:

\begin{theorem}
\label{th:tdualbij}
The left leg projection $\leleR_{*}:\h^1(X,\TD) \to \h^1(X,\TBFFR)$ is a bijection. 
\end{theorem}

\cref{th:tdualbij} can be proved by reducing the proof of our 
main result, \cref{th:main}, to the case that all occurring $\mathfrak{so}(n,\Z)$-matrices are zero. Since we never use \cref{th:tdualbij} directly, we will not write this out. 

\begin{remark}
\label{re:equivclass}
\cref{th:tdualbij} implies that the classifying spaces $|B\TD|$ and $|B\TBFFR|$ are equivalent. 
Indeed, we have $[X,|B \TBFFR|]\cong\h^1(X,\TBFFR) \cong \h^1(X,\TD)\cong[X,|B\TD|]$ for all $X$; using \cref{th:tdualbij} and the bijection of \cref{re:classcoho}. Thus, the Yoneda lemma implies the equivalence. An alternative proof of  \cref{th:tdualbij} would be to prove the equivalence between the classifying spaces  $|B\TD|$ and $|B\TBFFR|$ directly; this is the strategy pursued in \cite{Bunke2006a}.
We remark that the Lie 2-groups $\TD$ and $\TBFFR$ are \emph{not}  isomorphic, since they have different homotopy types, see \cref{eq:piTBFF,eq:piTD}. \end{remark}

\begin{remark}
The results of this section are related to the results of \cite{Bunke2006a}) in the following way. The bare existence of the map $\leleR_{*}$ shows the \quot{only if}-part of \cref{th:bsr} (corresponding to \cite[Theorem 2.23]{Bunke2006a}).
Surjectivity in \cref{th:tdualbij} provides the \quot{if}-part. Another result  \cite[Theorem 2.24 (2)]{Bunke2006a} is that two T-duality triples with isomorphic left legs are related under the action of $\mathfrak{so}(n,\Z)$ on triples defined in \cite{Bunke2006a}; in particular, they do not have to be isomorphic. This is not a contradiction to the injectivity of \cref{th:tdualbij} because in \cite{Bunke2006a} the left leg projection is the composite
\begin{equation*}
\alxydim{}{\h^1(X,\TD) \ar[r]^-{\leleR_{*}} &  \h^1(X,\TBFFR) \ar[r] & \h^1(X,\TBF)\text{,}}
\end{equation*}
of which the second map is not injective (\cref{re:F2ToF1}). 
\end{remark}

\begin{remark}
We obtain a well-defined, canonical map 
\begin{equation*}
\alxydim{@C=2cm}{\h^1(X,\TBFFR) \ar[r]^-{(\leleR_{*})^{-1}} & \h^1(X,\TD) \ar[r]^-{\releR_{*}} & \h^1(X,\TBFFR)\text{.}}
\end{equation*}
The existence of this map might be confusing, as it looks like a \quot{T-duality transformation} for arbitrary $F_2$ T-backgrounds, whereas the results of \cite{Bunke2006a} imply that such a transformation exist only for $n=1$ (since then $\mathfrak{so}(n,\Z)=\{0\}$). The point is, again, the non-injectivity of the map $\h^1(X,\TBFFR)\to\ \h^1(X,\TBF)$, which prevents to define a T-duality transformation on $\h^1(X,\TBF)$.
\end{remark}

\section{Half-geometric T-duality}

\label{sec:TFgeo}

In this section, we introduce and study the central objects of this article: half-geometric T-duality correspondences. 

\subsection{Half-geometric T-duality correspondences}

\label{sec:sonZactsTD}

We define an action of $\mathfrak{so}(n,\Z)$ on the 2-group $\TD[n]$ by crossed intertwiners in the sense of \cref{def:action}. A matrix $B\in \mathfrak{so}(n,\Z)$ acts by a crossed intertwiner 
\begin{equation*}
F_{e^{B}}=(\phi_{e^{B}},f_{e^{B}},\eta_{e^{B}}):\TD[n] \to \TD[n]\text{,}
\end{equation*}
which we define by: 
\begin{align*}
\phi_{e^{B}}(a \oplus \hat a) &:= (a \oplus (Ba+\hat a))
\\
f_{e^{B}}(m \oplus \hat m,t) &:=(m\oplus (Bm+\hat m),t)
\\
\eta_{e^{B}}(a\oplus \hat a,b\oplus \hat b) &:=\lbraket{a}{B}{b}\text{.}
\end{align*}

\begin{remark}
The notation $e^{B}$ is used in order to distinguish the $\mathfrak{so}(n,\Z)$-action on $\TD$ from the $\mathfrak{so}(n,\Z)$-action on $\TBFFR$, which was introduced in \cref{sec:TBFFR} in terms of crossed intertwiners $(\phi_B,f_B,\eta_B): \TBFFR\to \TBFFR$. In fact,  $e^{B}$ is our notation for the inclusion
\begin{equation*}
\mathfrak{so}(n,\Z) \to \mathrm{O}(n,n,\Z):B \mapsto e^{B} := \begin{pmatrix}
1 & 0 \\
B & 1 \\
\end{pmatrix}\text{.}
\end{equation*}  
We show in a separate paper that the automorphism 2-group of $\TD$ is a (non-central, non-splitting) extension of the split-orthogonal group $\mathrm{O}(n,n,\Z)\subset \mathrm{GL}(2n,\Z)$. This extension splits over the subgroup $\mathfrak{so}(n,\Z)$. Above action is induced from the action by automorphism via the inclusion $B \mapsto e^{B}$. 
\end{remark}

We define the half-geometric T-duality 2-group as the semi-direct product
\begin{equation*}
\TFgeo := \TD[n] \ltimes \mathfrak{so}(n,\Z)\text{,}
\end{equation*}
see \cref{sec:semidirect}.
It has the following invariants:
\begin{equation}
\label{eq:piTFgeo}
\pi_0(\TFgeo) = \T^{2n} \ltimes\mathfrak{so}(n,\Z)
\quand
\pi_1(\TFgeo) = \ueins\text{,} 
\end{equation}
where $\mathfrak{so}(n,\Z)$ acts on $\T^{2n}$ by multiplication with $e^{B}\subset \mathrm{GL}(2n,\Z)$. 

The Lie 2-group $\TFgeo$ represents a 2-stack over smooth manifolds that we call the 2-stack of \emph{half-geometric T-duality correspondences}. The reader is free to  pick any  model for this 2-stack. The simplest possibility is to take $(\sheaf{\TFgeo})^{+}$, which leads to a cocycle description carried out below in \cref{sec:coctfgeo}. Other possibilities are to take non-abelian $\TFgeo$-bundle gerbes \cite{aschieri,Nikolaus}, principal $\TFgeo$-2-bundles \cite{wockel1,pries2,Waldorf2016}, or principal $\infty$-bundles \cite{nikolausc}.

\begin{remark}
\label{sec:coctfgeo}
According to \cref{sec:semidirect}, a $\TFgeo$-cocycle with respect to an open cover $\{U_i\}_{i\in I}$ is a tuple $(B,a,\hat a,m,\hat m,t)$ consisting of matrices $B_{ij}\in \mathfrak{so}(n,\Z)$, numbers $m_{ijk},\hat m_{ijk}\in \Z^{n}$, and
smooth maps
\begin{equation*}
a_{ij},\hat a_{ij}: U_i \cap U_j \to \R^{n}
\quand
t_{ijk}:U_i \cap U_j \cap U_k \to \ueins
\end{equation*}
satisfying
\begin{align}
\label{coc:TFgeo:1}
B_{ik} &= B_{jk}+B_{ij}
\\
\label{coc:TFgeo:2}
a_{ik} &= m_{ijk} + a_{jk}+a_{ij}
\\
\label{coc:TFgeo:3}
\hat a_{ik} &= \hat m_{ijk} + \hat a_{jk}+\hat a_{ij}+B_{jk}a_{ij}
\\
\label{coc:TFgeo:4}
t_{ikl}+t_{ijk}
&=\hat a_{kl}m_{ijk}+\lbraket{m_{ijk}}{B_{kl}}{a_{ik}}+\lbraket{a_{jk}}{B_{kl}}{a_{ij}}+ t_{ijl}+ t_{jkl}\text{.}
\end{align} 
Note that \cref{coc:TFgeo:3} implies
\begin{equation}
\label{re:coc:TFgeo:5}
\hat m_{ikl}+\hat m_{ijk}+B_{kl}m_{ijk}=\hat m_{ijl}+\hat m_{jkl}\text{.}
\end{equation}
Also note that for $B_{ij}=0$ we obtain precisely the cocycles for $\TD$. 
Two $\TFgeo$-cocycles $(B,a,\hat a,m,\hat m,t)$ and $(B',a',\hat a',m',\hat m',t')$ are equivalent if there exist $C_i \in \mathfrak{so}(n,\Z)$, numbers $z_{ij},\hat z_{ij}\in \Z^{n}$, and smooth maps $p_i,\hat p_i: U_i \to \R^{n}$, $e_{ij}:U_i \cap U_j \to \ueins$ such that: 
\begin{align}
\label{coc:TFgeo:6}
C_j + B_{ij} &=B'_{ij}+ C_i
\\
\label{coc:TFgeo:7}
z_{ij} + p_j+ a_{ij}&=a_{ij}'+p_i
\\
\label{coc:TFgeo:8}
\hat z_{ij} + \hat p_j+C_ja_{ij}+\hat a_{ij}&=\hat a_{ij}'+B'_{ij}p_i+\hat p_i
\\
\nonumber
&\mquad\mquad\mquad\mquad\mquad t_{ijk}'+ \lbraket{a_{ij}'}{B_{jk}'}{p_i}-\lbraket{z_{ij}}{B_{jk}'}{p_j+a_{ij}}
+ e_{ij}-\hat a'_{jk}z_{ij}- \lbraket{p_j}{B_{jk}'}{a_{ij}}+ e_{jk}
\\&=e_{ik} - \lbraket{m_{ijk}}{C_k}{a_{ik}}
+t_{ijk}-\hat p_k m_{ijk} - \lbraket{a_{jk}}{C_k}{a_{ij}}
\label{coc:TFgeo:9}
\end{align}
Note that 
\cref{coc:TFgeo:7} implies
\begin{equation}
\label{re:coc:TFgeo:7}
m'_{ijk}+z_{ij}+z_{jk}=z_{ik}+m_{ijk}
\end{equation}
Also note that
 \cref{coc:TFgeo:8} implies
\begin{equation}
\label{re:coc:TFgeo:8}
\hat m'_{ijk}+\hat z_{ij}+\hat z_{jk} = \hat m_{ijk}+\hat z_{ik}+B'_{jk}z_{ij}+C_k m_{ijk}\text{.}
\end{equation}
\end{remark}

\subsection{The left leg of a half-geometric correspondence}

We recall from \cref{re:leleR} that the left leg of a T-duality correspondence  is represented by a crossed intertwiner
\begin{equation*}
\leleR =(\phi,f,\eta): \TD[n] \to \TBFFR\text{,}
\end{equation*} 
where  $\phi(a \oplus \hat a)=a$ and $f(m \oplus \hat m,t) = (\tau_{t,\hat m},m)$ and $\eta(a \oplus \hat a,a' \oplus \hat a')(c) =\hat a a'$.   
We have actions of $\mathfrak{so}(n,\Z)$ on both 2-groups by crossed intertwiners $(\id_{\R^{n}},f_B,\eta_B)$ and $(\phi_{e^{B}},f_{e^{B}},\eta_{e^{B}})$, respectively (defined in \cref{sec:sonZactsTD,sec:TBFFR}). We have the following key lemma.

\begin{lemma}
$\leleR$ is strictly $\mathfrak{so}(n,\Z)$-equivariant in the sense of \cref{def:equivint}.
\end{lemma}

\begin{proof}
We check for each $B \in \mathfrak{so}(n,\Z)$ the three conditions for equivariance  listed in \cref{re:equivintcond}:
\begin{enumerate}[(a)]

\item 
Equivariance of $\phi$:
$(\id_{\R^{n}} \circ \phi)(a \oplus \hat a)=a=\phi(a\oplus (Ba+\hat a))=(\phi \circ \phi_{e^{B}})(a \oplus \hat a)$.

\item
Equivariance of $f$:
\begin{multline*}
(f_B \circ f)(m \oplus \hat m,t) =f_B(\tau_{t,\hat m},m)
=(\tau_{t,\hat m}-\bra mB,m)
\\=(\tau_{t,Bn+\hat m},m)
=f(m \oplus(Bm+\hat m),t)
= (f \circ f_{e^{B}})(m \oplus\hat  m,t)\text{.}
\end{multline*}

\item
Compatibility of $\eta$:
\begin{align*}
&\mquad\eta_B(\phi(a_1\oplus \hat a_1),\phi(a_2\oplus \hat a_2))+ f_B(\eta(a_1\oplus \hat a_1,a_2\oplus \hat a_2)) 
\\&= \eta_B(a_1,a_2)+ f_B(0,\hat a_1a_2)
\\&= (0\oplus 0,\lbraket{a_2}{B}{a_1})+ (0 \oplus 0,\hat a_1a_2)
\\&= (0\oplus0,(-\braket{a_1}{B}{a_2}+\hat a_1a_2)+ f(0 \oplus 0, \lbraket{a_1}{B}{a_2})
\\&= \eta(a_1\oplus (Ba_1+\hat a_1),a_2\oplus (Ba_2+\hat a_2))+ f(0 \oplus 0, \lbraket{a_1}{B}{a_2})
\\&= \eta(\phi_{e^{B}}(a_1\oplus \hat a_1),\phi_{e^{B}}(a_2\oplus \hat a_2))+ f(\eta_{e^{B}}(a_1\oplus \hat a_1,a_2\oplus \hat a_2))\text{.}
\end{align*}

\vspace{-2.1em}

\end{enumerate}
\end{proof}

Due to the equivariance $\leleR$ induces a semi-strict homomorphism 
\begin{equation*}
\leleRsonZ:\TFgeo \to \TBF\text{,}
\end{equation*}
see \cref{sec:equivCI}.
As described there, it induces in cohomology following map $(\leleRsonZ)_{*}$. Consider a $\TFgeo$-cocycle $(B,a,\hat a,m,\hat m,t)$ with respect to an open cover $\{U_i\}_{i\in I}$ as in \cref{sec:coctfgeo}.
Its image under $(\leleRsonZ)_{*}$ is the $\TBF$-cocycle $(B,a,m,\tau)$ consisting of the matrices $B_{ij}$, the smooth maps $a_{ij}$, the numbers $m_{ijk}$, and the smooth maps \begin{equation*}
\tau_{ijk}:U_i \cap U_j \cap U_k \to C^{\infty}(\T^{n},\ueins)
\end{equation*}
defined by
\begin{equation}
\label{eq:tfgeo:lele:tau}
\tau_{ijk}(x)(a) :=t_{ijk} + \hat m_{ijk}(a-a_{ik}(x))- a_{ij}(x)\hat a_{jk}(x)\text{.}
\end{equation}

The following theorem is the main result of this article, and it will be proved in \cref{sec:proof}. 

\begin{theorem}
\label{th:main}
The left leg projection of a half-geometric T-duality correspondence,
\begin{equation*}
(\leleRsonZ)_{*}: \h^1(X,\TFgeo) \to \h^1(X,\TBF)\text{,}
\end{equation*}
is a bijection. In other words, up to isomorphism, every $F_1$ T-background is the left leg of a unique half-geometric T-duality correspondence. 
\end{theorem}

\begin{remark}
\cref{th:main} implies that the classifying spaces $|B\TFgeo|$ and $|B\TBF|$ are equivalent, see \cref{re:equivclass}. Showing this equivalence directly would provide an alternative proof of \cref{th:main}.  
The Lie 2-groups $\TFgeo$ and $\TBF$ are, however, not  equivalent, since they have different homotopy types, see \cref{eq:piF1,eq:piTFgeo}.  
\end{remark}

\begin{remark}
In a local version of T-folds,  surjectivity was proved by Hull on the level of differential forms, i.e. with curvature 3-forms $H\in \Omega^3(E)$ instead of bundle gerbes   \cite{Hull2007}. In that context, Hull proves that the existence of T-duals of a T-background $(E,\mathcal{G})$ requires  that $\iota_X\iota_YH=0$, where $X,Y$ are vertical vector fields on $E$; this is a local,  infinitesimal version of the $F_2$ condition. Hull explains that admitting \quot{non-geometric T-duals} allows to replace it by the weaker condition $\iota_X\iota_Y\iota_ZH=0$; this is a local, infinitesimal version of the $F_1$ condition. Thus, \cref{th:main} shows that our half-geometric T-duality correspondences realize Hull's non-geometric T-duals. 
\end{remark}

\begin{remark}
We consider the half-geometric T-duality correspondence $\mathcal{C}_B$ whose $\TFgeo$-cocycle is trivial except for the matrices $B_{ij}$, which then form a 2-cocycle 
$B_{ij}:U_i \cap U_j \to \mathfrak{so}(n,\Z)$.
This realizes the splitting homomorphism
\begin{equation*}
\idmorph{\mathfrak{so}(n,\Z)} \to \TD[n] \ltimes \mathfrak{so}(n,\Z)=\TFgeo\text{.}
\end{equation*}
We observe that the left leg of $\mathcal{C}_B$ is the $F_1$ T-background represented by a cocycle $(B,0,0,0)$, whose only data are the given matrices $B_{ij}$. It is the image of the cocycle $B$ under the 2-functor $I$ constructed in \cref{ex:F1}. In terms of its geometric version, if $Z$ is a principal $\mathfrak{so}(n,\Z)$-bundle over $X$ classified by $[B]\in \h^1(X,\mathfrak{so}(n,\Z))$, then the $F_1$ T-background    $(X \times \T^{n},\mathcal{R}_{\mathfrak{so}(n,\Z)}(Z))$ is the left leg of the half-geometric T-duality correspondence $\mathcal{C}_B$. \end{remark}

\begin{example}
\label{ex:F12}
We consider over $X=S^1$ the $\mathfrak{so}(2,\Z)$-principal bundle $Z$ whose components are $Z^{11}=Z^{22}=S^1 \times \Z$ and $Z^{12}=\R \to S^1$, and $Z^{21}=\R^{\vee} \to S^1$. If $B$ is a classifying cocycle, and $\mathcal{C}_B$ the corresponding half-geometric T-duality correspondence, then $\mathcal{C}_B$ is the \quot{non-geometric T-dual} of the $F_1$ T-background $(E,\mathcal{G})$ of  \cref{ex:F1ex},  with $E=\T^{3}$ and $\mathcal{G}$ representing the canonical class in $\h^3(\T^3,\Z)$. 
\end{example}

\subsection{Proof of the main result}

\label{sec:proof}

In  this section we prove \cref{th:main}. Our proof is an explicit calculation on the level of cocycles, and is performed in three  steps.

\paragraph{Step 1: Construction of a pre-image candidate.} We start with a $\TBF$-cocycle $(B,a,m,\tau)$ with respect to an open cover $\mathcal{U}=\{U_i\}_{i\in I}$, i.e. matrices $B_{ij}\in \mathfrak{so}(n,\Z)$, numbers $m_{ijk}\in \Z^{n}$ and smooth maps
\begin{equation*}
a_{ij}:U_i \cap U_j \to \mathbb{R}^{n}
\quand
\tau_{ijk}:U_i \cap U_j \cap U_k \to C^{\infty}(\T^{n},\ueins)
\end{equation*}
satisfying \cref{coc:F1:1,coc:F1:2,coc:F1:3}.
Next we construct a $\TFgeo$-cocycle (with respect to a refinement of the open cover $\mathcal{U}$). We define $\hat m_{ijk}\in \Z^{n}$ to be the $n$ winding numbers of $\tau_{ijk}(x):\T^{n} \to \ueins$, which is independent of $x\in U_i\cap U_j \cap U_k$. The cocycle condition for $\tau$ \cref{coc:F1:3} implies the necessary condition of \cref{re:coc:TFgeo:5}, namely
\begin{equation}
\label{eq:cocm}
\hat m_{ikl}+\hat m_{ijk}+B_{kl}m_{ijk} = \hat m_{ijl}+\hat m_{jkl}\text{,}
\end{equation}
since the $i$-th winding number of $-\bra{m_{ijk}}{B_{kl}}$ is the $i$-th component of $B_{kl}m_{ijk}$.
We define smooth maps $\tilde m_{ijk}:U_i \cap U_j \cap U_k \to \R$ by $\tilde m_{ijk}(x):=\hat m_{ijk}+B_{jk}a_{ij}(x)$. It is straightforward to check using \cref{eq:cocm} that $\tilde m_{ijk}$ is a \v Cech cocycle, i.e.  $[\tilde m]\in \check \h^2(M,\underline{\R}^{n})=0$. Thus, after passing to a refinement of $\mathcal{U}$ we can assume
that there exist smooth maps $\hat a_{ij}:U_i \cap U_j \to \R^{n}$ satisfying
\begin{equation*}
\hat a_{ik} = \hat m_{ijk} + \hat a_{jk}+\hat a_{ij} + B_{jk}a_{ij}\text{;}
\end{equation*}
this is cocycle condition \cref{coc:TFgeo:3} for $\TFgeo$-cocycles.
On the other hand, by definition of a winding number, there exist smooth maps
\begin{equation*}
\tilde\tau_{ijk}: (U_i \cap U_j \cap U_k) \times \mathbb{T}^{n} \to \R
\end{equation*}
such that $\tau_{ijk}(x)(a)=\tilde \tau_{ijk}(x,a) + a \hat m_{ijk}$ in $\ueins$. 
The cocycle condition for $\tau$ implies
\begin{multline*}
 \tilde\tau_{ijl}(x,a)+a \hat m_{ijl}+  \tilde \tau_{jkl} (x,a)+a \hat m_{jkl} 
\\\mquad\mquad= \tilde\tau_{ijk}(x,a-a_{kl}(x))-a_{kl}(x) \hat m_{ijk}+a \hat m_{ijk} + \tilde\tau_{ikl}(x,a)+a \hat m_{ikl}
\\-\braket{m_{ijk}}{B_{kl}}{a-a_{kl}(x)}-\lbraket{a_{ij}(x)}{B_{kl}}{a_{jk}(x)} - \lbraket{a_{ik}(x)}{B_{kl}}{m_{ijk}}+\varepsilon_{ijkl}
\end{multline*}
for a uniquely defined $\varepsilon_{ijkl}\in \Z$.
Subtracting $a$ times \cref{eq:cocm} we get
\begin{multline*}
 \tilde\tau_{ijl}(x,a)+  \tilde \tau_{jkl} (x,a)-  \tilde\tau_{ijk}(x,a-a_{kl}(x))- \tilde\tau_{ikl}(x,a)+a_{kl}(x)\hat m_{ijk} 
\\-\braket{m_{ijk}}{B_{kl}}{a_{kl}(x)}+\lbraket{a_{ij}(x)}{B_{kl}}{a_{jk}(x)} + \lbraket{a_{ik}(x)}{B_{kl}}{m_{ijk}}=\varepsilon_{ijkl}
\end{multline*}
Since $ \mathbb{T}^{n}$ is connected, we see that the expression
\begin{equation}
\label{eq:def:deltaijk}
\delta_{ijkl}(x):=\tilde\tau_{ijl}(x,a)+  \tilde \tau_{jkl} (x,a)-  \tilde\tau_{ijk}(x,a-a_{kl}(x)) - \tilde\tau_{ikl}(x,a)\in \R
\end{equation}
is independent of $a$, and
\begin{multline}
\label{eq:epsilon}
\varepsilon_{ijkl}=\delta_{ijkl}(x)+a_{kl}(x)\hat m_{ijk}-\braket{m_{ijk}}{B_{kl}}{a_{kl}(x)}\\+\lbraket{a_{ij}(x)}{B_{kl}}{a_{jk}(x)} + \lbraket{a_{ik}(x)}{B_{kl}}{m_{ijk}}\text{.}
\end{multline}
Using the independence of $a$, one can now check that $\delta_{ijkl}$ is a \v Cech cocycle.
Thus, $[\delta] \in \check \h^3(M,\underline{\R})=0$, and (after again passing to a refinement) there exist smooth maps $\omega_{ijk}: U_i \cap U_j \cap U_k \to \R$ with $\delta\omega=\delta$. Substituting this in \cref{eq:epsilon},
and pushing to $\ueins$-valued maps, we get
\begin{equation*}
(\delta \omega)_{ijkl}(x)+a_{kl}(x)\hat m_{ijk}-\braket{m_{ijk}}{B_{kl}}{a_{kl}(x)}+\lbraket{a_{ij}(x)}{B_{kl}}{a_{jk}(x)} + \lbraket{a_{ik}(x)}{B_{kl}}{m_{ijk}}=0\text{.}
\end{equation*}
We define
\begin{equation*}
t_{ijk}(x):=-\omega_{ijk}(x)+a_{ij}(x)\hat a_{jk}(x)+\hat m_{ijk} a_{ik}(x)\text{.}
\end{equation*}
A tedious but straightforward calculation shows that $t_{ijk}$ satisfies the cocycle condition \cref{coc:TFgeo:4} for $\TFgeo$-cocycles.
Thus, we have obtained a $\TFgeo$-cocycle $(B,a,\hat a,m,\hat m,t)$.

\paragraph{Step 2: Check that the candidate is a pre-image.}\ 
The left leg of our  $\TFgeo$-cocycle $(B,a,\hat a,m,\hat m,t)$ is given by $B_{ij},a_{ij}$, $m_{ijk}$, and 
\begin{align*}
\tau_{ijk}'(x,a) &:= t_{ijk}(x) +(a - a_{ik}(x))\hat m_{ijk}-a_{ij}(x)\hat a_{jk}(x)
= -\omega_{ijk}(x) + a \hat m_{ijk}\text{.}
\end{align*}
We prove that the $\TBF[n]$-cocycles $(B_{ij},a_{ij},m_{ijk},\tau_{ijk})$ and $(B_{ij},a_{ij},m_{ijk},\tau'_{ijk})$ are equivalent. For this it suffices to provide $\varepsilon_{ij}:U_i \cap U_j \to C^{\infty}(\T^{n},\ueins)$ such that
\begin{equation*}
\tau_{ijk}'+\lw{a_{jk}} \varepsilon_{ij}+\varepsilon_{jk} = \varepsilon_{ik}+ \tau_{ijk}\text{.}
\end{equation*}
since this is \cref{coc:F1:8} for $C_i=0$, $p_i=0$ and $z_{ij}=0$. 
In order to construct $\varepsilon_{ij}$, we consider the smooth maps  $\beta_{ijk}:U_i \cap U_j \cap U_k \to C^{\infty}(\T^{n},\R)$ defined by
\begin{equation*}
\beta_{ijk}(x)(a):= \tilde\tau_{ijk} (x,a)+\omega_{ijk}(x)\text{.}
\end{equation*}
These satisfy the following condition:
\begin{align}
\label{eq:cocbetaijk}
\beta_{ikl} + \lw{a_{kl}}\beta_{ijk} - \beta_{ijl} - \beta_{jkl}
\eqcref{eq:def:deltaijk}-\delta_{ijkl}+(\delta\omega)_{ijkl}=0
\end{align}
It also satisfies (after pushing to $C^{\infty}(\T^{n},\ueins)$):
\begin{equation*}
\beta_{ijk}(x)(a)=\tilde\tau_{ijk} (x,a)+\omega_{ijk}(x)= \tau_{ijk}(x)(a)-a \hat m_{ijk}+\omega_{ijk}(x)=\tau_{ijk}(x)(a)-\tau'_{ijk}(x)(a)\text{.}
\end{equation*}
Let $\{\psi_i\}_{i\in I}$ be a partition of unity subordinate to our open cover.
We define
\begin{equation*}
\varepsilon_{ij}(x)(a) := \sum_{h\in I} \psi_h(x)\beta_{hij}(x)(a)\text{.}
\end{equation*}
Then we obtain 
\begin{multline*}
\varepsilon_{ij}(x)(a-a_{jk})+\varepsilon_{jk} (x)(a)-\varepsilon_{ik}(x)(a)
\\[-1.2em]= \sum_{h\in I} \psi_h(x) \left ( \beta_{hij}(x)(a-a_{jk})+\beta_{hjk}(x)(a)-\beta_{hik}(x)(a)\right)
\eqcref{eq:cocbetaijk}  \sum_{h\in I} \psi_h(x) \beta_{ijk}(x,a)
\\= \beta_{ijk}(x,a)\text{.}
\end{multline*}
This shows the claimed equivalence.

\paragraph{Step 3: Injectivity of the left leg.} 
We suppose that we have two $\TFgeo$-cocycles $(B,a,\hat a,m,\hat m,t)$  and $(B',a',\hat a',m',\hat m',t')$ whose left legs $(B,a,m,\tau)$ and $(B',a',m',\tau')$ are equivalent. Thus,  there exist matrices $C_i \in \mathfrak{so}(n,\Z)$, numbers $z_{ij}\in \Z^{n}$ and smooth maps 
\begin{equation*}
p_i: U_i \to \R^{n}
\quand
\varepsilon_{ij}:U_i \cap U_j \to C^{\infty}(\T^{n},\ueins)
\end{equation*}
such that the cocycle conditions \cref{coc:F1:6,coc:F1:7,coc:F1:8,coc:F1:9} are satisfied.
Expressing \cref{coc:F1:8} in terms of $t_{ijk}$ and $t'_{ijk}$ using the definition of left legs, we obtain
\begin{align}
\nonumber&\mquad\mquad t'_{ijk}(x) -t_{ijk} (x)-\hat m'_{ijk}a_{ik}'(x)+\hat m_{ijk}(a_{ik}(x)+p_k(x))- a_{ij}'(x)\hat a_{jk}'(x)+ a_{ij}(x)\hat a_{jk}(x)
\\\nonumber&+ \lbraket{p_i(x)}{B_{jk}'}{a_{ij}'(x)}- \lbraket {p_j(x) +a_{ij}(x)}{B_{jk}'}{z_{ij}}+\braket{z_{ij}}{B_{jk}'}{a'_{jk}(x)}
\\\nonumber&- \lbraket {a_{ij}(x)}{B_{jk}'}{p_j(x)}+ \lbraket{a_{ik}(x)}{C_k}{m_{ijk}}
\\\nonumber&-\braket{m_{ijk}}{C_k}{p_k(x)}+ \lbraket{a_{ij}(x)}{C_k}{a_{jk}(x)}
\\\nonumber&\mquad=\varepsilon_{ik}(x,a)- \varepsilon_{ij}(x,a-a'_{jk}(x))- \varepsilon_{jk}(x,a) 
\\&+ (\hat m_{ijk}- \hat m'_{ijk})a-\braket{m_{ijk}}{C_k}{a}+\braket{z_{ij}}{B_{jk}'}{a}\text{.}
\label{eq:leftleginjhalf:1}
\end{align}
In order to show  the equivalence  between the $\TFgeo$-cocycles $(B,a,\hat a,m,\hat m,t)$  and $(B',a',\hat a',m',\hat m',t')$, 
our goal is to find numbers $\hat z_{ij}\in \Z^{n}$, and smooth  maps $\hat p_i:U_i \to \R^{n}$ and $e_{ij}:U_i \cap U_j \to \ueins$ such that the remaining required cocycle conditions \cref{coc:TFgeo:8,coc:TFgeo:9} are satisfied. 
We let $\hat z_{ij} \in \Z^{n}$ be the $n$ winding numbers of $\varepsilon_{ij}(x):\T^{n} \to \ueins$. \cref{eq:leftleginjhalf:1} implies
\begin{equation}
\label{eq:injlele:1}
0=\hat z_{ik}-\hat z_{ij}-\hat z_{jk}+\hat m_{ijk}- \hat m'_{ijk}+C_k m_{ijk}-B_{jk}'z_{ij}\text{;}
\end{equation}
this is  necessary for \cref{re:coc:TFgeo:8}.
We consider $\beta_{ij} := \hat z_{ij}+C_ja_{ij}+\hat a_{ij}-\hat a_{ij}'-B'_{ij}p_i$, which is by \cref{eq:injlele:1} an $\R^{n}$-valued \v Cech 1-cocycle, and chose $\hat p_i:U_i \to \R^{n}$ such that $\beta_{ij}=\hat p_i-\hat p_j$; this gives \cref{coc:TFgeo:8}. 
By definition of a winding number, there exist smooth maps $\tilde \varepsilon_{ij}:(U_i \cap U_j)\times \T^{n} \to \R$ such that
$\varepsilon_{ij}(x)(a)=\tilde \varepsilon_{ij}(x,a) + aw_{ij}$. Substituting in \cref{eq:leftleginjhalf:1} and subtracting \cref{eq:injlele:1} we get
\begin{align*}
&\mquad\mquad t'_{ijk}(x) -t_{ijk} (x)-\hat m'_{ijk}a_{ik}'(x)+\hat m_{ijk}(a_{ik}(x)+p_k(x))- a_{ij}'(x)\hat a_{jk}'(x)+ a_{ij}(x)\hat a_{jk}(x)
\\&+ \lbraket{p_i(x)}{B_{jk}'}{a_{ij}'(x)}- \lbraket {p_j(x) +a_{ij}(x)}{B_{jk}'}{z_{ij}}\\&+\braket{z_{ij}}{B_{jk}'}{a'_{jk}(x)}- \lbraket {a_{ij}(x)}{B_{jk}'}{p_j(x)}+ \lbraket{a_{ik}(x)}{C_k}{m_{ijk}}\\&-\braket{m_{ijk}}{C_k}{p_k(x)}+ \lbraket{a_{ij}(x)}{C_k}{a_{jk}(x)}-a'_{jk}(x)\hat z_{ij}
\\&\mquad=\varepsilon_{ijk}+\tilde \varepsilon_{ik}(x,a) -\tilde \varepsilon_{ij}(x,a-a_{jk}'(x)) -\tilde \varepsilon_{jk}(x,a)\text{.}
\end{align*}
as an equation in $\R$, for a  uniquely determined constant $\varepsilon_{ijk}\in \Z$.
This means that the expression
\begin{equation*}
\rho_{ijk}(x,a) :=\tilde \varepsilon_{ik}(x,a) -\tilde \varepsilon_{ij}(x,a-a_{jk}'(x)) -\tilde \varepsilon_{jk}(x,a)
\end{equation*}
is independent of $a$. It is straightforward to check that $\rho_{ijk}$ is a \v Cech cocycle, so that $[\rho] \in \check \h^2(X,\sheaf{\R})=0$.
Thus, there exist $e'_{ij}:U_i \cap U_j \to \R^{n}$ such that $\rho_{ijk}=(\delta e')_{ijk}$. Now we have
\begin{align*}
&\mquad\mquad t'_{ijk}(x) -t_{ijk} (x)-\hat m'_{ijk}a_{ik}'(x)+\hat m_{ijk}(a_{ik}(x)+p_k(x))- a_{ij}'(x)\hat a_{jk}'(x)+ a_{ij}(x)\hat a_{jk}(x)
\\&+ \lbraket{p_i(x)}{B_{jk}'}{a_{ij}'(x)}- \lbraket {p_j(x) +a_{ij}(x)}{B_{jk}'}{z_{ij}}\\&+\braket{z_{ij}}{B_{jk}'}{a'_{jk}(x)}- \lbraket {a_{ij}(x)}{B_{jk}'}{p_j(x)}+ \lbraket{a_{ik}(x)}{C_k}{m_{ijk}}\\&-\braket{m_{ijk}}{C_k}{p_k(x)}+ \lbraket{a_{ij}(x)}{C_k}{a_{jk}(x)}-a'_{jk}(x)\hat z_{ij}
\\&\mquad= e'_{ik} -e'_{ij} -e'_{jk}\text{.}
\end{align*}
We consider $\eta_{ijk}:=(p_i-p_j)(\hat p_k-\hat p_j)$. It is easy to check that $(\delta\eta)_{ijkl}=0$. After subtracting $\eta_{ijk}$ from $\rho_{ijk}$, we can assume that  $(\delta e')_{ijk}=\rho_{ijk} -\eta_{ijk}$. Finally, we consider 
\begin{equation*}
e_{ij} := e'_{ij} - \hat z_{ij}a'_{ij}-a_{ij}\hat p_j+\hat a_{ij}p_i+\braket{p_i}{C_j}{a_{ij}} \in \ueins\text{.}
\end{equation*}
It is again tedious, but straightforward to check that
\begin{multline*}
(\delta e)_{ijk}=t'_{ijk} -t_{ijk} - z_{ij} \hat a_{jk}'+m_{ijk}\hat p_k+ \lbraket{a_{ij}'}{B_{jk}'}{p_i}- \lbraket{p_j}{B_{jk}'}{a_{ij}}
\\+ \lbraket{m_{ijk}}{C_k}{a_{ik}}+ \lbraket{a_{jk}}{C_k}{a_{ij}}-\lbraket{z_{ij}}{B_{jk}'}{p_j+a_{ij}}
\end{multline*}
i.e., the remaining cocycle condition \cref{coc:TFgeo:9} is satisfied, and we have proved the equivalence of the two $\TFgeo$-cocycles. 

\setsecnumdepth{3}

\subsection{Remarks about half-geometric T-duality}

\subsubsection{Polarizations}

\label{sec:half-geo:pol}

By \cref{prop:fibresequence} we have for all smooth manifolds an exact sequence
\begin{equation*}
\h^1(X,\TD)/\mathfrak{so}(n,\Z) \to \h^1(X,\TFgeo) \to \h^1(X,\mathfrak{so}(n,\Z)) \to 0\text{.}
\end{equation*}
In particular, every half-geometric T-duality correspondence $\mathcal{C}$ has an underlying principal $\mathfrak{so}(n,\Z)$-bundle $p_{*}(\mathcal{C})$, and if that bundle is trivializable, then the half-geometric T-duality correspondence is isomorphic to a (geometric) T-duality correspondence. We have the following natural definition:

\begin{definition}
A \emph{polarization} of a half-geometric T-duality correspondence $\mathcal{C}$ over $X$ is a section $\sigma$ of the underlying $\mathfrak{so}(n,\Z)$-bundle $p_{*}(\mathcal{C})$.
\end{definition}

One can easily verify on the level of cocycles, that a choice of a polarization $\sigma$  of $\mathcal{C}$ determines a T-duality correspondence $\mathcal{C}_{\sigma}^{geo}$ together with  an isomorphism $\mathcal{C} \cong \mathcal{C}_{\sigma}^{geo}$. Alternatively, this follows from the fact that the sequence $\TD \to \TFgeo \to \mathfrak{so}(n,\Z)$ is a fibre sequence; see \cite{Nikolausa}. In particular, $R(\mathcal{C}_{\sigma}^{geo})$ is T-dual to $L(\mathcal{C}_{\sigma}^{geo})$, and the $F_2$ T-background $L(\mathcal{C}_{\sigma}^{geo})$ is isomorphic as T-backgrounds to the left leg of $\mathcal{C}$. 

Polarizations exist always locally. If $U\subset X$ is connected, then -- since $\mathfrak{so}(n,\Z)$ is discrete -- two polarizations $\sigma_1$ and $\sigma_2$ over $U$ differ by a uniquely defined matrix $D\in \mathfrak{so}(n,\Z)$, via $\sigma_2=\sigma_1\cdot D$. The corresponding T-duality correspondences $\mathcal{C}_{\sigma_1}^{geo}$ and $\mathcal{C}_{\sigma_2}^{geo}$ differ then by the action of precisely this $D\in \mathfrak{so}(n,\Z)$.

Assume $\h^1(X,\Z)=0$, for example when $X$ is connected and simply-connected. We have seen in \cref{re:TBF:3} that every $F_1$ T-background is isomorphic to an $F_2$ T-background, and hence T-dualizable. Correspondingly, every half-geometric T-duality correspondence is isomorphic to a (geometric) T-duality correspondence, under the choice of a global polarization.

\subsubsection{Inclusion of ordinary T-duality}

The semi-strict homomorphism $i:\TD \to \TFgeo$ allows to consider ordinary T-duality correspondences as half-geometric T-duality correspondences. We have the following result: 

\begin{proposition}
\label{prop:tdasgeo}
Let $\mathcal{C}_1$ and $\mathcal{C}_2$ be two T-duality correspondences over a  smooth manifold $X$. Then, the following two statements are equivalent:
\begin{enumerate}[(a)]

\item
$\mathcal{C}_1$ and $\mathcal{C}_2$ are isomorphic as half-geometric T-duality correspondences. 

\item
The left legs $L(\mathcal{C}_1)$ and $L(\mathcal{C}_2)$ are isomorphic as T-backgrounds.

\item
For each connected component of $X$ there exists a matrix $B\in \mathfrak{so}(n,\Z)$ such that $(F_{e^{B}})_{*}(\mathcal{C}_1)$ and $\mathcal{C}_2$ are isomorphic as T-duality correspondences.

\end{enumerate}
\end{proposition}

\noindent
In (c),  $F_{e^{B}}: \TD \to \TD$ is the action of $\mathfrak{so}(n,\Z)$ on $\TD$ defined in \cref{sec:sonZactsTD}.  

\begin{proof}
By construction, the diagram
\begin{equation*}
\alxydim{@C=2cm@R=3em}{\TD \ar[r]^{\leleR} \ar[d]  & \TBFFR \ar[d] \\ \TFgeo \ar[r]_{\leleRsonZ} & \TBF }
\end{equation*}
is strictly commutative, and induces the following diagram commutative in cohomology,
\begin{equation*}
\alxydim{@C=2cm@R=3em}{\h^1(X,\TD) \ar[r]^{\leleR_{*}} \ar[d]  & \h^1(X,\TBFFR)\ar[d] \\ \h^1(X,\TFgeo)  \ar[r]_{(\leleRsonZ)_{*}} & \h^1(X,\TBF) \text{.} }
\end{equation*}
Commutativity shows that (a) implies (b). The injectivity in \cref{th:main} shows the converse implication. By 
  \cref{prop:fibresequence} $i$ induces a well-defined map
\begin{equation*}
\h^1(X,\TD)/\mathfrak{so}(n,\Z) \to \h^1(X,\TFgeo)\text{,}
\end{equation*}
whose existence shows that (c) implies (a). Over each connected component  this map is injective, which shows that (a) implies (c).
\end{proof}

\begin{remark} 
The equivalence of (b) and (c) was already proved in \cite{Bunke2006a}.
\end{remark}

\subsubsection{Half-geometric T-duality correspondences with trivial torus bundle}

In order to investigate half-geometric T-duality correspondences whose left legs have trivial torus bundles,  we consider the following sequence of Lie 2-group homomorphisms:
\begin{equation}
\label{eq:trivtordual}
\alxydim{}{\idmorph{\mathfrak{so}(n,\Z)} \times \act{\R^{n}}{\Z^{n}} \times B\ueins \ar[r]^-{\tilde I} & \TFgeo_{\phantom{\frac{}{}X}} \ar[r]^-{\tilde T} & \idmorph{\T^{n}}\text{,}}
\end{equation}
where $\tilde T := T \circ  \leleRsonZ$, with $T:\TBF \to \idmorph{\T^{n}}$  defined in \cref{ex:F1}. In other words, $\tilde T$ projects to the underlying torus bundle of the left leg. 
 By $\act{\R^{n}}{\Z^{n}}$ we have denoted the crossed module $(\R^{n},\Z^{n},t,\alpha)$, where $t: \Z^{n} \to \R^{n}$ is the inclusion, and $\alpha$ is the trivial action.  
 The homomorphism $\tilde I$ is defined as follows. It sends  an object $(B,b,\ast)$ to $(0,b,B)$, and a morphism $(B,m,b,t): (B,b,t) \to (B,b+m,t)$ to $(0,b,0,m,t,B)$.

\begin{lemma}
\label{lem:fibtortrivhalf}
The following sequence induced by \cref{eq:trivtordual} in cohomology is exact: 
\begin{equation*}
\alxydim{}{\h^1(X,\mathfrak{so}(n,\Z)) \times \h^1(X,\act{\R^{n}}{\Z^{n}}) \times \h^3(X,\Z) \ar[r]^-{\tilde I_{*}} & \h^1(X,\TFgeo) \ar[r]^-{\tilde T_{*}} & \h^2(X,\Z^{n}) \ar[r] & 0}
\end{equation*}
Here we have used the usual identifications between non-abelian cohomology and ordinary cohomology, see \cref{re:ordinarycohomology}. 
\end{lemma}

\begin{proof}
We compare with the sequence of \cref{lem:trivtorus}. There is an obvious strict intertwiner $\zeta:\act{\R^{n}}{\Z^{n}} \to B\Z^{n}$, making the diagram
\begin{equation*}
\alxydim{@R=3em}{\idmorph{\mathfrak{so}(n,\Z)} \times \act{\R^{n}}{\Z^{n}} \times B\ueins \ar[r]^-{\tilde I} \ar[d]_{\id \times \zeta \times \id} & \TFgeo_{\phantom{\frac{}{}X}} \ar[d]^{\leleRsonZ} \\ \idmorph{\mathfrak{so}(n,\Z)} \times B\Z^{n} \times B\ueins \ar[r]_-{I} & \TBF_{\phantom{X}}}
\end{equation*}
commutative. It is easy to check that $\zeta$ induces a bijection 
\begin{equation*}
\h^1(X,\act{\R^{n}}{\Z^{n}}) \cong \h^1(X,B\Z^{n})=\h^2(X,\Z^{n})\text{.}
\end{equation*}
Together with \cref{th:main} we have the claim. 
\end{proof}

Restricting to ordinary T-duality correspondences, we obtain a diagram
\begin{equation*}
\alxydim{@R=3em}{\act{\R^{n}}{\Z^{n}} \times B\ueins \ar[d] \ar[r] & \TD \ar[r] \ar[d] & \idmorph{\T^{n}} \ar@{=}[d]\\ \idmorph{\mathfrak{so}(n,\Z)} \times \act{\R^{n}}{\Z^{n}} \times B\ueins  \ar[r]^-{\tilde I} & \TFgeo_{\phantom{\frac{}{}X}}  \ar[r]^-{\tilde T} & \idmorph{\T^{n}} }
\end{equation*} 
in which all horizontal and vertical sequences induce exact sequences in the sense of \cref{lem:fibtortrivhalf} and \cref{prop:fibresequence}. This diagram describes ordinary and half-geometric T-duality correspondences whose left legs have trivial torus bundles. 

The homomorphism $\tilde I$ lifts the 2-functor $I$ defined in \cref{ex:F1} along the left leg projection, in the sense that we have the following commutative diagram, which combines the results of this section and \cref{ex:F1}: 
\begin{equation*}
\alxydim{@C=1em@R=1.5em}{& \act{\R^{n}}{\Z^{n}} \times B\ueins \ar[ddd]^>>>>>>>>>{\zeta\times\id} \ar[dl]\ar[rr] && \TD   \ar[ddd]^>>>>>>>>>>{\leleR} \ar[dl] \ar[rr] && \idmorph{\T^{n}} \ar@{=}[ddd] \ar@{=}[dl] \\ \idmorph{\mathfrak{so}(n,\Z)} \times \act{\R^{n}}{\Z^{n}} \times B\ueins \ar[rr]^-<<<<<<<<<<<<<<<{\tilde I} \ar[ddd]_{\id \times \zeta\times\id}  && \TFgeo_{\phantom{\frac{}{}X}} \ar[rr] \ar[ddd]^<<<<<<<<{\leleRsonZ} && \idmorph{\T^{n}} \ar@{=}[ddd]\\ && \\  & B\Z^{n} \times B\ueins \ar[rr]  \ar[dl]  && \TBFFR \ar[dl] \ar[rr] && \idmorph{\T^{n}}  \\ \idmorph{\mathfrak{so}(n,\Z)} \times B\Z^{n} \times B\ueins  \ar[rr]_-{I} && \TBF \ar[rr] && \idmorph{\T^{n}} \ar@{=}[ur]}
\end{equation*}

In \cref{ex:F1} we have described a geometric counterpart of the homomorphism $I$, the 2-functor of \cref{eq:geoI}. The 2-functor $\tilde I$ has in general no such counterpart in classical geometry, since the resulting half-geometric T-duality correspondences are only \emph{half}-geometric. 
Geometrically accessible is only the restriction of $\tilde I$ to $\act{\R^{n}}{\Z^{n}} \times B\ueins$, which assigns a T-duality correspondence $((E,\mathcal{G}),(\hat E,\widehat{\mathcal{G}}),\mathcal{D})$ to a $\act{\R^{n}}{\Z^{n}}$-bundle gerbe $\widetilde{\mathcal{H}}$ and a $\ueins$-bundle gerbe $\mathcal{G}$ over $X$. For completeness, let us describe this correspondence.
Without explaining in more detail what a $\act{\R^{n}}{\Z^{n}}$-bundle gerbe is, we remark that it induces a $\Z^{n}$-bundle gerbe $\mathcal{H} := (\zeta)_{*}(\widetilde{\mathcal{H}})$, and, via the strict intertwiner $\act{\R^{n}}{\Z^{n}} \to \idmorph{\T^{n}}$, a $\T^{n}$-bundle $E$ over $X$. The two characteristic classes coincide:
\begin{equation*}
\mathrm{DD}(\mathcal{H}) = \mathrm{c}_1(E) \in \h^2(X,\Z^{n})\text{.}
\end{equation*}  
The left leg of this correspondence is the $F_2$ T-background  $(E,\mathcal{G})=(X \times \T^{n},\mathcal{R}_{\Z}(\mathcal{H})\otimes \pr_X^{*}\mathcal{G})$, see \cref{ex:F1}. For the right leg, we notice that the composition
\begin{equation*}
\alxydim{}{\act{\R^{n}}{\Z^{n}} \times B\ueins \ar[r]^-{\tilde I} & \TD \ar[r]^-{\rele} & \TBFF_{\phantom{X}}  }
\end{equation*}  
sends an object $(b,\ast)\in \R^{n}$ to $(0,b)$ in $\TD$ and then to $b\in \T^{n}$, and a morphism $(m,b,t)\in \Z^{n} \times \R^{n} \times \ueins$ to $(0,b,0,m,t)$ in $\TD$ and then to $(b,t)\in \T^{n} \times C^{\infty}(\T^{n},\ueins)$, where $t$ is regarded as a constant map. Thus, the right leg is the $F_3$ T-background $(\hat E,\widehat{\mathcal{G}})=(E,\pr_X^{*}\mathcal{G})$, see \cref{re:cocF3}. 
It remains to construct the 1-isomorphism $\mathcal{D}: \pr_1^{*}\mathcal{G} \to \pr_2^{*}\widehat{\mathcal{G}}$ over $E \times_X \hat E$. We identify the latter correspondence space canonically with $\T^{n} \times E$, so that $\mathcal{D}$ becomes a 1-isomorphism
\begin{equation*}
\mathcal{D}: (\id \times \pr_X)^{*}\mathcal{R}_{\Z}(\mathcal{H})\otimes \pr_X^{*}\mathcal{G} \to \pr_X^{*}\mathcal{G}
\end{equation*}
over $\T^{n} \times E$. This is equivalent to specifying a trivialization of $(\id \times \pr_X)^{*}\mathcal{R}_{\Z}(\mathcal{H})$. 

\subsubsection{Twisted K-theory}

The twisted K-theory $K^{\bullet}(\mathcal{C})$ of a half-geometric T-duality correspondence $\mathcal{C}$ is by definition the twisted K-theory of its left leg $(E,\mathcal{G})$, i.e. the $\mathcal{G}$-twisted K-theory of the manifold $E$, i.e.,
\begin{equation*}
K^\bullet(\mathcal{C}) := K^{\bullet}(E,\mathcal{G})\text{.}
\end{equation*}

We discuss the local situation.  
Consider an open set $U \subset X$ over which $\mathcal{C}$ admits a polarization $\sigma$, with associated T-duality correspondence $\mathcal{C}^{geo}_{\sigma}$ and  an isomorphism $\mathcal{A}_{\sigma}:\mathcal{C}|_U \to \mathcal{C}_{\sigma}^{geo}$. We denote the restriction of the left leg $(E,\mathcal{G})$ to $U$ by $(E_U,\mathcal{G}_U)$. The isomorphism $\mathcal{A}$ induces an isomorphism $\mathcal{A}_\sigma: (E_U,\mathcal{G}_{U}) \to L(\mathcal{C}^{geo}_{\sigma})$, and 
\begin{equation*}
\alxydim{}{K^{\bullet}(\mathcal{C}) = K^{\bullet}(E,\mathcal{G}) \ar[r]^-{\mathrm{res}} & K^{\bullet}(E_U,\mathcal{G}_U) \ar[r]^-{\mathcal{A}_\sigma} &  K^{\bullet}(L(\mathcal{C}^{geo}_{\sigma})) \ar[r]^-{\mathrm{T}} & K^{\bullet-n}(R(\mathcal{C}_{\sigma}^{geo}))\text{,}}
\end{equation*} 
where $\mathrm{T}$ denotes the T-duality  transformation for T-duality correspondences, defined in \cite{Bouwknegt2004} (for twisted de Rham cohomology) and (in full generality) in \cite{Bunke2005a}. If another polarization $\sigma'$ over $U$ is chosen, we see that there is a canonical isomorphism
\begin{equation*}
K^{\bullet}(R(\mathcal{C}_{\sigma}^{geo}))\cong K^{\bullet}(R(\mathcal{C}_{\sigma'}^{geo}))
\end{equation*}
between the locally defined twisted K-theories of the locally defined right legs. 
In an upcoming paper we will discuss more general versions of T-folds, and their  twisted K-theory.

\begin{appendix}

\setsecnumdepth{2}

\section{Lie 2-groups}

\label{sec:2groups}

In this appendix we collect required definitions and results in the context of Lie 2-groups, and provide a number of complimentary new results. 

\subsection{Crossed modules and crossed intertwiners}

A \emph{crossed module}  is a quadruple $\Gamma =(G,H,t,\alpha)$ with Lie groups $G$ and $H$, a Lie group homomorphism $t:H \to G$ and a smooth action $\alpha:G \times H \to H$ by Lie group homomorphisms, such that
\begin{equation*}
\alpha(t(h),h')=hh'h^{-1}
\quand
t(\alpha(g,h))=gt(h)g^{-1}
\end{equation*}
for all $g\in G$ and $h,h'\in H$. We write $U := \mathrm{Ker}(t)$, which is a central Lie subgroup of $H$ and invariant under the action of $G$ on $H$.

 Essential for this article is the choice of an  appropriate class of homomorphisms between crossed modules, which we call \quot{crossed intertwiners}. They are weaker than the obvious notion of a \quot{strict intertwiner} (namely, a pair of Lie group homomorphisms respecting all structure) but stricter than weak equivalences (also known as \quot{butterflies}). 
\begin{definition}
\label{def:CI}
Let $\Gamma=(G,H,t,\alpha)$ and $\Gamma'=(G',H',t',\alpha')$ be crossed modules. A \emph{crossed intertwiner} $F:\Gamma \to \Gamma'$ is a triple $F=(\phi,f,\eta)$ consisting of Lie group homomorphisms 
\begin{equation*}
\phi: G \to G'
\quand
f: H \to H'\text{,}
\end{equation*}
and of a smooth map $\eta: G \times G \to U'$ satisfying the following axioms for all $h,h'\in H$ and $g,g',g''\in G$:
\begin{enumerate}[({CI}1),leftmargin=3em]

\item
\label{CI1}
$\phi(t(h))=t(f(h))$.

\item
\label{CI2}
$\eta(t(h),t(h')) =1$.

\item
\label{CI4}
$\eta(g,t(h)g^{-1})\cdot f(\alpha(g,h))=\alpha'(\phi(g),\eta(t(h)g^{-1},g))\cdot \alpha'(\phi(g), f(h))$.

\item
\label{CI5}
$\eta(g,g')\cdot \eta(gg',g'')=\alpha'(\phi(g),\eta(g',g''))\cdot \eta(g,g'g'')\text{.}$
\end{enumerate}
\end{definition}

\noindent
We remark that these axioms imply the following:
\begin{itemize}

\item
$f(u)\in U'$ for all $u\in U$.

\item
$\eta(g,1)=1=\eta(1,g)$.

\item
$\eta(g,g^{-1})=\alpha'(\phi(g),\eta(g^{-1},g))$.

\end{itemize}
A crossed intertwiner $(\phi,f,\eta)$ is called  \emph{strict intertwiner} if  $\eta=1$. 
The composition of crossed intertwiners is defined by
\begin{equation}
\label{def:CIcomp}
(\phi_2,f_2,\eta_2)\circ (\phi_1,f_1,\eta_1):=(\phi_2 \circ \phi_1, f_2 \circ f_1,\eta_2 \circ (\phi_1 \times \phi_1) \cdot f_2 \circ \eta_1 )\text{.}
\end{equation}
It is straightforward but a bit tedious to show that the composition is again a crossed intertwiner,whereas it is easy to check that composition is associative.
The identity crossed intertwiner is $(\id_G,\id_H,1)$.
The invertible  crossed intertwiners from a crossed module $\Gamma$ to itself form a group $\mathrm{Aut}_{CI}(\Gamma)$, which we will use in \cref{sec:semidirect} in order to define group actions on crossed modules.

\begin{example}
\begin{itemize}

\item 
If $A$ is an abelian Lie group, then $BA:=(\{e\},A,t,\alpha)$ is a crossed module in a unique way. If $A'$ is another Lie group, then a crossed intertwiner $BA \to BA'$ is exactly the same as a Lie group homomorphism $A \to A'$. 
\item
If $G$ is any Lie group, then $\idmorph{G} := (G,\{\ast\},t,\alpha)$ is a crossed module in a unique way. If $G'$ is another Lie group, then a crossed intertwiner $\idmorph{G} \to \idmorph{G'}$ is exactly the same as a Lie group homomorphism $\phi:G \to G'$.

\end{itemize}
\end{example}

\subsection{Semi-strict Lie 2-groups}

\label{sec:semistrict}
\label{sec:CI}

Crossed modules of Lie groups correspond to \emph{strict} Lie 2-groups. We need a more general class of Lie 2-groups.

\begin{definition}
A \emph{semi-strict Lie 2-group} is a Lie groupoid $\Gamma$ together with smooth functors
\begin{equation*}
m: \Gamma \times \Gamma \to \Gamma
\quand
i: \Gamma \to \Gamma\text{,}
\end{equation*}
a distinguished element $1\in \obj{\Gamma}$,
and a smooth natural transformation  (\quot{associator}) 
\begin{equation*}
\alxydim{@C=2cm@R=3em}{\Gamma \times \Gamma \times \Gamma \ar[d]_{\id \times m} \ar[r]^{m \times \id} & \Gamma \times \Gamma \ar@{=>}[dl]|*+{\lambda} \ar[d]^{m} \\ \Gamma \times \Gamma \ar[r]_{m} & \Gamma}
\end{equation*}
such that:
\begin{enumerate}[(a)]

\item
$1$  is a strict unit with respect to $m$, i.e. $m(\gamma,\id_1)=\gamma=m(\id_1,\gamma)$ for all $\gamma\in \mathrm{Mor}(\Gamma)$. 
 
\item
$i$ provides strict inverses for $m$, i.e.
$m(\gamma,i(\gamma))=\id_1=m(i(\gamma),\gamma)$ for all $\gamma\in \mathrm{Mor}(\Gamma)$.

\item 
$\lambda$ is coherent over four copies of $\Gamma$,  all components are endomorphisms (i.e. $s \circ \lambda = t \circ \lambda$), and $\lambda(1,1,1)=\id_1$.

\end{enumerate} 
\end{definition}

The manifold $ \obj{\Gamma}$ is a Lie group with multiplication $m$, unit $1$, and inversion $i$; likewise, the set  $\pi_0\Gamma$ of isomorphism classes of objects is a group. The set $\pi_1\Gamma$ of automorphisms of $1\in  \obj{\Gamma}$ forms a group under composition and multiplication; these two group structures commute and are hence abelian.
Semi-strict Lie 2-groups are (coherent) Lie 2-groups in the sense of \cite{baez5} and have been considered in \cite{Jurco2015}. A semi-strict Lie 2-group is called \emph{strict Lie 2-group} if the associator is trivial, i.e. $\lambda(g_1,g_2,g_3)=\id_{g_1g_2g_3}$.

\begin{definition}
\label{def:sshom}
Let $\Gamma$ and $\Gamma'$ be semi-strict Lie 2-groups. A \emph{semi-strict homomorphism} between $\Gamma$ and $\Gamma'$ is a smooth functor $F:\Gamma \to \Gamma'$ together with a natural transformation (\quot{multiplicator}) 
\begin{equation*}
\alxydim{@=3em}{\Gamma \times \Gamma \ar[r]^-{m} \ar[d]_{F \times F} & \Gamma \ar@{=>}[dl]|*+{\chi} \ar[d]^{F} \\ \Gamma' \times \Gamma' \ar[r]_-{m'} & \Gamma' }
\end{equation*}
satisfying the following conditions:
\begin{enumerate}[(a)]

\item
its components are endomorphisms, i.e. $s(\chi(g_1,g_2))=t(\chi(g_1,g_2))$ for all $g_1,g_2\in  \obj{\Gamma}$.  

\item 
it respects the units: $\chi(1,1)=\id_{1}$.

\item
it is compatible with the associators $\lambda$ and $\lambda'$ in the sense that for all $g_1,g_2,g_3\in  \obj{\Gamma}$ we have 
\begin{multline*}
 \lambda'(F(g_{3}),F(g_{2}),F(g_{1}))\circ( \chi(g_{3},g_{2}) \cdot \id_{F(g_{1})})\circ\chi(g_{3}g_{2},g_{1})
\\
=(\id_{F(g_{3})} \cdot \chi(g_{2},g_{1}))\circ \chi(g_{3},g_{2}g_{1})\circ F(\lambda(g_{3},g_{2},g_{1}))
\end{multline*}

\end{enumerate}
\end{definition}

We remark that a semi-strict homomorphism $F$ induces a group homomorphism on the level of objects.
Semi-strict homomorphisms have an associative composition given by the composition of functors and the \quot{stacking} of the multiplicators. A semi-strict homomorphism is called \emph{strict homomorphism} if the multiplicator $\chi$ is trivial.

We consider crossed modules and crossed intertwiners as special cases of semi-strict Lie 2-groups and semi-strict homomorphisms. A crossed module $\Gamma=(G,H,t,\alpha)$ defines the Lie groupoid (we denote it by the same letter) $\Gamma$, with $ \obj{\Gamma}:=G$ and $\mor{\Gamma}=H \ltimes_{\alpha} G$ with source $(h,g)\mapsto g$ and target $(h,g)\mapsto t(h)g$, and the composition is induced from the group structure of $H$. The functors $m:\Gamma \times \Gamma \to \Gamma$ and $i:\Gamma \to \Gamma$ are defined using the Lie group structures on $G$ and $H$. The associator is trivial; thereby, $\Gamma$ is a strict Lie 2-group. It is well-known that every strict Lie 2-group is of this form.

Next we consider a crossed intertwiner $F=(\phi,f,\eta):\Gamma \to \Gamma'$ between crossed modules $\Gamma=(G,H,t,\alpha)$ and $\Gamma'=(G',H',t',\alpha')$, and identify $\Gamma$ and $\Gamma'$ with their associated strict Lie 2-groups.  We define a corresponding semi-strict homomorphism (denoted by the same letter), based on the smooth functor with the following assignments to objects $g\in  \obj{\Gamma}$ and morphisms $(h,g) \in \mor{\Gamma}$:
 
\begin{equation*}
F(g):=\phi(g)
\quand
F(h,g):= (\eta(t(h),g)^{-1}\cdot f(h),\phi(g))\text{.}
\end{equation*}
The smooth map $\eta$ defines a multiplicator
$\chi$ for $F$
with component map 
\begin{equation*}
\chi(g,g') := (\eta(g,g'),\phi(gg'))\text{.}
\end{equation*}
It is straightforward though again tedious to show that it satisfies all conditions of \cref{def:sshom}.
Thus, $(F,\chi)$ is a semi-strict homomorphism.
It is obvious that strict intertwiners induce strict homomorphisms.
Further, the smooth functor associated to a composition of crossed intertwiners is the composition of the separate functors,
and the \quot{stacking} of the corresponding multiplicators
 $\chi_1$ and $\chi_2$ is precisely the multiplicator of the composition.
Summarizing, we have defined a functor from the category of crossed modules and crossed intertwiners to the category of strict Lie 2-groups and semi-strict homomorphisms.

\subsection{Non-abelian cohomology for semi-strict Lie 2-groups}

\label{sec:semistrictcocycles}
\label{sec:strictint}
\label{sec:CI:induced}

To a semi-strict Lie 2-group $\Gamma$ we associate a presheaf of bicategories $\sheaf {B \Gamma}:=C^{\infty}(-,B\Gamma)$ of \emph{smooth $B\Gamma$-valued functions}. Explicitly, the bicategory $\sheaf {B \Gamma}(X)$ associated to a smooth manifold $X$ is the following:
\begin{itemize}

\item 
It has just one object.

\item
The 1-morphisms are all smooth maps $g:X \to  \obj{\Gamma}$; composition is the pointwise multiplication.

\item
The 2-morphisms between $g_1$ and $g_2$ are all smooth maps $h:X \to \mor{\Gamma}$ such that $s \circ h=g_1$ and $t \circ h=g_2$.  
Vertical composition is the pointwise composition in $\Gamma$, and horizontal composition is pointwise multiplication.

\end{itemize}

\begin{definition}
\label{def:classifying}
Let $\mathcal{F}$ be a 2-stack over smooth manifolds, and $\Gamma$ be a semi-strict Lie 2-group.
We say that $\mathcal{F}$ is \emph{represented} by  $\Gamma$, if there exists an isomorphism of 2-stacks
$\sheaf {B\Gamma}^{+} \cong \mathcal{F}$.
\end{definition}

Here, $\sheaf {B\Gamma}^{+}$ denotes the 2-stackification, which can be performed e.g. with a construction described in \cite{nikolaus2}.
The objects of $\sheaf {B\Gamma}^{+}(X)$ are called \emph{$\Gamma$-cocycles}, and the \emph{non-abelian cohomology} of $X$ with values in smooth $\Gamma$-valued functions is by definition the set of equivalence classes of $\Gamma$-cocycles, i.e., $\h^1(X,\Gamma) :=\hc 0 (\sheaf {B \Gamma}^{+}(X))$. 

\begin{remark}
\label{re:classcoho}
For strict Lie 2-groups $\Gamma$, there  is a  classifying space for the 0-truncation of a 2-stack $\mathcal{F}$ represented by a semi-strict Lie 2-group $\Gamma$: one can use a certain geometric realization $|\Gamma|$ such that 
\begin{equation*}
\hc0{\mathcal{F}(X)}\cong \hc 0 {\sheaf{B\Gamma}^{+}(X) \cong} \h^1(X,\Gamma) \cong [X,|\Gamma|]\text{,}
\end{equation*}
where the last bijection was shown in \cite{baez8}. More precisely, in \cite[Theorem 1]{baez8} it was shown for well-pointed strict topological 2-groups that $|\Gamma|$ represents the \emph{continuous} non-abelian cohomology, and in \cite[Prop. 4.1]{Nikolaus} we have proved for Lie 2-groups (which are automatically well-pointed) that continuous and smooth non-abelian cohomologies coincide.  
\end{remark}

\begin{remark}
\label{re:ordinarycohomology}
\begin{enumerate}[(a)]

\item 
For a Lie group $G$ and the strict Lie 2-group $\Gamma=\idmorph{G}$ we have $\h^1(X,\idmorph{G})=\check\h^1(X,G)$, the \v Cech cohomology with values in the sheaf of smooth $G$-valued functions. In particular, if $G$ is discrete, this is the ordinary cohomology $\h^1(X,G)$.  
\item
For an abelian Lie group $A$ and strict Lie 2-group  $\Gamma=BA$ we have $\h^1(X,BA)=\check \h^2(X,A)$. In particular, $\h^1(X,B\ueins)=\check \h^2(X,\ueins)$, which is isomorphic to $\h^3(X,\Z)$. 

\end{enumerate}
\end{remark}

The $\Gamma$-cocycles for semi-strict Lie 2-groups have been worked out in \cite{Jurco2015}. With respect to  an open cover $\{U_i\}_{i\in I}$, they are pairs $(g,\gamma)$ consisting of  smooth maps
\begin{align*}
g_{ij} :U_i \cap U_j \to \obj\Gamma
\quand
\gamma_{ijk}:U_i \cap U_j \cap U_k \to \mor\Gamma
\end{align*}
such that the following conditions are satisfied:
\begin{enumerate}[(1)]

\item 
\label{coc:ses:1}
$g_{ii}=1$, $\gamma_{iij}=\gamma_{ijj}=\id_{g_{ij}}$.

\item 
\label{coc:ses:2}
$s(\gamma_{ijk})=g_{jk}\cdot g_{ij}$ and $t(\gamma_{ijk})=g_{ik}$.

\item
\label{coc:ses:3}
$\gamma_{ikl}\circ (\id_{g_{kl}} \cdot \gamma_{ijk})\circ \lambda(g_{kl},g_{jk},g_{ij})= \gamma_{ijl}\circ (\gamma_{jkl} \cdot \id_{g_{ij}})$.

\end{enumerate}
Here we have abbreviated the multiplication $m$ of $\Gamma$ by \quot{$\cdot$}.
Two $\Gamma$-cocycles $(g,\gamma)$ and $(g',\gamma')$ are equivalent, if there exist smooth maps $h_i :U_i \to \obj\Gamma$ and $\varepsilon_{ij}:U_i \cap U_j \to \mor\Gamma$ such that
\begin{enumerate}[(1)]

\setcounter{enumi}{3}

\item 
\label{coc:ses:4}
$\varepsilon_{ii}=\id_{h_i}$.

\item 
\label{coc:ses:5}
$s(\varepsilon_{ij})=h_j\cdot g_{ij}$ and $t(\varepsilon_{ij})=g_{ij}'\cdot h_i$.

\item 
\label{coc:ses:6}
$(\gamma'_{ijk}\cdot \id_{h_i}) \circ \lambda(g_{jk}',g_{ij}',h_i)^{-1} \circ(\id_{g'_{jk}} \cdot \varepsilon_{ij}) \circ \lambda(g_{jk}',h_j,g_{ij}) \circ (\varepsilon_{jk}\cdot \id_{g_{ij}})$

\nopagebreak
\hfill $= \varepsilon_{ik} \circ (\id_{h_k} \cdot \gamma_{ijk}) \circ \lambda(h_k,g_{jk},g_{ij})$.

\end{enumerate}
A semi-strict homomorphism $F:\Gamma \to \Gamma'$ induces a map 
\begin{equation*}
F_{*}:\h^1(X,\Gamma) \to \h^1(X,\Gamma')
\end{equation*}
in non-abelian cohomology. It can be described on the level of $\Gamma$-cocycles in the following way. If $(g,\gamma)$ is a $\Gamma$-cocycle with respect to an open cover $\{U_i\}_{i\in I}$ then the corresponding $\Gamma'$-cocycle $(g',\gamma') := F_{*}(g,\gamma)$ is given by  $\gamma'_{ijk}:=F(\gamma_{ijk})\circ \chi(g_{jk},g_{ij})^{-1}$ and $g'_{ij} :=F(g_{ij})$.

If the Lie 2-group $\Gamma$ is strict, one can reduce and reformulate $\Gamma$-cocycles in terms of the corresponding crossed module $(G,H,t,\alpha)$, resulting in the usual cocycles for non-abelian cohomology. 
Concerning a cocycle $(g,\gamma)$, we keep the functions $g_{ij}$ as they are, and write $\gamma_{ijk}=(a_{ijk},g_{jk}g_{ij})$ under the decomposition $\mor\Gamma=H\ltimes G$, for smooth maps 
$a_{ijk}:U_i \cap U_j \cap U_k \to H$
satisfying conditions equivalent to \cref{coc:ses:1,coc:ses:2,coc:ses:3}:  $a_{iij}=a_{ijj}=1$, 
\begin{equation*}
t(a_{ijk})g_{jk}g_{ij}=g_{ik}
\quand
a_{ikl}\cdot \alpha(g_{kl},a_{ijk})= a_{ijl}\cdot a_{jkl}\text{.}
\end{equation*}
Similarly, for an equivalence between cocycles, we write $\varepsilon_{ij}=(e_{ij},h_jg_{ij})$ for  smooth maps $e_{ij}:U_i \cap U_j \to H$ satisfying conditions equivalent to \cref{coc:ses:4,coc:ses:5,coc:ses:6}: $e_{ii}=1$ and 
\begin{equation}
\label{eq:equivcoccm}
t(e_{ij})h_jg_{ij}=g_{ij}'h_i
\quand
a_{ijk}'\alpha(g'_{jk},e_{ij})e_{jk} 
= e_{ik}\alpha(h_k,a_{ijk})\text{.}
\end{equation}
Now let $F:\Gamma \to \Gamma'$ be a crossed intertwiner between crossed modules $\Gamma$ and $\Gamma'$, with $F=(\phi,f,\eta)$. Passing to the associated semi-strict homomorphism, using its induced map in cohomology, and reformulating in terms of crossed modules, we obtain the following.
If $(g,a)$ is a $\Gamma$-cocycle, then the corresponding $\Gamma'$-cocycle $(g',a') := F_{*}(g,a)$ is given by $g_{ij}'=\phi(g_{ij})$ and
$a_{ijk}' :=\eta(t(a_{ijk}),g_{jk}g_{ij})^{-1}\cdot f(a_{ijk})\cdot \eta(g_{jk},g_{ij})^{-1}$.

\subsection{Semi-direct products}

\label{sec:semidirect}

Let $U$ be a (discrete) group, and let $\Gamma$ be a crossed module. 

\begin{definition}
\label{def:action}
An \emph{action} of $U$ on $\Gamma$ by crossed intertwiners is a group homomorphism \begin{equation*}
\varphi:U \to \mathrm{Aut}_{CI}(\Gamma)\text{.}
\end{equation*}
\end{definition}

The crossed intertwiner $\varphi(u)$ associated to $u\in U$ will be denoted by $F_u=(\phi_u,f_u,\eta_u)$, and as before we may consider $\Gamma$ as a strict Lie 2-group, $F_u$ as a smooth functor $F_u:\Gamma \to \Gamma$ with some multiplicator $\chi_u$. 

Given an action of $U$ on $\Gamma$ by crossed intertwiners, we  define a semi-strict Lie 2-group $\Gamma \ltimes_{\varphi} U$, called the \emph{semi-direct product} of $\Gamma$ with $U$.
Its underlying Lie groupoid is $\Gamma \times \idmorph{U}$. We equip it with a multiplication functor
$m: (\Gamma \times \idmorph{U}) \times (\Gamma \times \idmorph{U}) \to \Gamma \times \idmorph{U}$ defined
on objects and morphisms by
\begin{equation*}
(g_2,u_2)\cdot (g_1,u_1) := (g_2F_{u_2}(g_1),u_2u_1)
\quand
(\gamma_2,u_2)\cdot (\gamma_1,u_1) := (\gamma_2 \cdot F_{u_2}(\gamma_1),u_2u_1)\text{.} \end{equation*}
It is straightforward to check that this is a functor.
The associator $\lambda$ for the multiplication $m$
is given by
\begin{equation*}
\lambda((g_3,u_3),(g_2,u_2),(g_1,u_1)) := (\id_{g_3} \cdot \chi_{u_3}(g_2,F_{u_2}(g_1))^{-1},u_3u_2u_1)\text{.}
\end{equation*}
The inversion functor $i:\Gamma \to \Gamma$ is defined by 
\begin{equation*}
i(g,u) := (F_{u^{-1}}(g^{-1}),u^{-1})
\quand
i(\gamma,u) := (F_{u^{-1}}(\gamma^{-1}),u^{-1})\text{.}
\end{equation*}
It is again straightforward, though again a bit tedious, to check that all conditions for semi-strict Lie 2-groups are satisfied. Its invariants are
\begin{equation*}
\pi_0(\Gamma\ltimes_{\varphi} U) = \pi_0(\Gamma) \ltimes U
\quand
\pi_1(\Gamma\ltimes_{\varphi} U)  = \pi_1(\Gamma)\text{,} 
\end{equation*} 
where the action of $\pi_0$ on $\pi_1$ is induced from the one of $\Gamma$.

\label{sec:semidirect:cocycles}

We investigate how the $(\Gamma\ltimes_{\varphi} U)$-cocycles look like, reducing the $\Gamma$-cocycles of \cref{sec:semistrictcocycles} to the present situation. 
For an open cover $\{U_i\}_{i\in I}$ a cocycle is a triple $(u,g,\gamma)$ consisting of smooth maps
\begin{align*}
u_{ij} : U_i \cap U_j \to U
\quomma
g_{ij} :U_i \cap U_j \to \obj\Gamma
\quand
\gamma_{ijk}:U_i \cap U_j \cap U_k \to \mor\Gamma
\end{align*}
such that the following conditions are satisfied:
\begin{enumerate}[(1)]

\item 
\label{coc:ss:1}
$g_{ii}=1$, $u_{ii}=1$, $\gamma_{iij}=\gamma_{ijj}=\id_{g_{ij}}$.

\item 
\label{coc:ss:2}
$u_{jk}\cdot u_{ij}=u_{ik}$ and $s(\gamma_{ijk})=g_{jk}\cdot \phi_{u_{jk}}(g_{ij})$ and $t(\gamma_{ijk})=g_{ik}$

\item
\label{coc:ss:3}
$\gamma_{ikl}\circ (\id_{g_{kl}} \cdot F_{u_{kl}}(\gamma_{ijk}))
= \gamma_{ijl}\circ (\gamma_{jkl} \cdot F_{u_{jl}}(\id_{g_{ij}}))\circ (\id_{g_{kl}} \cdot \chi_{u_{kl}}(g_{jk},F_{u_{jk}}(g_{ij})))$.

\end{enumerate}
Two $(\Gamma \ltimes_{\varphi} U)$-cocycles $(u,g,\gamma)$ and $(u',g',\gamma')$ are equivalent, if there exist smooth maps $h_i :U_i \to \mor{\Gamma}$, $v_i: U_i \to U$ and $\varepsilon_{ij}:U_i \cap U_j \to \mor{\Gamma}$ such that
\begin{enumerate}[(1)]

\setcounter{enumi}{3}

\item 
\label{coc:ss:4}
$\varepsilon_{ii}=\id_{h_i}$.

\item 
\label{coc:ss:5}
$v_j \cdot u_{ij}=u'_{ij}\cdot v_i$ and $s(\varepsilon_{ij})=h_j\cdot \phi_{v_j}(g_{ij})$ and $t(\varepsilon_{ij})=g_{ij}'\cdot \phi_{u_{ij}'}(h_i)$.

\item 
\label{coc:ss:6}
$(\gamma'_{ijk} \cdot F_{u'_{ik}}(\id_{h_i})) \circ (\id_{g_{jk}'} \cdot \chi_{u_{jk}'}(g_{ij}',F_{u_{ij}'}(h_i))) \circ(\id_{g'_{jk}} \cdot F_{u'_{jk}}(\varepsilon_{ij}))$

\nopagebreak
\hspace{6em}$\circ (\id_{g_{jk}'} \cdot \chi_{u_{jk}'}(h_j,F_{v_j}(g_{ij}))^{-1}) \circ (\varepsilon_{jk}\cdot F_{v_ku_{jk}}(\id_{g_{ij}}))$

\nopagebreak
\hfill$=\varepsilon_{ik} \circ (\id_{h_k} \cdot F_{v_k}(\gamma_{ijk})) \circ (\id_{h_k} \cdot \tilde\eta_{v_k}(g_{jk},F_{u_{jk}}(g_{ij}))^{-1})\text{.}$
\end{enumerate}

We formulate these results in terms of  the crossed module $\Gamma=(G,H,t,\alpha)$ and of the crossed intertwiner $F_u=(\phi_u,f_u,\eta_u)$ associated to $u\in U$.
For a cocycle $(u,g,\gamma)$, we start by writing $\gamma_{ijk}=(a_{ijk},g_{jk}\cdot \phi_{u_{jk}}(g_{ij}))$ for functions $a_{ijk}:U_i \cap U_j \cap U_k \to H$, and get the following conditions:
\begin{enumerate}[(1')]

\item 
\label{CSPprime:1}
$g_{ii}=1$, $u_{ii}=1$, $a_{iij}=a_{ijj}=1$.

\item 
\label{CSPprime:2}
$u_{jk}\cdot u_{ij}=u_{ik}$ and $t(a_{ijk})\cdot g_{jk}\cdot \phi_{u_{jk}}(g_{ij})=g_{ik}$

\item
\label{CSPprime:3}
$a_{ikl}\cdot \alpha(g_{kl},\eta_{u_{kl}}(g_{ik}\phi_{u_{jk}}(g_{ij})^{-1}g_{jk}^{-1},g_{jk} \phi_{u_{jk}}(g_{ij}))^{-1} f_{u_{kl}}(a_{ijk}))$

\hfill$
= a_{ijl}\cdot a_{jkl}\cdot \alpha(g_{kl},\eta_{u_{kl}}(g_{jk},\phi_{u_{jk}}(g_{ij})))$.
\end{enumerate}
For an equivalence $(\varepsilon,h,v)$ we write $\varepsilon_{ij}=(e_{ij},h_j\cdot \phi_{v_j}(g_{ij}))$ with $e_{ij}:U_i \cap U_j \to H$, and get
\begin{enumerate}[(1')]

\setcounter{enumi}{3}

\item 
\label{CSPprime:4}
$e_{ii}=1$.

\item 
\label{CSPprime:5}
$v_j \cdot u_{ij}=u'_{ij}\cdot v_i$ and $t(e_{ij}) \cdot h_j\cdot \phi_{v_j}(g_{ij})=g_{ij}'\cdot \phi_{u_{ij}'}(h_i)$.

\item 
\label{CSPprime:6}
$a_{ijk}'\cdot \alpha(g_{jk}',\eta_{u_{jk}'}(g_{ij}',\phi_{u_{ij}'}(h_i)))\cdot \alpha(g'_{jk},\eta_{u_{jk}'}(t(e_{ij}),h_j \phi_{v_j}(g_{ij})))^{-1}$

\quad$\cdot\; \alpha(g'_{jk},f_{u_{jk}'}(e_{ij}))\cdot \alpha( g_{jk}',\eta_{u_{jk}'}(h_j,\phi_{v_j}(g_{ij}))^{-1})\cdot e_{jk}$

\hfill
$=e_{ik} \cdot \alpha(h_k,\eta_{v_k}(t(a_{ijk}),g_{jk} \phi_{u_{jk}}(g_{ij})))^{-1}\cdot\alpha(h_k, f_{v_k}(a_{ijk})) \cdot \alpha(h_k,\eta_{v_k}(g_{jk},\phi_{u_{jk}}(g_{ij}))^{-1})$

\end{enumerate}

Finally, we notice that the semi-direct product fits into a sequence
\begin{equation}
\label{eq:sequencesemi}
\alxydim{}{\Gamma \ar[r]^-{i} & \Gamma \ltimes_{\varphi} U \ar[r]^-{p} & \idmorph{U}}
\end{equation}
of semi-strict Lie 2-groups and semi-strict homomorphisms defined in the obvious way. We want to investigate the induced sequence in non-abelian cohomology. For $u\in U$ the crossed intertwiner $F_u: \Gamma \to \Gamma$ induced a map $(F_{u})_{*}:\h^1(X,\Gamma) \to \h^1(X,\Gamma)$, forming an action of $U$ on the set $\h^1(X,\Gamma)$. 

\begin{lemma}
\label{lem:Uorbits}
Consider $x_1,x_2 \in \h^1(X,\Gamma)$. Then, $i_{*}(x_1)=i_{*}(x_2)$ if  there exists $u\in U$ with $(F_u)_{*} (x_1)=x_2$. If $X$ is connected, then \quot{only if} holds, too.
\end{lemma}

\begin{proof}
We show first the \quot{only if}-part under the assumption that $X$ is connected. Let $(g,a)$ and $(\tilde g,\tilde a)$ be $\Gamma$-cocycles with respect to some open cover of $X$, such that $i_{*}(g,a)=(1,g,a)$ and $i_{*}(\tilde g,\tilde a)=(1,\tilde g,\tilde a)$ are equivalent. Thus, after a possible refinement of the open cover, there exists equivalence data $(e,h,v)$ satisfying \cref{CSPprime:4,CSPprime:5,CSPprime:6}. Reduced to the case $u_{ij}=u_{ij}'=1$, these are
\begin{enumerate}[(1')]

\setcounter{enumi}{3}

\item 
$e_{ii}=1$.

\item 
$v_j= v_i$ and $t(e_{ij}) \cdot h_j\cdot \phi_{v_j}(g_{ij})=\tilde g_{ij}\cdot h_i$.

\item 
$\tilde a_{ijk}\cdot\; \alpha(\tilde g_{jk},e_{ij})\cdot e_{jk}=e_{ik} \cdot \alpha(h_k,\eta_{v_k}(t(a_{ijk}),g_{jk} g_{ij})^{-1}\cdot f_{v_k}(a_{ijk}) \cdot \eta_{v_k}(g_{jk},g_{ij})^{-1})$

\end{enumerate}
The first part of \cref{CSPprime:5}, together with the fact that $X$ is connected, shows that there exists $u\in U$ with $u=v_i$ for all $i\in I$. We have $(F_u)_{*}(g,a)=(g',a')$ with 
\begin{equation}
\label{eq:lem:Uorbit:1}
g_{ij}':=\phi_u(g_{ij})
\quand
a_{ijk}' :=\eta_u(t(a_{ijk}),g_{jk}g_{ij})^{-1}\cdot f_u(a_{ijk})\cdot \eta_u(g_{jk},g_{ij})^{-1}\text{.}
\end{equation}   
Under these definitions, the second part of \cref{CSPprime:5} and \cref{CSPprime:6} become exactly \cref{eq:equivcoccm}, showing that $(g,a)$ and $(g'',a'')$ are equivalent.

In order to show the \quot{if}-part, let $(g,a)$ be a $\Gamma$-cocycle with respect to some open cover of $X$. Let $u\in U$ with associated crossed intertwiner $F_u=(\phi_u,f_u,\eta_u)$. We have $(F_u)_{*}(g,a)=(g',a')$ with $g_{ij}'$ and $a'_{ijk}$ defined exactly as in \cref{eq:lem:Uorbit:1}, and  $i_{*}(g',a')=(1,g',a')$. We have to show that the $(\Gamma \ltimes_{\varphi}U)$-cocycles $(1,g,a)$ and $(1,g',a')$ are equivalent. Indeed, we employ equivalence data $(e,h,v)$ with $e_{ij}:=1$, $h_{i}:=1$ and $v_i:= u$.
Then, \cref{CSPprime:4} is trivial, \cref{CSPprime:5} is satisfied since $v_i=u=v_j$ and $\phi_u(g_{ij})=g_{ij}'$
and \cref{CSPprime:6} is precisely above definition of $a_{ijk}'$. 
\end{proof}

The \quot{if}-part of \cref{lem:Uorbits}
shows that the map $i_{*}:\h^1(X,\Gamma) \to \h^1(X,\Gamma\ltimes_{\varphi} U)$ is constant on the $U$-orbits, i.e. it descents into the (set-theoretic) quotient $\h^1(X,\Gamma)/U$. Now we are in position to formulate all properties of the sequence induced by \cref{eq:sequencesemi} in non-abelian cohomology.       

\begin{proposition}
\label{prop:fibresequence}
Suppose a group $U$ acts on a crossed module by crossed intertwiners. Then, the sequence
\begin{equation*}
\alxydim{}{\h^1(X,\Gamma)/U \ar[r]^-{i_{*}} & \h^1(X,\Gamma \ltimes_{\varphi} U) \ar[r]^-{p_{*}} & \check\h^1(X,U) \ar[r] & 0}
\end{equation*}
is an exact sequence of pointed sets, for all smooth manifolds $X$. If $X$ is connected, then the first map is injective.
\end{proposition}

\begin{proof}
First of all, $p_{*}$ is surjective, since a given \v Cech 2-cocycle $u_{ij}:U_i \cap U_j \to U$ can be lifted to a $(\Gamma \ltimes_{\varphi} U)$-cocycle $(u,g,a)$ by putting $g_{ij}=1$ and $a_{ijk}=\id$. Second, the  composition $p \circ i : \Gamma \to \idmorph{U}$ is the trivial functor. In order to see exactness at $\h^1(X,\Gamma \ltimes_{\varphi} U)$, let $(u,g,a)$ be a $(\Gamma \ltimes_{\varphi} U)$-cocycle with respect to an open cover $\{U_i\}_{i\in I}$ of $X$. If we assume that $p_{*}(u,g,a)=u=0$, then there exist $v_i:U_i \to U$ such that $ u_{ij}=v_j^{-1}\cdot v_i$. We define data $(1,1,v_i)$ for an equivalence between $(u,g,a)$ and another $(\Gamma \ltimes_{\varphi} U)$-cocycle $(u_{ij}',g_{ij}',a_{ijk}')$, which we define such that \cref{CSPprime:4,CSPprime:5,CSPprime:6} are satisfied.  We get $u'_{ij}=1$. By inspection, we observe that any  $(\Gamma \ltimes_{\varphi} U)$-cocycle $(u',g',a')$ with $u'=1$ yields a $\Gamma$-cocycle $(g',a')$ such that $i_{*}(g',a')=(1,g',a')$.
Finally, the injectivity of $i_{*}$ is the \quot{only if}-part of \cref{lem:Uorbits}.
\end{proof}

\subsection{Equivariant crossed intertwiners}

\label{sec:equivCI}

We suppose that $F:\Gamma \to \Gamma'$ is a crossed intertwiner between two crossed modules $\Gamma$ and $\Gamma'$, on which a discrete group $U$ acts by crossed intertwiners. The crossed intertwiners of the actions are denoted by $F_u:\Gamma \to \Gamma$ and $F_u':\Gamma'\to \Gamma'$, respectively.

\begin{definition}
\label{def:equivint}
The crossed intertwiner $F$ is called \emph{strictly $U$-equivariant} if \begin{equation*}
F_u' \circ F=F \circ F_u
\end{equation*}
for all $u\in U$.
\end{definition}

\begin{remark}
\label{re:equivintcond}
We write $F=(\phi,f,\eta)$ as well as $F_u=(\phi_u,f_u,\eta_u)$ and $F_u'=(\phi_u',f_u',\eta_u')$. Then, the commutativity of the diagram splits  into three conditions, namely
\begin{equation}
\label{eq:CIequiv1}
\phi_u' \circ \phi=\phi \circ \phi_u
\quand
f_u' \circ f = f \circ f_u
\end{equation}
as well as
\begin{equation}
\label{eq:CIequiv2}
\eta_u'(\phi(g_1),\phi(g_2))\cdot f_u'(\eta(g_1,g_2)) = \eta(\phi_u(g_1),\phi_u(g_2))\cdot f(\eta_u(g_1,g_2))
\end{equation}
for all $g_1,g_2\in G$. 
\end{remark}

A strictly $U$-equivariant crossed intertwiner $F$ induces a  semi-strict homomorphism
\begin{equation*}
F_U:\Gamma \ltimes_{\varphi} U \to \Gamma' \ltimes_{\varphi} U
\end{equation*}
between the semi-direct products. Indeed, as a functor it is defined by $F_U:= F \times \id_{\idmorph{U}}$, and its multiplicator 
is defined by
\begin{equation*}
\chi((g_2,u_2),(g_1,u_1)) := (\chi(g_2,F_{u_2}(g_1)),u_2u_1)\text{,}
\end{equation*}
where $\chi$ is the multiplicator of  $F$.
Checking that all conditions for a multiplicator are satisfied is again tedious but straightforward.

We go one step further and  consider the induced map in cohomology,
\begin{equation*}
(F_U)_{*}:\h^1(X,\Gamma\ltimes_{\varphi} U) \to \h^1(X,\Gamma'\ltimes_{\varphi} U)\text{.}
\end{equation*}
We describe this map at the level of $(\Gamma \ltimes U)$-cocycle $(u,g,a)$ in the formulation with the crossed module $\Gamma=(G,H,t,\alpha)$ as in \cref{sec:semidirect:cocycles}. Thus, $(u,g,a)$ consists of  smooth maps
\begin{align*}
u_{ij} : U_i \cap U_j \to U
\quomma
g_{ij} :U_i \cap U_j \to G
\quand
a_{ijk}:U_i \cap U_j \cap U_k \to H
\end{align*}
satisfying \cref{CSPprime:1,CSPprime:2,CSPprime:3}.
The map induced by a general semi-strict homomorphism  was described in \cref{sec:semistrict}; here it reduces to $g'_{ij}:=\phi(g_{ij})$, $u'_{ij}:=u_{ij}$, and
\begin{align*}
\alpha'_{ijk}&:=\eta(t(a_{ijk}),g_{jk} \phi_{u_{jk}}(g_{ij}))^{-1}\cdot f(a_{ijk})\cdot\eta(g_{jk},\phi_{u_{jk}}(g_{ij}))^{-1}\text{.}
\end{align*}

\setsecnumdepth{1}

\section{The Poincaré bundle}

\label{sec:poi}

In this section we recall and introduce required facts about the Poincaré bundle. 
We work with writing $\ueins=\R/\Z$ additively. 
Basically, the Poincaré bundle is the following principal $\ueins$-bundle $\poi$ over $\T^2=\ueins \times \ueins$. Its total space is
\begin{equation*}
\poi := (\R \times \R \times \ueins) \;/\; \sim
\end{equation*}
with $(a,\hat a,t)\sim (a+n,\hat a+m,n\hat a+t)$ for all $n,m\in \Z$ and $t\in \ueins$. The bundle projection is $(a,\hat a,t) \mapsto (a,\hat a)$, and the $\ueins$-action is $(a,\hat a,t)\cdot s := (a,\hat a,t+s)$.

\begin{remark}
\label{re:poiconn}
The Poincaré bundle $\poi$ carries a canonical connection, which descends from the 1-form $\tilde \omega\in \Omega^1(\R \times \R \times \ueins)$ defined by $\tilde\omega := a \mathrm{d}\hat a-\mathrm{d}t$. 
The curvature of $\omega$ is $\pr_1^{*}\theta \wedge \pr_2^{*}\theta \in \Omega^2(\T^{2})$, where $\theta\in\Omega^1(\ueins)$ is the Maurer-Cartan form. Since $\h^{*}(\T^{2},\Z)$ is torsion free, this shows  that the first Chern class of $\poi$ is $\pr_1 \cup \pr_2 \in \h^2(\T^2,\Z)$, where $\pr_i: \T^2 \to S^1$ is the projection, whose homotopy class is an element of $[\T^2,S^1]=\h^1(\T^2,\Z)$. 
\end{remark}

The Poincaré bundle has quite difficult (non)-equivariance effects, which we shall explore in the following.
We let $\R^2$ act on $\T^2$ by addition, and let $r_{x,\hat x}:\T^2 \to \T^2$ denote the action of $(x,\hat x)\in \R^2$. It lifts to $\poi$ in terms of a bundle isomorphism
\begin{equation*}
R_{x,\hat x}:\poi \to \poi
\end{equation*}
defined by $R_{x,\hat x}(a,\hat a,t) := (x+a,\hat x+\hat a,t+a\hat x)$.
Equivalently, we can regard $R_{x,\hat x}$ as a bundle isomorphism $\tilde R_{x,\hat x}: \poi \to r_{x,\hat x}^{*}\poi$ over the identity on $\T^2$.
The lifts $R_{x,\hat x}$ do \emph{not} define an action of $\R^2$ on $\poi$. Indeed, we find
\begin{equation*}
R_{x',\hat x'}(R_{x,\hat x}(a,\hat a,t)) 
= R_{x'+x,\hat x'+\hat x}(a,\hat a,t)\cdot x\hat x'\text{.}
\end{equation*}
In order to treat this \quot{error}, we define $\nu: \R^2 \times \R^2 \to \ueins$ by $\nu((x',\hat x'),(x,\hat x)) := \hat x'x$,
so that 
\begin{equation*}
R_{x',\hat x'}\circ R_{x,\hat x}=R_{x'+x,\hat x'+\hat x} \cdot \nu((x',\hat x'),(x,\hat x))\text{.} 
\end{equation*}
Equivalently, we have
\begin{equation*}
r_{x,\hat x}^{*}\tilde R_{x',\hat x'}\circ \tilde R_{x,\hat x} = \tilde R_{x'+x,\hat x'+\hat x}\cdot \nu((x',\hat x'),(x,\hat x))\text{.}
\end{equation*}
Next, we restrict  to $\Z^{2} \subset \R^2$. For $(m,\hat m)\in \Z^2$  the bundle morphism $R_{m,\hat m}$ covers the identity on $\T^2$, and it is given by multiplication with the smooth map \begin{equation*}
f_{m,\hat m}:\T^2 \to \ueins:(a,\hat a) \mapsto a\hat m-m\hat a\text{.}
\end{equation*}
Since the restriction of $\nu$ to $\Z^2 \times \Z^2$ vanishes, this is a genuine action of $\Z^2$ on $\poi$.

\label{sec:npoi}

Next we generalize the Poincaré bundle to $n$-fold tori. Let $B\in \mathfrak{so}(n,\Z)$ be a skew-symmetric matrix. We define the principal $\ueins$-bundle
\begin{equation*}
\poi_B := \bigotimes_{1\leq j<i\leq n} \pr_{ij}^{*}\poi^{\otimes B_{ij}}
\end{equation*}
over $\T^{n}$, where $\pr_{ij}:\T^{n} \to \T^2$ denotes the projection to the two indexed factors. Braiding of tensor factors gives a canonical isomorphism $\poi_{B_1+B_2}\cong \poi_{B_1}\otimes \poi_{B_2}$, and we have $\poi_{0}=\T^{n} \times \ueins$. Hence,  assigning to $B$ the first Chern class of $\poi_B$ is group homomorphism
\begin{equation*}
\mathfrak{so}(n,\Z) \to \h^2(\T^{n},\Z):B \mapsto \mathrm{c_1}(\poi_B)\text{.}
\end{equation*}
Via the Künneth formula it is easy to see that it is an isomorphism.

For $a=(a_1,...,a_n)\in \R^{n}$ we define a map $R_{B}(a): 
\poi_B \to \poi_B$ tensor-factor-wise as $R_{a_i,a_j}:\poi \to \poi$; this is a $\ueins$-equivariant smooth map that covers the action of $\R^{n}$ on $\T^{n}$ by addition. We have from the definitions 
\begin{equation}
\label{eq:equivpoiB}
R_{B}(a')\circ R_{B}(a)=R_B(a'+a)\cdot \lbraket a{B}{a'}\text{.} 
\end{equation}
If we denote by $\tilde R_B(a):\poi_B \to r_a^{*}\poi_B$ the corresponding bundle morphism over the identity of $\T^{n}$, then we can rewrite \cref{eq:equivpoiB} as 
\begin{equation}
\label{eq:equivpoiBtilde}
r_a^{*}\tilde R_{B}(a')\circ \tilde R_{B}(a) =\tilde R_{B}(a'+a)\cdot \lbraket {a}B{a'}\text{.}
\end{equation} 
Concerning the restriction to integers $m\in \Z^{n}$, we note that $R_B(m)$ acts factor-wise  as $R_{m_i,m_j}$, i.e. by multiplication with the smooth map $f_{m_i,m_j}: \T^2 \to \ueins$. Thus, $R_B(m)$ is multiplication with the map
\begin{equation}
\label{eq:equivpoiBZ}
f_m: \T^{n} \to \ueins:a \mapsto \braket aBm
\end{equation}
In particular, we obtain from \cref{eq:equivpoiBtilde} and \cref{eq:equivpoiBZ} 
\begin{equation}
\label{eq:equivpoiBZcomb}
\tilde R_{B}(a+m)=\tilde R_{B}(a) \cdot  (\ket{B}{m}- \lbraket {m}B{a})\text{.}
\end{equation}

\begin{remark}
\label{re:transpoiB}
The Poincaré bundle $\poi$ has a smooth section
$\chi:\R^2 \to \poi: (a,\hat a) \mapsto (a,\hat a,0)$ along the projection $\R^{2}\to \T^{2}$, whose transition function is $((a+m,\hat a+\hat m),(a,\hat a))\mapsto m\hat a$. 
The pullback of $\poi_B$ along $\R^{n} \to \T^{n}$ has an induced section whose transition function is 
\begin{equation*}
\R^{n} \times_{\T^{n}} \R^{n} \to \ueins: (a+m,a)\mapsto \lbraket m{B}{a}.
\end{equation*}
This clarifies how $\poi_B$ can be obtained from a local gluing process.
\end{remark}

\label{ex:nfold}

Finally, we  reduce the previous consideration to the matrix 
\begin{equation*}
B_n:=\begin{pmatrix}0 & -E_n \\
E_n & 0 \\
\end{pmatrix} \in \mathfrak{so}(2n,\Z)\text{.}
\end{equation*}
The corresponding principal $\ueins$-bundle over $\T^{2n}$ is called the $n$-fold Poincaré bundle and denoted by $\poi[n] := \poi_{B_n}$. Note that for $n=1$ we get $\poi[1]= \pr_{21}^{*}\poi\cong \poi^{\vee}$. 
We write $\tilde R(a,\hat a):\poi[n] \to r_{a,\hat a}^{*}\poi[n]$ for the bundle isomorphism $\tilde R_B(a,\hat a)$. We have from \cref{eq:equivpoiBtilde}
\begin{equation*}
r_{a,\hat a}^{*}\tilde R(a'\oplus \hat a')\circ \tilde R(a \oplus \hat a)=\tilde R((a'+a) \oplus (\hat a' \oplus \hat a)) \cdot \hat aa'\text{,}
\end{equation*} 
and \cref{eq:equivpoiBZcomb} becomes
\begin{align*}
\tilde R((a+n)\oplus (\hat a+\hat m)) 
= \tilde R(a\oplus \hat a) \cdot \eta_{m,\hat m,a} \text{,}
\end{align*}
with a smooth map $\eta_{m,\hat m,a}: \T^{2n} \to \ueins$ defined by
\begin{equation*}
\eta_{m,\hat m,a}(x,y) :=   \braket{x\oplus y}{B}{m \oplus \hat m}- \lbraket {m \oplus \hat m}B{a \oplus \hat a} = -x\hat m+my-a\hat m\text{.}
\end{equation*}

\begin{remark}
\label{eq:curvpoin}
The connection of \cref{re:poiconn} induces a connection on the $n$-fold Poincaré bundle $\poi[n]$ of curvature
\begin{equation*}
\sum_{i=1}^{n} \pr_{i+n}^{*}\theta \wedge \pr_{i}^{*}\theta\in \Omega^2(\T^{2n})\text{,}
\end{equation*}
and the transition function of \cref{re:transpoiB} reduces to
$(x+z,x)\mapsto [z,x]$
in the notation of \cref{re:tildeF}, where $x\in \R^{2n}$ and $z\in \Z^{2n}$.
\end{remark}

\end{appendix}

\setsecnumdepth{1}

\bibliographystyle{kobib}
\bibliography{kobib}

\end{document}